\documentclass[11pt]{article}

\usepackage[margin=1in]{geometry}
\usepackage{amsthm,amsmath,amsfonts,amssymb}
\usepackage[numbers]{natbib}
\usepackage{mathtools}
\usepackage{mathrsfs}
\usepackage{booktabs}
\usepackage{comment}
\usepackage{graphicx}
\usepackage{url}
\usepackage{verbatim}
\usepackage{xcolor}

\theoremstyle{plain}
\newtheorem{theorem}{Theorem}[section]
\newtheorem{lemma}[theorem]{Lemma}
\newtheorem{corollary}[theorem]{Corollary}
\newtheorem{proposition}[theorem]{Proposition}
\theoremstyle{definition}
\newtheorem{definition}[theorem]{Definition}
\newtheorem{assumption}{Assumption}
\newtheorem{example}{Example}
\newtheorem{remark}{Remark}

\newcommand{\Lip}{\operatorname{Lip}}
\newcommand\norm[1]{\left\lVert#1\right\rVert}
\newcommand\normt[1]{\left\lVert#1\right\rVert_{TV}}

\newcommand\normx[1]{\Vert#1\Vert}

\newcommand{\T}{\mathcal{T}}
\newcommand{\Q}{Q}
\newcommand{\Z}{{\cal P}(\mathbb{X})}
\newcommand{\X}{\mathbb{X}}
\newcommand{\Y}{\mathbb{Y}}
\newcommand{\U}{\mathbb{U}}
\newcommand{\sY}{\mathbb{Y}}
\newcommand{\sU}{\mathbb{U}}

\newcommand{\PP}{\mathcal{P}}
\newcommand{\bR}{\mathbb{R}}

\newcommand{\dWBL}{d^{\mathrm{BL}}_{W_1}}
\newcommand{\dWW}{d^{\mathrm{W}}_{W_1}}
\newcommand{\WBL}{W_1^{\mathrm{BL}}}
\newcommand{\WW}{W_1^{\mathrm{W}}}
\newcommand{\CD}{C_D}

\title{Sensitivity to Information Channel and Model Perturbations in Partially Observable Stochastic Control}
\author{Yunus Emre Demirci\thanks{Department of Mathematics, University of Michigan, Ann Arbor, MI, USA. \texttt{edemirci@umich.edu}}
\and Ali Devran Kara\thanks{Department of Mathematics, Florida State University, Tallahassee, FL, USA. \texttt{akara@fsu.edu}}
\and Serdar Y\"uksel\thanks{Department of Mathematics and Statistics, Queen's University, Kingston, ON, Canada. \texttt{yuksel@queensu.ca}}}
\date{}

\begin{document}

\maketitle

\begin{abstract}
Studying the stability of partially observed Markov decision processes (POMDPs) with respect to perturbations in the transition and observation kernels is a significant problem. While asymptotic robustness/stability results, as approximate transition kernels and/or measurement kernels converge to the true ones, have been previously reported, explicit bounds on value differences and mismatch costs have been limited in scope for POMDPs. The main technical contribution of this paper is a collection of explicit and uniform bounds on the distance between the filter kernels (that is, the kernels of the belief-MDP reduction) of two POMDPs, in the bounded-Lipschitz and Wasserstein metrics, in terms of the total variation and Wasserstein distances between their transition and measurement kernels; we also establish the sharpness of these bounds, demonstrate through counterexamples that the metrics employed cannot be weakened, and present verifiable conditions, stated directly in terms of the defining system functions, under which they apply. Building on these kernel proximity bounds, we show that control policies optimized for approximate models yield explicit performance guarantees when applied to the true model, under both discounted and average cost criteria. Two settings are treated separately: policy transfer, where the controller knows the true model and only the policy is designed for the approximate one, and model misspecification, where the controller also computes its beliefs with the approximate model, so that the model error enters through the belief updates as well; in both settings the prior may be incorrect. As a particular application, we consider the case where the state space and the measurement spaces are quantized to obtain finite models, and we obtain explicit error bounds which decay to zero as the approximations get finer. This provides explicit performance guarantees for model reduction in POMDPs.
\end{abstract}



\section{Introduction}

Partially observable Markov decision processes (POMDPs) present challenging mathematical problems with significant applied and practical relevance.
This paper studies perturbation, robustness and stability properties of non-linear filtering (also known as belief-MDPs) in the context of POMDPs, and also considers the control-free case, known as hidden Markov models (HMMs) or partially observable Markov processes (POMPs).

Robustness to model perturbations for fully observable stochastic control problems is a relatively well studied problem (see e.g. \cite{gordienko2008discounted, gordienko2009average,kara2020robustness,Lan81,hernandez2012adaptive}), however, there are comparatively far fewer studies in the partially observable context: In previous such studies, robustness for partially observable models has been studied under different perturbation settings: \cite{kleptsyna2016pert} studied robustness in a filter stability context, \cite{kara2020robustness} studied robustness properties (with positive results, and counterexamples in the absence of continuous convergence of kernels) when the state transition dynamics were perturbed while the measurement/observation kernels are fixed, and conversely, in \cite{WuVerdu,YukselOptimizationofChannels,hogeboom2021continuity}, continuity and robustness were examined in settings where only the measurement/observation kernels were perturbed (with an estimation and information theoretic angle). For finite state hidden Markov models in continuous time, \cite{chigansky2007model} showed that the filter with misspecified initial condition, transition rates and observation function converges to the optimal filter uniformly over the infinite time horizon as the misspecified model converges to the true one, and \cite[Section 4]{cohen2023exponential} recovered and refined these estimates, also giving computable error bounds for approximate filters. In these cases, robustness properties in perturbations were demonstrated under those in the total variation metric or metrics inducing weak convergence, with the Hilbert projective metric being considered in \cite{kleptsyna2016pert} and \cite{cohen2023exponential}. Recent monographs \cite{Kri25,YukselBasarBook24} provide comprehensive and complementary reviews of the recent literature.

In this paper, (i) we generalize and refine these results by also allowing for simultaneous perturbations in state transition and measurement kernels, and (ii) while doing so obtain quantitative bounds, thereby generalizing previous studies which often only considered continuity properties. We, in particular, quantify the proximity of non-linear filter kernels in terms of the proximities of the transition kernels and the measurement kernels under Wasserstein and bounded-Lipschitz norms, beyond asymptotic convergence. We then present explicit robustness implications for POMDPs given these regularity properties, for policies designed for an approximate model, both when the controller runs the filter of the true model (policy transfer) and when it runs the filter of the approximate model (misspecified model); the latter is the implementable case, in which the model error enters through the belief updates as well. The problem we present touches on research in a variety of directions, and unifies several results, as we summarize in the following.

\noindent{\bf Continuity and robustness in information structures and observation channels.} Perturbations in measurement channels and information structures lead to an angle which is quite subtle due to fragile dependence on information; to this end, \cite{YukselOptimizationofChannels,hogeboom2021continuity,WuVerdu} studied the problem of continuity of optimal costs when the channels are perturbed. \cite{YukselOptimizationofChannels,YukselBasarBook24} studied the impact of perturbations in measurement channels and information structures, showing that continuity of the optimal cost in observation channels holds under uniform convergence in total variation, but fails under weaker forms of convergence unless structural conditions are imposed, such as garbling \cite[Theorem 8.3.4]{YukselBasarBook24}\cite{WuVerdu} or additive noise perturbations. Notably, \cite{YukselOptimizationofChannels} established that optimization problems are not sequentially continuous under weak convergence of observation channels in general, but they are upper semi-continuous under convexity and continuity assumptions. Moreover, if channels are perturbed via a de-garbling sequence---where each channel is a garbling of the next---then sequential continuity can be recovered, as shown in \cite[Theorem 8.3.4]{YukselBasarBook24}\cite{hogeboom2021continuity}. A more recent contribution \cite{MarkovianContinuity2025} establishes continuity of MMSE estimators under weak convergence also by a de-garbling like condition.

\noindent{\bf Weak Feller and Wasserstein regularity of belief-MDP kernels.}
The regularity properties of the transition kernel of the belief-MDP play a fundamental role in the existence and approximation of optimal policies in POMDPs. Early results were provided by \cite{CrDo02}, \cite{FeKaZg14}, and \cite{KSYWeakFellerSysCont}, which established the weak Feller continuity of the belief-MDP transition kernel under different structural assumptions; see also \cite{feinberg2022markov,feinberg2023equivalent} for a unifying semi-uniform Feller framework.
Recently, \cite{demirci2023average,DKY_ACC2025} extended these by presenting sufficient conditions under which the transition kernel of the belief-MDP exhibits uniform Wasserstein continuity. Further studies explored related continuity properties: \cite{Timothy2025} examined POMDPs with initial-state dependent costs via belief-state augmentation, while \cite{TimothyNair2024} derived explicit Lipschitz continuity bounds for value functions in hypothesis testing problems formulated as POMDPs.

\noindent{\bf Continuity and robustness in POMDPs to incorrect transition kernels.} In the context of MDPs there have been several studies on robustness to model mismatch including \cite{gordienko2008discounted, gordienko2009average,Lan81,hernandez2012adaptive,kara2020robustness,KaraYuksel2021Chapter,kara2022robustness} as well as the closely related distributionally robust formulation (see e.g. \cite{blanchet2016,esfahani2015}). Since weak convergence is closely related to POMDPs (due to weak Feller continuity being applicable unlike stronger notions), \cite{kara2020robustness,KaraYuksel2021Chapter,kara2022robustness} established that under weak continuity assumptions on the transition kernels and continuity in total variation for observation channels, robustness holds; however, convergence does not generally hold under mere weak or setwise convergence unless these stronger conditions are imposed. Total variation leads to strong bounds as reported in \cite[Appendix A.2 and Theorem 2.5]{kara2020near}; also in this direction a recent contribution \cite[Theorem 15.11]{Kri25} obtains a robustness result involving the total variation distance of the kernels. These results provide convergence, but typically without a rate of convergence. For fully observed MDPs, \cite{zhou2024robustness} provides rates of convergence under Wasserstein regularity conditions; a primary motivation for the present paper is to make such quantitative results applicable to POMDPs by establishing the required regularity and proximity properties of the belief-MDP (filter) kernels.

\subsection{Partially Observed Markov Decision Processes}\label{sec:model}

Consider a stochastic process $ \{X_k\}_{k \in \mathbb{Z}_+} $ taking values in a Polish metric space $(\mathbb{X}, d)$, governed by the dynamics
\begin{align}\label{updateEq}
    X_{k+1} = F(X_k, U_k, W_k), \quad
    Y_k = G(X_k, V_k),
\end{align}
where $\{Y_k\}$ is a measurement sequence taking values in a standard Borel space $\mathbb{Y}$. We assume that the initial state $X_0$ admits a probability measure $\mu \in \mathcal{P}(\mathbb{X})$, and that $\{W_k\}$ and $\{V_k\}$ are mutually independent i.i.d. noise processes.

We denote by $\mathbb{B}(\mathbb{X})$ the Borel $\sigma$-field on $\mathbb{X}$, and by $\mathcal{P}(\mathbb{X})$ the space of probability measures on $(\mathbb{X}, \mathbb{B}(\mathbb{X}))$, equipped with the topology of weak convergence. Similarly, let $\mathcal{P}(\mathcal{P}(\mathbb{X}))$ denote the space of probability measures on $\mathcal{P}(\mathbb{X})$, also equipped with the weak convergence topology. The set of continuous and bounded functions on $\mathbb{X}$ is denoted by $\mathbb{C}(\mathbb{X})$. Throughout this paper, $\mathbb{N}$ represents the set of positive integers.

At each time step $k$, the decision maker selects a control action $U_k$, incurring a cost $c(X_k, U_k)$. The decision maker only has causal access to the measurement sequence $\{Y_k\}$ and past control actions $\{U_k\}$. Formally, an \emph{admissible policy} $\gamma$ is a sequence of control/decision functions $\{\gamma_k\}_{k \in \mathbb{Z}_+}$, where each $\gamma_k$ is measurable with respect to the $\sigma$-algebra generated by the information available at time $k$:
\[
I_k = \{Y_{[0,k]}, U_{[0,k-1]}\}, \quad \text{with} \quad I_0 = \{Y_0\},
\]
so that
\[
U_k = \gamma_k(I_k), \quad k \in \mathbb{Z}_+,
\]
where we use the notation $Y_{[0,k]}:=\{Y_0,\dots,Y_k\}$. We denote by $\Gamma$ the set of all such admissible policies. Implicitly, policies are also allowed to depend on the prior distribution $\mu$.

We assume that all of the random variables are defined on a common probability space $(\Omega, {\cal F}, P)$ given the initial distribution on the state, and a policy, on the infinite product space consistent with finite dimensional distributions, by the Ionescu Tulcea Theorem \cite{HernandezLermaMCP}. We will sometimes write the probability measure on this space as $P_\mu^\gamma$ to emphasize the policy $\gamma$ and the initialization $\mu$. We note that (\ref{updateEq}) can also, equivalently (via stochastic realization results \cite[Lemma~1.2]{gihman2012controlled}, \cite[Lemma~3.1]{BorkarRealization}, \cite[Lemma F]{aumann1961mixed}), be represented with transition kernels: the state transition kernel is denoted by $\mathcal{T}$ so that for Borel $B \subset \mathbb{X}$
\[{\mathcal{T}}(B|x,u) := P(X_1 \in B | X_0=x,U_0=u).\] We will denote the measurement kernel by $Q$ so that for Borel $B \subset \mathbb{Y}$:
\[Q(B|x) := P(Y_0 \in B | X_0=x).\]
While our analysis will primarily be stated at the level of the kernels $(\mathcal{T},Q)$, in Section \ref{sec:FGconditions} we will provide explicit sufficient conditions, stated directly in terms of the functions $F$ and $G$ and the noise distributions in (\ref{updateEq}), under which all of the assumptions imposed on $(\mathcal{T},Q)$, as well as the kernel distances appearing in our bounds, can be computed or estimated.

For (\ref{updateEq}), we are interested in minimizing either the average-cost optimization criterion
\begin{equation}\label{expCost}
    J_{\infty}(c,\mu,\mathcal{T}, Q,\gamma) := \limsup_{N \to \infty} \frac{1}{N} E_\mu^\gamma\left[\sum_{k=0}^{N-1} c(X_k,U_k)\right],
\end{equation}
with the optimal cost defined as
\[
J^*_{\infty}(c,\mu,\mathcal{T}, Q) := \inf_{\gamma \in \Gamma} J_{\infty}(c,\mu,\mathcal{T}, Q, \gamma),
\]
or the discounted cost criterion (for some $\beta \in (0,1)$)
\begin{equation}\label{expDiscCost}
    J_{\beta}(c,\mu,\mathcal{T}, Q,\gamma) := E_\mu^\gamma\left[\sum_{k=0}^{\infty} \beta^k c(X_k,U_k)\right],
\end{equation}
with the optimal discounted cost defined as
\[
J^*_{\beta}(c,\mu,\mathcal{T}, Q) := \inf_{\gamma \in \Gamma} J_{\beta}(c,\mu,\mathcal{T}, Q, \gamma)
\]
over all admissible control policies $\gamma = \{\gamma_0, \gamma_1, \cdots\} \in \Gamma$ with $X_0 \sim \mu$. With ${\cal P}(\mathbb{U})$ denoting the set of probability measures on $\mathbb{U}$ endowed with the weak convergence topology, we will also, when needed, allow for independent randomizations so that $\gamma_k(I_k)$ is ${\cal P}(\mathbb{U})$-valued for each realization of $I_k$. Here $c: \mathbb{X} \times \mathbb{U} \to \mathbb{R}_+$ is the stage-wise cost function. Since we study robustness to perturbations in the kernels, we explicitly include $\mathcal{T}$ and $Q$ as arguments of the cost functionals.

One may also consider the control-free case where the system equation (\ref{updateEq}) does not have control dependence; in this case only a decision is to be made at every time stage; $U$ is present only in the cost expression in (\ref{expCost}). This important special case has been studied extensively in the theory of non-linear filtering.

\subsection{Belief-MDP Reduction and the Filter Kernel}
It is well-known that any POMDP can be reduced to a (completely observable) MDP \cite{Yus76}, \cite{Rhe74}, whose states are the posterior state probabilities, or beliefs, of the observer; that is, the state at time $k$ is
\begin{align}
z_k(\,\cdot\,) := P\{X_{k} \in \,\cdot\, | Y_0,\ldots,Y_k, U_0, \ldots, U_{k-1}\} \in {\cal P}(\mathbb{X}). \nonumber
\end{align}
We call this equivalent MDP the belief-MDP. The belief-MDP has state space ${\cal P}(\mathbb{X})$ and action space $\U$. Here, ${\cal P}(\mathbb{X})$ is equipped with the Borel $\sigma$-algebra generated by the topology of weak convergence \cite{Bil99}. Since $\mathbb{X}$ is a Borel space, ${\cal P}(\mathbb{X})$ is metrizable with the Prokhorov metric, which makes ${\cal P}(\mathbb{X})$ into a Borel space \cite{Par67}. The transition probability $\eta$ of the belief-MDP can be constructed as follows. If we define the measurable function
\[F(z,u,y) := P\{X_{k+1} \in \,\cdot\, | z_k = z, U_k = u, Y_{k+1} = y\}\]
from ${\cal P}(\mathbb{X})\times\mathbb{U}\times\sY$ to ${\cal P}(\mathbb{X})$ and the stochastic kernel $H(\,\cdot\, | z,u) := P\{Y_{k+1} \in \,\cdot\, | z_k = z, U_k = u\}$ on $\sY$ given ${\cal P}(\mathbb{X}) \times \sU$, then the filter kernel $\eta$ can be written as
\begin{align}
\eta(\,\cdot\,|z,u) = \int_{\sY} 1_{\{F(z,u,y) \in \,\cdot\,\}} H(dy|z,u). \label{kernelFilter}
\end{align}
The one-stage cost function of the belief-MDP is given by
\begin{align}
\tilde{c}(z,u) := \int_{\mathbb{X}} c(x,u) z(dx). \label{weak:eq8}
\end{align}
When the cost function $c(x,u)$ is continuous and bounded on $\mathbb{X} \times \mathbb{U}$, an application of the generalized dominated convergence theorem \cite[Theorem 3.5]{Lan81}, \cite[Theorem 3.5]{serfozo1982convergence} shows that $\tilde{c}: {\cal P}(\mathbb{X}) \times \mathbb{U} \to \mathbb{R}$ is also continuous and bounded, and thus Borel measurable.

In particular, the belief-MDP is a (fully observed) Markov decision process with the components $({\cal P}(\mathbb{X}),\sU,\eta,\tilde{c})$. For finite horizon problems and a large class of infinite horizon discounted cost problems, it is then a standard result that an optimal control policy will use the belief $z_k$ as a sufficient statistic for optimal policies (see \cite{Yus76,Rhe74,Blackwell2}).

\subsection{Main Contributions}\label{sec:contributions}

The core results of the paper consist of \emph{proximity bounds for filter kernels}: explicit and uniform bounds on the distance between the belief-MDP (filter) kernels of two POMDPs in terms of the distances between their transition and measurement kernels. Once these bounds are available, robustness results for POMDPs follow as a corollary by viewing the POMDP as a fully observed MDP on the belief space and invoking (i) regularity results for the belief-MDP kernel and its value functions, and (ii) quantitative robustness results for fully observed MDPs. Our contributions are as follows:

\begin{itemize}
\item[(i)] \textbf{Proximity of filter kernels in terms of proximities of
transition and measurement kernels.} Suppose that a POMDP is defined by a state transition kernel ${\cal T}$ and an observation kernel $Q$, and let ${\cal T}_n,Q_n$ be their respective approximations. For the belief-MDPs induced by these models, we establish explicit and uniform error bounds on the resulting filter kernels in both bounded-Lipschitz and Wasserstein metrics. In Theorem~\ref{main1}, we show that the bounded-Lipschitz distance between the filter kernels is bounded by the sum of the total variation distances of the transition and observation kernels; this bound requires no assumptions on the models. In Theorem~\ref{main2}, this result is refined under the assumption that the observation kernels are Lipschitz continuous in total variation, leading to a bound in which the total variation distance between the transition kernels is replaced by their Wasserstein-1 distance. Theorems~\ref{W1main_1} and~\ref{W1main_2} extend these bounds to the Wasserstein-1 distance between the filter kernels for compact $\X$, with the belief space metrized either by $\rho_{BL}$ or by $W_1$; the latter is the version used in Section \ref{robustness}. We also show, via explicit examples, that the linear dependence of these bounds on the kernel distances cannot be improved in general (Examples~\ref{prop:sharpness} and~\ref{ex:channel_sharp}), and that the choice of metrics is necessary: the observation channel term cannot be relaxed from total variation to the bounded-Lipschitz or Wasserstein distances, even under channel regularity, and the Wasserstein refinement for the transition kernels fails without the channel regularity assumption (Examples~\ref{prop:TVnecessity} and~\ref{prop:LQnecessity}). The importance of the Wasserstein refinement is evident in the finite-model application of Section \ref{sec:quantization}: the quantized transition kernels do not converge to the true kernel in total variation (see Remark \ref{rem:whyW1}), so the total variation based bounds are not enough there, while Theorems \ref{main2} and \ref{W1main_2} yield convergent bounds. We further provide sufficient conditions for all of our assumptions, as well as computable bounds for the kernel distances themselves, directly in terms of the system functions $F$ and $G$ in (\ref{updateEq}) (Section~\ref{sec:FGconditions}).
Finally, since all of these bounds are uniform over the initial belief, they also yield bounds when, in addition, the initialization is incorrect (Corollary~\ref{cor:init_pert_W1} in Section~\ref{sec:initpert}).

\item[(ii)] \textbf{Computable bounds on continuity and robustness to model perturbation.}
Building on (i), we show that the optimal costs for approximate models converge to the optimal cost of the original system with explicit rates, and that policies optimized for approximate models perform nearly optimally when applied to the original system. Let $\gamma_n^*$ be an optimal policy for $(\mathcal{T}_n,Q_n)$. We treat two settings separately.
\begin{itemize}
\item[(a)] \emph{Policy transfer.} The controller knows the true model and the optimal policy of the approximate model, and applies this policy to the true model. Since it knows the true model, it updates its filter $\pi_t$ with the correct model, possibly from an incorrect prior $\nu_0$, and applies $U_t=\gamma^*_n(\pi_t)$ (Sections \ref{sec:disc_robust} and \ref{sec:avg_robust}).
\item[(b)] \emph{Misspecified model.} The controller does not know the true model and wrongly assumes that the model is the approximate one: it applies the optimal policy of the approximate model and also updates its filter with $(\T_n,Q_n)$, which gives the filter $\pi^n_t$ of that model, and applies $U_t=\gamma^*_n(\pi^n_t)$. This is the implementable case, for example with a data-driven model, in which only the approximate model is used, and the model error enters also through the belief updates (Section \ref{sec:misspecified}).
\end{itemize}
In both settings we provide explicit upper bounds for
\[
J_\beta(c,\mu,{\cal T},Q,\gamma_n^*) - J_\beta^*(c,\mu,{\cal T},Q), \quad 
J_\infty(c,\mu,{\cal T},Q,\gamma_n^*) - J_\infty^*(c,\mu,{\cal T},Q).
\]
These bounds are obtained by combining the proximity bounds in (i) with robustness results for fully observed MDPs (adapted from \cite{zhou2024robustness}) applied on the belief space, through a performance difference inequality that separates the model error from the belief error (Lemma \ref{lem:perfdiff}); for the required regularity conditions we build on \cite{demirci2023average,DKY_ACC2025}, together with sufficient conditions for the contraction conditions involved (Remark \ref{rem:K2suff} in Section \ref{sub:ex_K2}). This contrasts prior research where only asymptotic convergence was presented \cite{kara2020robustness}, or where only measurements were perturbed \cite{YukselOptimizationofChannels}, or where only total variation perturbations of the kernels were considered for finite space \cite{Kri25}. We also give explicit bounds on the difference between the optimal costs of the two models started at different priors, that is, on the error committed by a designer who optimizes for an incorrect model \emph{and} an incorrect prior (Corollary \ref{cor:prior_disc}). Under the contraction condition $K_2<1$, the average cost bounds are unaffected by the prior mismatch, since the optimal average cost is then constant over initial beliefs.

\item[(iii)] \textbf{Finite-POMDP approximation via joint quantization of state and measurement spaces.}
As a primary implication of our analysis, when the state and measurement spaces are uncountable, by simultaneously quantizing the state and observation spaces, we construct finite-state, finite-observation POMDP models that approximate the original system, and, as a corollary of (i) and (ii), we establish explicit and non-asymptotic performance bounds on the suboptimality of policies derived from such finite models. These results, given in Corollary~\ref{thm:joint_discounted_bound} (discounted cost) and, under an additional regularity condition on the finite model, Corollary~\ref{thm:joint_average_bound} (average cost), provide constructive guarantees for finite-POMDP-based control design, with error decaying linearly in the quantization resolutions.
\end{itemize}

\subsection{Convergence Notions on Measures and Kernels}\label{sec:notation}

A sequence of probability measures $\left\{\mu_{n}\right\}_{n\in\mathbb{N}}$ from ${\cal P}(\mathbb{X})$ converges weakly to $\mu \in {\cal P}(\mathbb{X})$ if for every bounded continuous function $f: \mathbb{X} \to \mathbb{R}$,
$$
\int_{\mathbb{X}} f(x) \mu_{n}(d x) \rightarrow \int_{\mathbb{X}} f(x) \mu(d x) \quad \text { as } \quad n \rightarrow \infty .
$$
When $\mathbb{X}$ is a separable, completely metrizable (i.e., Polish) space, the space $\mathcal{P}(\mathbb{X})$ equipped with the weak topology is itself a Polish space \cite[Chapter 2, Section 6]{Par67}. Consider the classes of test functions
\begin{align*}
&\operatorname{BL}(\mathbb{X}) := \left\{ f: \mathbb{X} \to \mathbb{R} \mid \norm{f}_\infty + \norm{f}_L \leq 1 \right\},\\
&\operatorname{T}(\mathbb{X}) := \left\{ f: \mathbb{X} \to \mathbb{R} \mid \norm{f}_\infty \leq 1 \right\},\\ 
& \operatorname{W}(\mathbb{X}) := \left\{ f: \mathbb{X} \to \mathbb{R} \mid \norm{f}_L \leq 1 \right\},
\end{align*}
where
\[
\norm{f}_\infty = \sup_{x \in \mathbb{X}} |f(x)|, \quad
\norm{f}_L = \sup_{x \neq y} \frac{|f(x) - f(y)|}{d(x, y)}.
\]
Given a class $F$ of measurable test functions, we define, for any $\mu, \nu \in {\cal P}(\X)$,
$$d_F(\mu, \nu) := \sup_{f \in F} \left( \int_{\mathbb{X}} f(x) \mu(dx) - \int_{\mathbb{X}} f(x) \nu(dx) \right).$$
This leads to:
\begin{itemize}
    \item $F = \operatorname{T}(\mathbb{X})$: the total variation metric $\normt{\mu - \nu}$;
    \item $F = \operatorname{BL}(\mathbb{X})$: the bounded-Lipschitz metric $\rho_{BL}(\mu, \nu)$;
    \item $F = \operatorname{W}(\mathbb{X})$: the Wasserstein-1 metric $W_1(\mu, \nu)$.
\end{itemize}

We also define uniform metrics on stochastic kernels; let ${\cal T}, {\cal S}: \mathbb{X} \times \mathbb{U} \to {\cal P}(\mathbb{X})$:
\begin{align}\label{uniform_metric}
d_{F}(\mathcal{T}, \mathcal{S}):=\sup_{(x, u) \in(\mathbb{X}, \mathbb{U})}\ \sup_{f \in F} \left| \int_{\mathbb{X}} f(y) \mathcal{T}(dy \mid x, u) - \int_{\mathbb{X}} f(y) \mathcal{S}(dy \mid x, u)  \right|.
\end{align}
By choosing different function classes for $F$, we obtain the uniform total variation distance $d_{TV}({\mathcal{T}, \mathcal{S}})$, the uniform bounded-Lipschitz distance $d_{BL}(\mathcal{T}, \mathcal{S})$, and the uniform Wasserstein-1 distance $d_{W_1}(\mathcal{T}, \mathcal{S})$, respectively. The same construction applies to kernels defined on the belief space: for $\eta, \eta' : \PP(\X) \times \U \to {\cal P}(\PP(\X))$ we write
\begin{align}\label{uniform_metric_belief}
d_F\left(\eta, \eta^{\prime}\right):=\sup _{(z, u) \in \PP(\X) \times \U} \sup _{f \in F}\left|\int f(z_1) \eta(dz_1 \mid z, u)-\int f(z_1) \eta^{\prime}(dz_1 \mid z, u)\right|,
\end{align}
where now the test functions in $F$ are defined on $\PP(\X)$. For the bounded-Lipschitz and Wasserstein-1 distances between filter kernels, the test functions depend on the metric on $\PP(\X)$, which will be either $\rho_{BL}$ or $W_1$. For compact $\X$ these two metrics induce the same (weak) topology on $\PP(\X)$ (Remark \ref{rem:metrics}(ii)), but they give different Lipschitz constants. For $d_{BL}(\eta,\eta')$ we always metrize $\PP(\X)$ by $\rho_{BL}$. For the Wasserstein-1 distance we indicate the metric on $\PP(\X)$ by a superscript:
\[
\dWBL(\eta,\eta') \ \text{ when $\PP(\X)$ is metrized by $\rho_{BL}$}, \qquad \dWW(\eta,\eta') \ \text{ when $\PP(\X)$ is metrized by $W_1$},
\]
and in the same way $\WBL(\Theta,\Theta')$ and $\WW(\Theta,\Theta')$ for $\Theta,\Theta'\in\PP(\PP(\X))$. Since $\rho_{BL} \le W_1$,
\begin{align}\label{eq:WBL_vs_W1}
\WBL(\Theta,\Theta')\le \WW(\Theta,\Theta'), \qquad \dWBL(\eta,\eta')\le \dWW(\eta,\eta').
\end{align}
For $f:\PP(\X)\to\bR$, $\norm{f}_{\Lip}$ denotes the Lipschitz constant of $f$ with respect to $W_1$, and $\norm{f}_{L,\rho_{BL}}$ the one with respect to $\rho_{BL}$.

\begin{remark}[Relations between the metrics]\label{rem:metrics}
Let $(\mathbb{X},d)$ be a metric space and $\mu,\nu\in{\cal P}(\X)$.
\begin{itemize}
\item[(i)] Since $\operatorname{BL}(\mathbb{X}) \subset \operatorname{T}(\mathbb{X})$ and $\operatorname{BL}(\mathbb{X}) \subset \operatorname{W}(\mathbb{X})$,
$\rho_{BL}(\mu,\nu) \leq \normt{\mu-\nu}$ and $\rho_{BL}(\mu,\nu) \leq W_1(\mu,\nu)$.
\item[(ii)] If $\mathbb{X}$ is bounded with diameter $D:=\sup_{x,y\in\X}d(x,y)$, then for any $f \in \operatorname{W}(\mathbb{X})$ one may subtract a constant so that $\norm{f - c}_\infty \leq D/2$; consequently
\[
W_1(\mu,\nu) \;\leq\; \frac{D}{2}\, \normt{\mu-\nu}, \qquad
W_1(\mu,\nu) \;\leq\; \Big(1+\frac{D}{2}\Big)\, \rho_{BL}(\mu,\nu).
\]
In particular, on a compact (hence bounded) space, $\rho_{BL}$ and $W_1$ are equivalent metrics and both metrize the topology of weak convergence on ${\cal P}(\X)$ \cite{villani2008optimal,Bog07}. The total variation metric, on the other hand, dominates both, but is \emph{not} equivalent to them even on compact spaces: e.g., $\normt{\delta_{x} - \delta_{x'}} = 2$ whenever $x \neq x'$, while $W_1(\delta_x,\delta_{x'}) = d(x,x')$. The same bounds hold for the uniform metrics on kernels on $\X$. For the filter kernels, since $\rho_{BL}\le W_1\le (1+\frac{D}{2})\rho_{BL}$ on $\PP(\X)$,
\begin{align*}
&d_{BL}(\eta,\eta') \le \dWBL(\eta,\eta') \le 2\, d_{BL}(\eta,\eta'),\\
&\dWBL(\eta,\eta') \le \dWW(\eta,\eta') \le \Big(1+\frac{D}{2}\Big)\dWBL(\eta,\eta'),
\end{align*}
where $\dWBL \le 2\, d_{BL}$ uses that the diameter of $(\PP(\X),\rho_{BL})$ is at most $2$.
\item[(iii)] The total variation metric might be too strong for some approximation schemes: as detailed in Remark \ref{rem:whyW1}, finitely supported (e.g., quantized) approximations might not converge in total variation, while they do converge in $W_1$; this is precisely why the refinements of our bounds in which $d_{TV}(\T_n,\T)$ is replaced by $d_{W_1}(\T_n,\T)$ are essential for the applications in Section \ref{sec:quantization}. Finally, on the belief space ${\cal P}(\X)$, the bounded-Lipschitz and Wasserstein metrics are the natural choices, since they metrize the weak topology with respect to which the filter kernel might have continuity (weak Feller) properties, whereas the filter kernel is in general not continuous in total variation (too demanding for filter kernels).
\end{itemize}
\end{remark}

\section{Preliminaries: Filter Kernel Regularity}\label{sec:preliminaries}

\subsubsection*{Weak Feller continuity of the belief-MDP}\label{sec:regularity}

This subsection reviews the weak Feller property of the kernel defined in (\ref{kernelFilter}); that is, the property that for every $f \in C({\cal P}(\mathbb{X}))$, \[\int f(z_1)\eta(dz_1 |z_0=z,u_0=u) : {\cal P}(\mathbb{X}) \times \mathbb{U} \to \mathbb{R}\]
is continuous in $(z,u)$.

\begin{assumption}\label{TV_channel}
\begin{itemize}
\item[(i)] The transition probability $\mathcal{T}(\cdot|x,u)$ is weakly continuous (weak Feller) in $(x,u)$, i.e., for any $(x_n,u_n)\to (x,u)$, $\mathcal{T}(\cdot|x_n,u_n)\to \mathcal{T}(\cdot|x,u)$ weakly.
\item[(ii)] The observation channel $Q(\cdot|x,u)$ is continuous in total variation, i.e., for any $(x_n,u_n) \to (x,u)$, $Q(\cdot|x_n,u_n) \rightarrow Q(\cdot|x,u)$ in total variation.
\end{itemize}
\end{assumption}

\begin{assumption}\label{TV_kernel}
\begin{itemize}
\item[(i)] The transition probability $\mathcal{T}(\cdot|x,u)$ is continuous in total variation in $(x,u)$, i.e., for any $(x_n,u_n)\to (x,u)$, $\mathcal{T}(\cdot|x_n,u_n) \to \mathcal{T}(\cdot|x,u)$ in total variation.
\item[(ii)] The observation channel $Q(\cdot|x)$ is independent of the control variable.
\end{itemize}
\end{assumption}

\begin{theorem} \label{Weak_tran_channel_thm_comb}
\begin{itemize}
\item[(i)] \cite[Theorem 3.7]{FeKaZg14}
Under Assumption \ref{TV_channel}, the transition probability $\eta(\cdot|z,u)$ of the filter process is weakly continuous in $(z,u)$.
\item[(ii)] \cite{KSYWeakFellerSysCont}
Under Assumption \ref{TV_kernel}, the transition probability $\eta(\cdot|z,u)$ of the filter process is weakly continuous in $(z,u)$.
\end{itemize}
\end{theorem}

For the semi-uniform Feller framework, see \cite[Corollary 6.11(b)]{feinberg2022markov} and the equivalent conditions in \cite{feinberg2023equivalent}. Moreover, \cite[Examples 4.1 and 4.3]{FeKaZg14} show that, when $\T$ is only weakly continuous, the total variation continuity of $Q$ in Assumption \ref{TV_channel}(ii) cannot be relaxed to setwise or weak continuity. On the weak Feller property, further results are reported in \cite{CrDo02,feinberg2022markov,feinberg2023equivalent,kara2020near}.

\subsubsection*{Uniform Wasserstein continuity of the belief-MDP}\label{wass_cont}

Recently, \cite{demirci2023average,demirciRefined2023} presented the following regularity results for controlled filter processes, which will later be critical for our robustness results in Section \ref{robustness}. In Theorems \ref{ergodicity} and \ref{WassersteinCont2} below, $\PP(\X)$ is metrized by $W_1$. Let us first recall the following:

\begin{definition}[Dobrushin coefficient {\cite[Equation 1.16]{dobrushin1956central}}]
For a kernel operator $K:S_{1} \to \mathcal{P}(S_{2})$ (that is, a regular conditional probability from $S_1$ to $S_2$) for standard Borel spaces $S_1, S_2$, the Dobrushin coefficient is defined as
\begin{align}
\delta(K)&=\inf\sum_{i=1}^{n}\min(K(x,A_{i}),K(y,A_{i}))\label{Dob_def}
\end{align}
where the infimum is over all $x,y \in S_{1}$ and all finite partitions $\{A_{i}\}_{i=1}^{n}$ of $S_{2}$.
\end{definition}

\begin{assumption}\label{main_assumption}
\noindent
\begin{enumerate}
\item \label{compactness}
$(\mathbb{X}, d)$ is a compact metric space, with diameter $D=\sup_{x,y \in \mathbb{X}} d(x,y)$.
\item \label{regularity}
There exists 
$\alpha \in \mathbb{R}^{+}$ such that 
$$
\left\|\mathcal{T}(\cdot \mid x, u)-\mathcal{T}\left(\cdot \mid x^{\prime}, u\right)\right\|_{T V} \leq \alpha\, d(x, x^{\prime})
$$
for every $x,x' \in \mathbb{X}$, $u \in \mathbb{U}$.
\end{enumerate}
\end{assumption}

\begin{theorem}[{\cite[Theorem 2.2]{demirci2023average}}]\label{ergodicity}
    Assume that $\mathbb{X}$ and $\mathbb{Y}$ are Polish spaces. 
    If Assumption \ref{main_assumption} holds, then we have
    $$
    \WW\left(\eta(\cdot \mid z_0, u), \eta\left(\cdot \mid z_0^{\prime},u\right)\right) 
    \leq K_2 W_{1}\left(z_0, z_0^{\prime}\right),$$
with    
\begin{align}\label{K2DefWasF}
K_2:=\frac{\alpha D (3-2\delta(Q))}{2},
\end{align}
    for all $z_0,z_0' \in \cal{P}(\mathbb{X})$, $u \in \mathbb{U}$.
\end{theorem}

Similarly, \cite{DKY_ACC2025} established the following complementary result under the assumption that the transition kernel is continuous with respect to the $W_1$ metric:

\begin{assumption}\label{weakTtvQweakTtvQ2}
\begin{itemize}
\item[(i)] $(\mathbb{X}, d)$ is a compact metric space.
\item[(ii)] There exists a constant $\theta \in (0,1)$ such that
$$
W_1\left(\mathcal{T}(\cdot \mid x, u), \mathcal{T}\left(\cdot \mid x^{\prime}, u\right)\right) \leq \theta\, d\left(x, x^{\prime}\right)
$$
for every $x, x^{\prime} \in \mathbb{X}$, $u \in \mathbb{U}$.
\item[(iii)] There exists a constant $L_Q \in \mathbb{R}^{+}$ such that
$$
\left\|Q(\cdot \mid x) - Q\left(\cdot \mid x^{\prime}\right)\right\|_{TV} \leq L_Q\, d\left(x, x^{\prime}\right)
$$
for every $x, x^{\prime} \in \mathbb{X}$.
\end{itemize}
\end{assumption}

\begin{theorem}\label{WassersteinCont2}
    Assume that $\mathbb{X}$ and $\mathbb{Y}$ are Polish spaces. 
    Under Assumption \ref{weakTtvQweakTtvQ2}, we have
$$
\WW\left(\eta\left(\cdot \mid z_0, u\right), \eta\left(\cdot \mid z_0^{\prime}, u\right)\right) \leq \left( \theta +\frac{3 \theta L_Q D}{2} \right) W_1\left(z_0, z_0^{\prime}\right)
$$
for all $z_0, z_0^{\prime} \in {\cal P}(\X)$, $u \in \mathbb{U}$, where $D=\sup _{x, y \in \mathbb{X}} d(x, y)$.
\end{theorem}
For a proof see Appendix \ref{app:wasserstein}.

\begin{remark}[Trade-off between regularity of $\T$ and regularity of $Q$]\label{rem:tradeoff}
As in Theorem \ref{Weak_tran_channel_thm_comb}, one can trade off regularity assumptions on the transition kernel against those on the observation kernel. Theorems \ref{ergodicity} and \ref{WassersteinCont2} manifest this trade-off for uniform Wasserstein-1 regularity: the total variation regularity is imposed on the transition kernel in the former, whereas the latter imposes only Wasserstein regularity on the transition kernel at the price of total variation regularity of the channel. We write $K_2$ for the regularity coefficient of $\eta$, i.e., $K_2 = \alpha D(3-2\delta(Q))/2$ under Assumption \ref{main_assumption} and $K_2 = \theta + 3\theta L_Q D/2$ under Assumption \ref{weakTtvQweakTtvQ2}.
\end{remark}

\begin{remark}\label{rem:metric_choice}
The constant $K_2$ in Theorem \ref{ergodicity} depends on the metric on $\X$. Let
\[
d_{\T}(x,x') := \sup_{u\in\U}\normt{\T(\cdot|x,u)-\T(\cdot|x',u)},
\]
and assume that $d_{\T}$ is a metric (otherwise, one can use $d_{\T}+\epsilon d$ and let $\epsilon\to 0$). Since $d_{\T}\le\alpha d$, $(\X,d_{\T})$ is compact. With respect to $d_{\T}$, Assumption \ref{main_assumption} holds with $\alpha=1$ and diameter at most $2(1-\delta(\T))$, where $\delta(\T):=\inf_{u\in\U}\delta(\T(\cdot|\cdot,u))$. Hence Theorem \ref{ergodicity} holds with
\[
K_2=(1-\delta(\T))(3-2\delta(Q)),
\]
and $(1-\delta(\T))(3-2\delta(Q))<1$ is sufficient for contraction.
\end{remark}

\section{Proximity of Belief-MDP Kernels in terms of Proximities of Transition and Measurement Kernels}\label{sec:proximity}

This section contains the main technical results of the paper. Consider two models:
\begin{itemize}
\item[(i)] A true model determined by \((\T, \Q)\), leading to a filter kernel $\eta$. Given a fixed prior, this induces the measure $P$ on ${\cal B}(\mathbb{X}^{\mathbb{Z}_+} \times \mathbb{Y}^{\mathbb{Z}_+})$.
\item[(ii)] An approximate model determined by \((\T_n, \Q_n)\), leading to a filter kernel $\eta^{{\cal T}_n,Q_n}$. Given a fixed prior, this induces the measure $P^{{\cal T}_n,Q_n}$ on ${\cal B}(\mathbb{X}^{\mathbb{Z}_+} \times \mathbb{Y}^{\mathbb{Z}_+})$.
\end{itemize}
Our goal is to bound $d_{BL}(\eta,\eta^{{\cal T}_n,Q_n})$, $\dWBL(\eta,\eta^{{\cal T}_n,Q_n})$ and $\dWW(\eta,\eta^{{\cal T}_n,Q_n})$, defined via (\ref{uniform_metric_belief}), in terms of the kernel distances $d_{TV}(\T_n,\T)$, $d_{W_1}(\T_n,\T)$ and $d_{TV}(Q_n,Q)$. The proofs of the results of this section are given in Appendix \ref{app:proofs}.

\subsection{Filter Kernel Proximities under the Bounded-Lipschitz Metric}

We first present a bound which holds without any assumptions on the models.

\begin{theorem}\label{main1}
 $d_{BL}(\eta,\eta^{{\cal T}_n,Q_n}) \leq 2 \left( d_{TV}(\mathcal{T}_n,\mathcal{T}) + d_{TV}(Q_n,Q)\right).$
\end{theorem}

The following assumption leads to a refinement in which the total variation distance between the transition kernels is replaced by their (weaker) Wasserstein-1 distance; as discussed in Remarks \ref{rem:metrics} and \ref{rem:whyW1}, this refinement is what makes the bounds applicable to quantized approximations.

\begin{assumption}\label{channel_reg}
For the measurement channels $\{Q_n\}_n$ and $Q$, there exists $L_Q<\infty$ such that
\begin{align*}
&\|Q_n(\cdot|x)-Q_n(\cdot|x')\|_{TV}\leq L_Q\, d(x,x'),\\
&\|Q(\cdot|x)-Q(\cdot|x')\|_{TV}\leq L_Q\, d(x,x'),
\end{align*}
for all $x,x'\in\mathbb{X}$ and all $n$.
\end{assumption}

\begin{remark}\label{rem:LQuniform}
A common constant $L_Q$ for the whole family $\{Q_n\}$ may appear restrictive; however, it is satisfied when the approximate channels are constructed from $Q$. If each $Q_n$ is a garbling (stochastic degradation) of $Q$, i.e., if for a stochastic kernel $K_n$ on $\Y$ given $\Y$ we have $Q_n(\cdot|x) = \int_{\Y} K_n(\cdot|y)\, Q(dy|x)$, then, for all $x,x'$,
\[
\|Q_n(\cdot|x)-Q_n(\cdot|x')\|_{TV} \leq \|Q(\cdot|x)-Q(\cdot|x')\|_{TV} \leq L_Q\, d(x,x'),
\]
so the Lipschitz constant of $Q$ is inherited, uniformly in $n$. This covers, in particular, the quantized channels of Section \ref{sec:observation_quant}.
\end{remark}

\begin{theorem}\label{main2}
Under Assumption \ref{channel_reg}, we have that
$$d_{BL}(\eta,\eta^{{\cal T}_n,Q_n}) \leq (2 L_Q + 1)\, d_{W_1}\left(\mathcal{T}_n,\mathcal{T}\right) + 2\, d_{TV}(Q_n,Q).$$
\end{theorem}

\subsection{Filter Kernel Proximities under the Wasserstein Metric}

We now extend the bounds to the $W_1$ distance on the belief space, assuming $\X$ is compact. We state each bound both for $\dWBL$ and for $\dWW$ (see Section \ref{sec:notation}). Throughout, $D:=\sup_{x,y \in \X} d(x,y)$ and
\begin{align}\label{eq:CD}
\CD := 2 + \frac{D}{D+2} \;\le\; 3 .
\end{align}
First, we state the Wasserstein-1 version of Theorem \ref{main1}.

\begin{theorem}\label{W1main_1}
  Let $\X$ be a compact metric space with diameter $D$. Then:
\begin{itemize}
\item[(a)] (\,$\PP(\X)$ metrized by $\rho_{BL}$\,)
\[\dWBL(\eta,\eta^{{\cal T}_n,Q_n}) \leq \CD \left( d_{TV}(\mathcal{T}_n,\mathcal{T}) + d_{TV}(Q_n,Q)\right);\]
\item[(b)] (\,$\PP(\X)$ metrized by $W_1$\,)
\[\dWW(\eta,\eta^{{\cal T}_n,Q_n}) \leq \frac{3D}{2} \left( d_{TV}(\mathcal{T}_n,\mathcal{T}) + d_{TV}(Q_n,Q)\right).\]
\end{itemize}
\end{theorem}

The constant $\CD$ in part (a) is bounded by $3$, while the constant in part (b) grows with $D$, and this cannot be avoided (Remark \ref{rem:scaling}). We nevertheless need part (b). In Section \ref{robustness} these bounds are multiplied by the Lipschitz constants of the value functions, and both must be computed with the same metric on $\PP(\X)$. Since the cost $c$ is $K_1$-Lipschitz on $\X$ (Assumption \ref{cost_assumption}), the belief cost $\tilde c(\cdot,u)$ is $K_1$-Lipschitz with respect to $W_1$, and the value functions are Lipschitz with respect to $W_1$ with the explicit constant (\ref{LipBoundJ}) (Theorem \ref{ExistenceDiscount}). Therefore Section \ref{robustness} uses part (b).

Next, we present the Wasserstein-1 version of Theorem \ref{main2}.

\begin{theorem}\label{W1main_2}
Let $\X$ be compact with diameter $D$. Under Assumption \ref{channel_reg}, we have
\begin{itemize}
\item[(a)] (\,$\PP(\X)$ metrized by $\rho_{BL}$\,)
\[\dWBL(\eta,\eta^{{\cal T}_n,Q_n}) \leq \CD \left( L_Q\, d_{W_1}\left(\mathcal{T}_n,\mathcal{T}\right) + d_{TV}(Q_n,Q)\right) + d_{W_1}\left(\mathcal{T}_n,\mathcal{T}\right);\]
\item[(b)] (\,$\PP(\X)$ metrized by $W_1$\,)
\[\dWW(\eta,\eta^{{\cal T}_n,Q_n}) \leq \frac{3D}{2} \left( L_Q\, d_{W_1}\left(\mathcal{T}_n,\mathcal{T}\right) + d_{TV}(Q_n,Q)\right) + d_{W_1}\left(\mathcal{T}_n,\mathcal{T}\right).\]
\end{itemize}
In both parts, $d_{W_1}(\T_n,\T)$ is the distance (\ref{uniform_metric}) between the transition kernels on $\X$.
\end{theorem}

\begin{remark}\label{rem:sharppointer}
Examples~\ref{prop:sharpness} and~\ref{ex:channel_sharp} show that the linear dependence on $d_{W_1}(\T_n,\T)$ and $d_{TV}(Q_n,Q)$ can not be improved; Examples~\ref{prop:TVnecessity} and~\ref{prop:LQnecessity} show, respectively, that the channel term cannot be measured in any metric weaker than total variation, and that the Wasserstein refinement of Theorems~\ref{main2} and~\ref{W1main_2} genuinely requires the channel regularity of Assumption~\ref{channel_reg}.
\end{remark}

\begin{remark}\label{rem:scaling}
In part (b) of Theorems \ref{W1main_1} and \ref{W1main_2} the constant in front of the total variation terms has to grow linearly in $D$: scaling the metric $d$ by $\lambda>0$ scales $\dWW(\eta,\eta')$ and $D$ by $\lambda$, but leaves $d_{TV}$ unchanged. For an explicit example, let $\X=\{0,D\}$, let $\T=\T_n$ be the identity kernel, $\Y=\{0,1\}$, and, for $\delta\in(0,1)$,
\[
Q(\cdot|0)=(0,1),\quad Q(\cdot|D)=(1-\delta,\delta),\qquad Q_n(\cdot|0)=(\delta,1-\delta),\quad Q_n(\cdot|D)=(1,0),
\]
so that $d_{TV}(Q_n,Q)=2\delta$, while $d_{TV}(\T_n,\T)=d_{W_1}(\T_n,\T)=0$. For $s\in[0,1]$, let $b_s:=s\,\delta_0+(1-s)\,\delta_D$ denote the belief with mass $s$ at state $0$; then $W_1(b_s,b_{s'})=D|s-s'|$, so $(\PP(\X),W_1)$ is isometric to $[0,D]$. At the uniform prior $z_0=b_{1/2}$, the observation $y=0$ has probability $\frac{1-\delta}{2}$ under $Q$ and $\frac{1+\delta}{2}$ under $Q_n$, and Bayes' rule gives
\[
\eta(\cdot|z_0)=\tfrac{1-\delta}{2}\,\delta_{b_0}+\tfrac{1+\delta}{2}\,\delta_{b_{1/(1+\delta)}},\qquad
\eta^{\T,Q_n}(\cdot|z_0)=\tfrac{1+\delta}{2}\,\delta_{b_{\delta/(1+\delta)}}+\tfrac{1-\delta}{2}\,\delta_{b_1}.
\]
Computing the Wasserstein distance on $[0,D]$ through the distribution functions,
\begin{align*}
\WW\bigl(\eta(\cdot|z_0),\eta^{{\cal T},Q_n}(\cdot|z_0)\bigr)
&= D\Bigl[\tfrac{1-\delta}{2}\cdot\tfrac{\delta}{1+\delta} + \delta\cdot\tfrac{1-\delta}{1+\delta} + \tfrac{1-\delta}{2}\cdot\tfrac{\delta}{1+\delta}\Bigr]\\
&= 2D\delta\,\frac{1-\delta}{1+\delta}
= D\,\frac{1-\delta}{1+\delta}\, d_{TV}(Q_n,Q).
\end{align*}
Since $\delta\in(0,1)$ is arbitrary, the constant in front of $d_{TV}(Q_n,Q)$ in part (b) of Theorems \ref{W1main_1} and \ref{W1main_2} must be at least $D$.
\end{remark}

\begin{remark}\label{rem:fgpointer}
Proposition~\ref{prop:FG_regularity} translates each assumption on $(\T,Q)$ into conditions on $F$ and $G$ in~\eqref{updateEq}, and bounds the kernel distances by the corresponding perturbations of $F$ and $G$.
\end{remark}

\subsection{Simultaneous Perturbations of the Initial Distribution}\label{sec:initpert}

The proximity bounds concern the kernels; combining them with the regularity results of Section \ref{sec:regularity} allows us to also account for a mismatch in the initialization, i.e., to compare $\eta(\cdot|\mu_0,u)$ and $\eta^{{\cal T}_n,Q_n}(\cdot|\nu_0,u)$ for different priors $\mu_0, \nu_0$.

\begin{corollary}\label{cor:init_pert_W1}
Let $\X$ be compact with diameter $D$, let $\Y$ be a Polish space, and assume that Assumption \ref{main_assumption} or Assumption \ref{weakTtvQweakTtvQ2} holds ($K_2$ as in Remark \ref{rem:tradeoff}). Then, for all $\mu_0,\nu_0 \in {\cal P}(\X)$ and $u \in \U$,
\begin{align*}
 \WW\left(\eta\left(\cdot \mid \mu_0, u\right), \eta^{{\cal T}_n,Q_n}\left(\cdot \mid \nu_0, u\right)\right)
&\leq K_2 W_1\left(\mu_0, \nu_0\right)\\
&\quad + \frac{3D}{2} \left( d_{TV}(\T_n,\T) + d_{TV}(Q_n,Q) \right).
\end{align*}
Moreover, if Assumption \ref{channel_reg} also holds, then
\begin{align*}
\WW\left(\eta(\cdot \mid \mu_0, u), \eta^{{\cal T}_n,Q_n}\left(\cdot \mid \nu_0, u\right)\right)
&\leq K_2 W_1\left(\mu_0, \nu_0\right)\\
&\quad + \frac{3D}{2} \left( L_Q\, d_{W_1}(\T_n,\T) + d_{TV}(Q_n,Q)\right)
 + d_{W_1}(\T_n,\T).
\end{align*}
The same bounds hold with $\WW$ on the left-hand side replaced by $\WBL$ and $3D/2$ replaced by $\CD$.
\end{corollary}

\begin{proof}
By the triangle inequality,
\begin{align*}
\WW(\eta(\cdot|\mu_0,u), \eta^{{\cal T}_n,Q_n}(\cdot|\nu_0,u)) &\leq \WW(\eta(\cdot|\mu_0,u),\eta(\cdot|\nu_0,u))\\
&\quad + \WW(\eta(\cdot|\nu_0,u), \eta^{{\cal T}_n,Q_n}(\cdot|\nu_0,u));
\end{align*}
the first term is bounded by Theorem \ref{ergodicity} (respectively, Theorem \ref{WassersteinCont2}) and the second by Theorem \ref{W1main_1}(b) (respectively, Theorem \ref{W1main_2}(b)). For the $\WBL$ version, the first term is bounded by $\WW(\eta(\cdot|\mu_0,u),\eta(\cdot|\nu_0,u))$ by (\ref{eq:WBL_vs_W1}), and the second by part (a) of the same theorems.
\end{proof}

\section{Robustness Bounds via the Belief-MDP Reduction}\label{robustness}

We now study the robustness problem. A POMDP is equivalent to a fully observed MDP on the belief space with transition kernel $\eta$ and cost $\tilde{c}$. The robustness problem is then reduced to that of the belief-MDP. The explicit robustness bounds for POMDPs follow by applying the proximity results of Section \ref{sec:proximity} together with the MDP robustness results of \cite{zhou2024robustness} on the belief space. On prior work, we refer the reader to \cite{kara2020robustness} for robustness to transition kernels and to \cite{WuVerdu,YukselOptimizationofChannels,YukselBasarBook24,hogeboom2021sequential} for the special case of perturbations of measurement channels.

\subsection{Regularity of Value Functions and Optimal Solutions}\label{sec:value_regularity}

For the robustness analysis, we will critically build on the following regularity properties of optimal solutions. We first formalize the regularity notion used throughout.
\begin{definition}[Wasserstein regularity]\label{def:wass_regular}
A transition kernel $\eta:{\cal P}(\X) \times \U \to {\cal P}({\cal P}(\X))$ is said to be \emph{Wasserstein-regular with constant $K$} if
\[
\WW\bigl(\eta(\cdot|z,u), \eta(\cdot|z',u)\bigr) \leq K\, W_1(z,z') \qquad \text{for all } z,z' \in {\cal P}(\X),\ u \in \U.
\]
\end{definition}
Theorems \ref{ergodicity} and \ref{WassersteinCont2} state precisely that the filter kernel $\eta$ is Wasserstein-regular with constant $K_2$ (as in Remark \ref{rem:tradeoff}) under Assumption \ref{main_assumption} or \ref{weakTtvQweakTtvQ2}.

\noindent{\bf Discounted cost.} An implication of Theorem \ref{ergodicity} or Theorem \ref{WassersteinCont2} is that solutions to the discounted cost optimality equation are Lipschitz continuous \cite[Theorems 3.1 and 3.2]{tutorialkara2024partially}:

\begin{assumption}\label{cost_assumption}
The cost function $c$ is continuous and bounded, and there exists $K_1 \in \mathbb{R}^+$ such that $|c(x,u) - c(x',u)| \leq K_1 d(x,x')$ for every $x,x' \in \mathbb{X}$, $u \in \mathbb{U}$.
\end{assumption}

\begin{theorem} \cite[Lemma 4.2]{demirci2023average}\label{ExistenceDiscount}
If the cost function $c: \mathbb{X} \times \mathbb{U} \to \mathbb{R}$ is continuous and bounded, and $\mathbb{U}$ is compact, then under the assumptions of Theorem \ref{Weak_tran_channel_thm_comb}, for any $\beta \in (0,1)$, there exists an optimal solution to the discounted cost problem with a continuous and bounded value function. Furthermore, under Assumption \ref{cost_assumption} and under Assumption \ref{main_assumption} with $K_2=\frac{\alpha D (3-2\delta(Q))}{2}$, or under Assumption \ref{weakTtvQweakTtvQ2} with $K_2 = \theta +\frac{3 \theta L_Q D}{2}$, if $\beta K_2 < 1$ then the value function $J^*_\beta(\eta)$ is Lipschitz continuous on $({\cal P}(\X),W_1)$, with the explicit bound
\begin{align}\label{LipBoundJ}
\left\Vert J^*_{\beta}(\eta) \right\Vert _{\Lip} \leq \frac{K_1}{1 - \beta K_2}
\end{align}
\end{theorem}

\noindent{\bf Average cost.} 
The average cost optimality equation (ACOE) plays a crucial role in the analysis and the existence results for MDPs under the infinite horizon average cost criterion. The triplet $(h,\rho^*,\gamma^*)$, where $h:{\cal P}(\mathbb{X})\to \mathbb{R}$ is a measurable function, $\gamma^*$ is a stationary policy, and $\rho^*\in\mathbb{R}$ is a constant, forms the ACOE if 
\begin{align}\label{acoe}
h(z)+\rho^*&=\inf_{u\in\mathbb{U}}\left\{\tilde{c}(z,u) + \int h(z_1)\eta(dz_1|z,u)\right\}\nonumber\\
&=\tilde{c}(z,\gamma^*(z)) + \int h(z_1)\eta(dz_1|z,\gamma^*(z))
\end{align}
for all $z\in{\cal P}(\mathbb{X})$. It is well known (see e.g. \cite[Theorem 5.2.4]{HernandezLermaMCP}) that if (\ref{acoe}) is satisfied with the triplet $(h,\rho^*,\gamma^*)$, and furthermore if $h$ satisfies $\sup_{\gamma\in\Gamma}\lim_{t \to \infty}\frac{E_z^\gamma[h(Z_t)] }{t}=0$ for all $z\in{\cal P}(\mathbb{X})$ (which holds automatically when $h$ is bounded, as is the case below), then $\gamma^*$ is an optimal policy for the POMDP under the infinite horizon average cost criterion, and $J^*_\infty(z)=\inf_{\gamma\in\Gamma}J_\infty(z,\gamma)=\rho^*$ for all $z\in {\cal P}(\mathbb{X})$.

\begin{theorem}\label{mainEmre}
Suppose that $\U$ is compact and Assumption \ref{cost_assumption} holds. Assume either
\begin{itemize}
\item[(a)]\cite[Theorem 1.2]{demirci2023average} Assumption \ref{main_assumption} together with total variation continuity of $\T(\cdot|x,u)$ in $u$, with $K_2=\frac{\alpha D (3-2\delta(Q))}{2} < 1$; or
\item[(b)]\cite[Theorem 3.7]{DKY_ACC2025} Assumption \ref{weakTtvQweakTtvQ2} together with weak continuity of $\T(\cdot|x,u)$ in $u$, with $K_2=\theta +\frac{3 \theta L_Q D}{2}  < 1$.
\end{itemize}
Then a solution $(h,\rho^*,\gamma^*)$ to the average cost optimality equation (\ref{acoe}) exists. Consequently, an optimal stationary control policy exists, and the optimal cost is constant over initial beliefs.
\end{theorem}

The solution $h$ is obtained through a vanishing discount argument as a limit of discounted value functions, whose Lipschitz constants are uniformly bounded by (\ref{LipBoundJ}); in particular, $h$ is Lipschitz with $\norm{h}_{\Lip} \leq K_1/(1-K_2)$; see \cite{demirci2023average}. We also note that, for this setup, minorization-type conditions which are typically imposed in average cost optimal control are not applicable for belief-MDPs, which is why the contraction approach above is adopted.

\begin{remark}[A convex-analytic alternative]\label{rem:convex_analytic}
For completeness, we note that under only weak Feller regularity of $\eta$ (e.g., under Theorem \ref{Weak_tran_channel_thm_comb}), with $c$ continuous and bounded and $\U$ compact, an average-cost optimal policy can also be obtained through the convex analytic method \cite[Theorem 3]{anotherLookPOMDPs}, building on \cite{Borkar2}. We emphasize, however, that the optimality is guaranteed for initial distributions compatible with the minimizing invariant measure (as in the minimal-pair formulations of \cite{HernandezLermaMCPFurther}), rather than for every initialization as in Theorem \ref{mainEmre}.
\end{remark}

\subsection{Policies Designed for an Approximate Model: Two Settings}\label{sec:prior_robust}

Let us define the optimal cost function for the belief-MDP (i.e., the fully observable case) as
\(J^*_\beta(\tilde{c}, \mu, \eta)\),
where $\tilde{c}$ is the transformed cost function, $\mu$ is the initial distribution, $\eta$ is the transition kernel, and $\beta$ is the discount factor. We denote this function by \( J^*_\beta(\eta) (\cdot) \) when considering it as a function of \( \mu \), and let the corresponding optimal policy be denoted by \( \gamma^*_\eta \). 
Similarly, we define \( J_\beta(\eta, \gamma) (\cdot) \) as the cost function of the policy \( \gamma \) as a function of the initial distribution.

From the definition of the optimal cost, we have the equivalences
$J^*_{\beta}(c,\mu,{\cal T}, Q) = J^*_\beta(\tilde{c}, \mu, \eta)$
and
$J^*_{\beta}(c,\mu,{\cal T}_n, Q_n) = J^*_\beta(\tilde{c}, \mu, \eta^{{\cal T}_n, Q_n})$, where the belief-MDPs are started at $\mu$; see below for the role of the first observation $Y_0$.

Let $\gamma^*_{\eta^{{\cal T}_n,Q_n}}$ be an optimal stationary policy of the approximate model $(\T_n,Q_n)$, and let the true system be $(\T,Q)$ with $X_0\sim\mu_0$. We denote by $F^{\T',Q'}(z,u,y)$ the Bayes map in (\ref{kernelFilter}) computed with a model $(\T',Q')$, so that $F=F^{\T,Q}$, and by $\Phi^{Q'}(z,y)$ the conditional law of $X_0$ given $Y_0=y$ when $X_0\sim z$ and $Y_0$ is generated by $Q'$; these are fixed versions of the conditional probabilities, defined for every $y\in\Y$. The true filter is $z_0=\Phi^{Q}(\mu_0,Y_0)$, $z_{t+1}=F(z_t,U_t,Y_{t+1})$. We distinguish two ways in which $\gamma^*_{\eta^{{\cal T}_n,Q_n}}$ can be used on the true system.
\begin{itemize}
\item[(I)] \emph{Policy transfer.} The controller knows the true model and runs its filter, possibly from an incorrect prior $\nu_0$,
\begin{align}\label{eq:true_filter_nu}
\pi_0 := \Phi^{Q}(\nu_0,Y_0), \qquad \pi_{t+1} := F(\pi_t,U_t,Y_{t+1}), \qquad U_t=\gamma^*_{\eta^{{\cal T}_n,Q_n}}(\pi_t).
\end{align}
For $\nu_0=\mu_0$ we have $\pi_t=z_t$. Both a correct and an incorrect prior are treated in Sections \ref{sec:disc_robust} and \ref{sec:avg_robust}.
\item[(II)] \emph{Misspecified model.} The controller assumes that the model is $(\T_n,Q_n)$ with prior $\nu_0$, while the true model is $(\T,Q)$ with prior $\mu_0$, and runs the filter of the assumed model,
\begin{align}\label{eq:approx_filter}
\pi^n_0 := \Phi^{Q_n}(\nu_0,Y_0), \qquad \pi^n_{t+1} := F^{\T_n,Q_n}(\pi^n_t,U_t,Y_{t+1}), \qquad U_t=\gamma^*_{\eta^{{\cal T}_n,Q_n}}(\pi^n_t).
\end{align}
This is the implementable case where the controller uses only $(\T_n,Q_n)$ and $\nu_0$. It is treated in Section \ref{sec:misspecified}.
\end{itemize}
In both cases the resulting policy of the POMDP is admissible: $\gamma^*_{\eta^{{\cal T}_n,Q_n}}$ is a stationary policy of the belief-MDP, and $\pi_t$ and $\pi^n_t$ are functions of $Y_{[0,t]}$ and $U_{[0,t-1]}$. Since $I_0=\{Y_0\}$, the first action is taken after $Y_0$ is observed; thus the belief-MDP of the true model starts at $z_0=\Phi^{Q}(\mu_0,Y_0)$ and $J^*_\beta(c,\mu_0,\T,Q)=E^{\mu_0}[J^*_\beta(\eta)(z_0)]$. The identification $J^*_\beta(c,\mu,\T,Q)=J^*_\beta(\eta)(\mu)$ above corresponds to starting the belief-MDP at $\mu$, that is, to choosing the first action before $Y_0$ is observed. We write $J_\beta(c,\mu_0,\T,Q,\gamma^*_{\eta^{{\cal T}_n,Q_n}}\circ\pi)$ and $J_\beta(c,\mu_0,\T,Q,\gamma^*_{\eta^{{\cal T}_n,Q_n}}\circ\pi^n)$ for the resulting costs, and similarly for $J_\infty$. Throughout this section we abbreviate
\[
\Delta_n := \dWW\bigl(\eta,\eta^{{\cal T}_n,Q_n}\bigr),
\]
which is bounded by Theorem \ref{W1main_1}(b) and, under Assumption \ref{channel_reg}, by the refined bound of Theorem \ref{W1main_2}(b). We write $\bar\Delta_n$ for the right-hand side of the applicable bound, that is,
\begin{align}\label{eq:Deltabar_tv}
\bar\Delta_n := \frac{3D}{2}\bigl(d_{TV}(\T_n,\T)+d_{TV}(Q_n,Q)\bigr),
\end{align}
or, under Assumption \ref{channel_reg},
\begin{align}\label{eq:Deltabar_w1}
\bar\Delta_n := \frac{3D}{2}\bigl(L_Q\, d_{W_1}(\T_n,\T)+d_{TV}(Q_n,Q)\bigr) + d_{W_1}(\T_n,\T),
\end{align}
so that $\Delta_n\le\bar\Delta_n$.

\subsection{Continuity of the Optimal Costs}\label{sec:value_cont}

Using the Lipschitz regularity from Theorems \ref{ExistenceDiscount} and \ref{mainEmre}, and the proximity results from Theorems \ref{W1main_1} and \ref{W1main_2}, we establish robustness and continuity properties of optimal solutions, building on the MDP-level results of \cite{zhou2024robustness} (see also \cite{bozkurt2025model} for related results). We first bound the difference between the optimal costs of the two models; no policy is applied to the true model here, so these bounds do not depend on the setting of Section \ref{sec:prior_robust}. The bounds on the cost of applying $\gamma^*_{\eta^{{\cal T}_n,Q_n}}$ to the true model are given in Sections \ref{sec:disc_robust}--\ref{sec:misspecified}.

The following theorem is adapted from \cite[Theorem 2.4]{zhou2024robustness}; a proof sketch is given in Appendix \ref{app:mdp_proofs}. Throughout this section, $\X$ is compact, the filter kernels are compared in $\dWW$, and Lipschitz constants are taken with respect to $W_1$ (see the discussion after Theorem \ref{W1main_1}).

\begin{theorem}[Difference in Value Functions]
\label{upperbound1proof}
Let $\X$ be compact. Suppose that $c$ is continuous and bounded, $\mathbb{U}$ is compact, and $\eta,\eta^{{\cal T}_n,Q_n}$ are weak Feller. If $J^*_{\beta}(\eta)$ is Lipschitz continuous on $(\mathcal{P}(\mathbb{X}),W_1)$, then
\[
\| J^*_{\beta}(\eta) - J^*_{\beta}(\eta^{{\cal T}_n,Q_n}) \|_{\infty} \leq \frac{\beta}{1-\beta}\|J^*_{\beta}(\eta)\|_{\Lip}\, \dWW(\eta,\eta^{{\cal T}_n,Q_n}).
\]
In particular, if Assumption \ref{cost_assumption} holds and $\eta$ is Wasserstein-regular (Definition \ref{def:wass_regular}) with constant $K$ satisfying $\beta K<1$, then $\|J^*_{\beta}(\eta)\|_{\Lip} \le K_1/(1-\beta K)$ by (\ref{LipBoundJ}), and the right-hand side is at most $\frac{\beta}{1-\beta}\frac{K_1}{1-\beta K}\, \dWW(\eta,\eta^{{\cal T}_n,Q_n})$.
\end{theorem}

Under Assumption \ref{main_assumption} with $\beta K_2 < 1$ (or under Assumption \ref{weakTtvQweakTtvQ2} with $\beta K_2 = \beta( \theta +\frac{3 \theta L_Q D}{2}) < 1$) together with Assumption \ref{cost_assumption}, the Lipschitz condition of Theorem \ref{upperbound1proof} holds with $\Vert J^*_{\beta}(\eta) \Vert _{\Lip} \leq \frac{K_1}{1 - \beta K_2}$ by (\ref{LipBoundJ}). Compactness of $\U$ and the weak Feller property of $\eta$ and $\eta^{{\cal T}_n,Q_n}$ are not implied by these assumptions, which only concern regularity in $x$, so we assume them explicitly below; Theorem \ref{Weak_tran_channel_thm_comb} provides sufficient conditions. Combining with Theorems~\ref{W1main_1}(b) and~\ref{W1main_2}(b), we obtain:

\begin{corollary}\label{discount_cont}
Suppose that $\U$ is compact, that $\eta$ and $\eta^{{\cal T}_n, Q_n}$ are weak Feller, and that Assumption~\ref{main_assumption} or Assumption~\ref{weakTtvQweakTtvQ2} holds together with Assumption \ref{cost_assumption}, with $\beta K_2 < 1$. Then:
\begin{align*}
& \left\Vert J^*_{\beta}(\eta) - J^*_{\beta}({\eta^{{\cal T}_n,Q_n}}) \right\Vert_{\infty} \leq
\frac{\beta}{1-\beta} \frac{K_1}{1 - \beta K_2} \frac{3D}{2} \left( d_{TV}(\T_n,\T) + d_{TV}(Q_n,Q)\right).
\end{align*}
Furthermore, under Assumption~\ref{channel_reg}, we have the refined bound:
\begin{align*}
& \left\Vert J^*_{\beta}(\eta) - J^*_{\beta}({\eta^{{\cal T}_n,Q_n}}) \right\Vert_{\infty}\\
&\qquad \leq
\frac{\beta}{1-\beta} \frac{K_1}{1 - \beta K_2} \left[\frac{3D}{2} \left( L_{Q}\,d_{W_1}(\T_n,\T) + d_{TV}(Q_n,Q)\right) + d_{W_1}(\T_n,\T)\right].\nonumber
\end{align*}
\end{corollary}

\medskip
\noindent{\bf Incorrect prior.} We first record the error of a designer who optimizes for an incorrect model and an incorrect prior, that is, the difference between the optimal costs of the belief-MDPs of the two models started at $\mu_0$ and at $\nu_0$.

\begin{corollary}[Model and prior mismatch]\label{cor:prior_disc}
Suppose that $\U$ is compact, that $\eta$ and $\eta^{{\cal T}_n,Q_n}$ are weak Feller, and that Assumption \ref{main_assumption} or Assumption \ref{weakTtvQweakTtvQ2} holds for the true model, together with Assumption \ref{cost_assumption}, with $\beta K_2 < 1$. Then, for all $\mu_0,\nu_0 \in {\cal P}(\X)$,
\begin{align*}
\bigl|J^*_{\beta}(\eta)(\mu_0) - J^*_{\beta}(\eta^{{\cal T}_n,Q_n})(\nu_0)\bigr|
\;\leq\; \frac{K_1}{1-\beta K_2}\, W_1(\mu_0,\nu_0) \;+\; \frac{\beta}{1-\beta}\,\frac{K_1}{1-\beta K_2}\,\Delta_n .
\end{align*}
\end{corollary}

\begin{proof}
By the triangle inequality,
\begin{align*}
\bigl|J^*_{\beta}(\eta)(\mu_0) - J^*_{\beta}(\eta^{{\cal T}_n,Q_n})(\nu_0)\bigr|
\leq \bigl|J^*_{\beta}(\eta)(\mu_0) - J^*_{\beta}(\eta)(\nu_0)\bigr| + \bigl\|J^*_{\beta}(\eta) - J^*_{\beta}(\eta^{{\cal T}_n,Q_n})\bigr\|_{\infty}.
\end{align*}
The first term is bounded by $\|J^*_\beta(\eta)\|_{\Lip}\, W_1(\mu_0,\nu_0) \leq \frac{K_1}{1-\beta K_2} W_1(\mu_0,\nu_0)$, by Theorem \ref{ExistenceDiscount}; the second is bounded by Theorem \ref{upperbound1proof}.
\end{proof}

\medskip
\noindent{\bf Average cost.} In the following, we discuss the average cost setup building on \cite[Theorem 2.6 and Theorem 2.9]{zhou2024robustness}. Compared with the discounted case, the average cost case requires additional care: as noted after Theorem \ref{mainEmre}, minorization conditions typically used in average-cost MDP theory are not applicable on the belief space, and existence of solutions to the ACOE itself requires the contraction analysis of Theorem \ref{mainEmre}. The following theorem makes the required conditions fully explicit; it follows from \cite[Theorems 2.3 and 2.6]{zhou2024robustness} applied on $(\PP(\X),W_1)$, and a proof sketch is given in Appendix \ref{app:mdp_proofs}.

\begin{theorem}\label{vanishingproof1}
Assume:
\begin{itemize}
\item[(i)] The cost function $c: \mathbb{X} \times \mathbb{U} \to \mathbb{R}_+$ is continuous and bounded, and $\mathbb{U}$ is compact.
\item[(ii)] $\eta$ and $\eta^{{\cal T}_n,Q_n}$ are weak Feller (e.g., under Theorem \ref{Weak_tran_channel_thm_comb}).
\item[(iii)] $\X$ is compact, so that $(\PP(\X),W_1)$ is a compact metric space.
\item[(iv)] There exist $\beta^*\in(0,1)$ and constants $L, L_n<\infty$ such that $\|J^*_{\beta}(\eta)\|_{\Lip}\le L$ and $\|J^*_{\beta}(\eta^{{\cal T}_n,Q_n})\|_{\Lip}\le L_n$ for all $\beta\in[\beta^*,1)$.
\end{itemize}
Then canonical triplets $(h,\rho^*,\gamma^*)$ solving the ACOE (\ref{acoe}) exist for both belief-MDPs, with $h$ Lipschitz continuous ($\norm{h}_{\Lip}\le L$, respectively $L_n$), and $J^*_\infty(\eta)\equiv\rho^*$, $J^*_\infty(\eta^{{\cal T}_n,Q_n})\equiv\rho^*_n$ are constant. Moreover,
\[
\bigl|\rho^* - \rho^*_n\bigr| \leq L\, \dWW(\eta,\eta^{{\cal T}_n,Q_n}).
\]
\end{theorem}

\begin{remark}\label{rem:abelian}
Under Assumption \ref{main_assumption} or Assumption \ref{weakTtvQweakTtvQ2} with $K_2<1$, together with Assumption \ref{cost_assumption}, (\ref{LipBoundJ}) gives $\|J^*_{\beta}(\eta)\|_{\Lip} \leq K_1/(1-\beta K_2) \leq K_1/(1-K_2)$ for all $\beta\in(0,1)$, so (iv) holds with $L=K_1/(1-K_2)$; similarly $L_n = K_1/(1-K_2^{\T_n,Q_n})$ when the approximate model satisfies the same assumptions with $K_2^{\T_n,Q_n}<1$.
\end{remark}

\begin{corollary}\label{cont_average}
Suppose that $\U$ is compact, that $\eta$ and $\eta^{{\cal T}_n, Q_n}$ are weak Feller, and that Assumption~\ref{main_assumption} or Assumption~\ref{weakTtvQweakTtvQ2}, together with Assumption \ref{cost_assumption}, holds for both the true and approximate models. Assume also that $K_2 < 1$ and $K_2^{\T_n, Q_n} < 1$, where $K_2^{\T_n, Q_n}$ is the regularity coefficient of the approximate model as in Corollary \ref{discount_robust_twosided}.
The optimal cost functions satisfy:
    \[
    \left\Vert J^*_{\infty}(\eta) - J^*_{\infty}(\eta^{{\cal T}_n,Q_n}) \right\Vert_{\infty} \leq
    \frac{K_1}{1 - K_2} \cdot \frac{3D}{2} \left( d_{TV}(\T_n,\T) + d_{TV}(Q_n,Q) \right).
    \]
 Furthermore, under Assumption~\ref{channel_reg}, we have the refined bound:
    \[
    \left\Vert J^*_{\infty}(\eta) - J^*_{\infty}(\eta^{{\cal T}_n,Q_n}) \right\Vert_{\infty} \leq
    \frac{K_1}{1 - K_2} \left[\frac{3D}{2} \left( L_Q\, d_{W_1}(\T_n,\T) + d_{TV}(Q_n,Q) \right) + d_{W_1}(\T_n,\T)\right].
    \]
\end{corollary}

Since the optimal average cost is constant over initial beliefs under our assumptions, the prior does not enter: for all $\mu_0,\nu_0\in\PP(\X)$, $\bigl|J^*_\infty(c,\mu_0,\T,Q)-J^*_\infty(c,\nu_0,\T_n,Q_n)\bigr|=|\rho^*-\rho^*_n|$, which is bounded by Corollary \ref{cont_average}.

\subsection{Robustness of Policies under Discounted Cost}\label{sec:disc_robust}

We now bound the cost of applying $\gamma^*_{\eta^{{\cal T}_n,Q_n}}$ to the true model. Here we consider setting (I) of Section \ref{sec:prior_robust} with a correct prior, so that the policy is applied along the true filter $z_t$ and its cost is $J_\beta(\eta,\gamma^*_{\eta^{{\cal T}_n,Q_n}})$; an incorrect prior is treated at the end of this subsection.

The following is implied by \cite[Theorem 2.7, eq. (2.8)]{zhou2024robustness} together with \cite[Lemma 1.1]{zhou2024robustness}; see Appendix \ref{app:mdp_proofs} for a proof sketch. Since the prior is correct, the belief-MDP is a fully observed MDP with the correct initialization and the correct update, so the MDP results apply directly.
\begin{theorem}
\label{upperbound3proof}
Under the conditions of Theorem \ref{upperbound1proof}, we have
\begin{align*}&
\left\Vert J_{\beta}(\eta,\gamma^*_{\eta^{{\cal T}_n,Q_n}}) - J^*_{\beta}(\eta) \right\Vert_{\infty}
\leq \frac{2\beta}{(1-\beta)^{2}}  \left \Vert J^*_{\beta}(\eta)\right\Vert_{\Lip} \dWW(\eta,\eta^{{\cal T}_n,Q_n}).
\end{align*}
\end{theorem}

\begin{corollary}\label{discount_robust}
Suppose that $\U$ is compact, that $\eta$ and $\eta^{{\cal T}_n, Q_n}$ are weak Feller, and that Assumption~\ref{main_assumption} or Assumption~\ref{weakTtvQweakTtvQ2} holds together with Assumption \ref{cost_assumption}, with $\beta K_2 < 1$. Then:
\begin{align*}
&\|J_{\beta}(\eta,\gamma^*_{\eta^{{\cal T}_n,Q_n}})-J^*_{\beta}(\eta)\|_{\infty}\leq
\frac{2\beta K_1}{(1-\beta)^2(1-\beta K_2)}\frac{3D}{2}\bigl(d_{TV}(\T_n,\T)+d_{TV}(Q_n,Q)\bigr).
\end{align*}
Furthermore, under Assumption~\ref{channel_reg}, we have the refined bound:
\begin{align*}
& \|J_{\beta}(\eta,\gamma^*_{\eta^{{\cal T}_n,Q_n}})-J^*_{\beta}(\eta)\|_{\infty}\\
&\qquad \leq
\frac{2\beta K_1}{(1-\beta)^2(1-\beta K_2)}\left[\frac{3D}{2} \left( L_{Q}\,d_{W_1}(\T_n,\T) + d_{TV}(Q_n,Q)\right) + d_{W_1}(\T_n,\T)\right].\nonumber
\end{align*}
\end{corollary}

\begin{remark}[Strategic measures vs. belief-MDP reduction]
The above results build on the belief-MDP reduction and the regularity of the corresponding belief transition kernel. An alternative approach is to directly compare the strategic measures on the observation, state, and action processes of the perturbed and the original model. \cite[Theorem 2.3 and Theorem 3.1]{kara2020near} and \cite[Theorem 3.3]{devranby2025near} follow this approach for perturbations on the kernels. While the conditions in the papers above require total variation continuity in the states and actions, and therefore are more relaxed, their use for finite model approximations requires a tailored analysis via the product topology on measurements and, if applicable, priors. By contrast, the analysis via belief-MDP reduction only uses regularity of the reduced kernel, for which readily available results apply. For finite models, a related bound, also via total variation perturbations on kernels, and consistent with either of the approaches, is presented in \cite[Theorem 15.11]{Kri25}.
\end{remark}

Theorem \ref{upperbound3proof} addresses the case where the approximate value function may not be Lipschitz; a further refinement is possible if it is. This refinement avoids the term $(1-\beta)^2$ in the denominator and is thus also useful for the average cost analysis. For this, we first note that the argument of Theorem \ref{upperbound1proof} also applies to a fixed policy evaluated under the two filter kernels.

\begin{lemma}\label{lem:fixedpolicy}
Let $\X$ be compact, let $c$ be bounded and measurable, and let $\gamma:\PP(\X)\to\U$ be a measurable (deterministic stationary) policy. Then, for every $\beta\in(0,1)$,
\begin{itemize}
\item[(i)]
\begin{align*}
&\| J_{\beta}(\eta,\gamma) - J_{\beta}(\eta^{{\cal T}_n,Q_n},\gamma) \|_{\infty}\\
&\qquad \leq \frac{\beta}{1-\beta}\, \sup_{z\in\PP(\X)}\Bigl|\int J_{\beta}(\eta^{{\cal T}_n,Q_n},\gamma)(z_1)\,\bigl(\eta - \eta^{{\cal T}_n,Q_n}\bigr)(dz_1|z,\gamma(z))\Bigr| ;
\end{align*}
\item[(ii)] if, in addition, $J_{\beta}(\eta^{{\cal T}_n,Q_n},\gamma)$ is Lipschitz continuous on $(\PP(\X),W_1)$, then
\[
\| J_{\beta}(\eta,\gamma) - J_{\beta}(\eta^{{\cal T}_n,Q_n},\gamma) \|_{\infty} \leq \frac{\beta}{1-\beta}\,\|J_{\beta}(\eta^{{\cal T}_n,Q_n},\gamma)\|_{\Lip}\, \dWW(\eta,\eta^{{\cal T}_n,Q_n}),
\]
and the roles of $\eta$ and $\eta^{{\cal T}_n,Q_n}$ may be interchanged, both in the assumption and on the right-hand side.
\end{itemize}
In particular, for $\gamma=\gamma^*_{\eta^{{\cal T}_n,Q_n}}$ an optimal stationary policy of the approximate model, $J_{\beta}(\eta^{{\cal T}_n,Q_n},\gamma)=J^*_\beta(\eta^{{\cal T}_n,Q_n})$, and (ii) bounds the difference between the cost predicted by the approximate model and the cost incurred on the true model by
\[
\frac{\beta}{1-\beta}\,\|J^*_\beta(\eta^{{\cal T}_n,Q_n})\|_{\Lip}\, \dWW(\eta,\eta^{{\cal T}_n,Q_n}).
\]
\end{lemma}

No regularity of the policy $\gamma$ is needed in Lemma \ref{lem:fixedpolicy}; the Lipschitz constant in (ii) is that of the value of $\gamma$ under the approximate model, which for $\gamma=\gamma^*_{\eta^{{\cal T}_n,Q_n}}$ is given by (\ref{LipBoundJ}) applied to the approximate model.

The following is an adaptation of \cite[Theorem 2.7, eq. (2.9)]{zhou2024robustness} (see also \cite[Remark 2.9]{zhou2024robustness}) to our context; it follows by combining Theorem \ref{upperbound1proof} with Lemma \ref{lem:fixedpolicy}(ii), see Appendix \ref{app:mdp_proofs}.

\begin{theorem}\label{upperbound2proof}
Let $\X$ be compact. If $c$ is continuous and bounded, $\mathbb{U}$ is compact, $\eta,\eta^{{\cal T}_n,Q_n}$ are weak Feller, and both $J^*_{\beta}(\eta)$ and $J^*_{\beta}(\eta^{{\cal T}_n,Q_n})$ are Lipschitz continuous on $(\PP(\X),W_1)$, then
\begin{align*}
&\left\Vert J_{\beta}(\eta,\gamma^*_{\eta^{{\cal T}_n,Q_n}}) - J^*_{\beta}(\eta) \right\Vert_{\infty}
\leq
\frac{\beta}{1-\beta}\left( \|J^*_{\beta}(\eta)\|_{\Lip} + \|J^*_{\beta}(\eta^{{\cal T}_n,Q_n})\|_{\Lip} \right)  \dWW(\eta,\eta^{{\cal T}_n,Q_n}).
\end{align*}
\end{theorem}

Since the proof of Theorem \ref{upperbound1proof} is symmetric in the two models,
\begin{align*}
&\left\Vert J^*_{\beta}(\eta)-J^*_{\beta}(\eta^{{\cal T}_n,Q_n})\right\Vert_{\infty}\\
&\qquad\le\frac{\beta}{1-\beta}\min\Bigl\{\|J^*_{\beta}(\eta)\|_{\Lip},\ \|J^*_{\beta}(\eta^{{\cal T}_n,Q_n})\|_{\Lip}\Bigr\}\,\dWW(\eta,\eta^{{\cal T}_n,Q_n})
\end{align*}
whenever one of the two value functions is Lipschitz. Hence the constant in Theorem \ref{upperbound2proof} can be replaced by
\[
\|J^*_{\beta}(\eta^{{\cal T}_n,Q_n})\|_{\Lip}+\min\Bigl\{\|J^*_{\beta}(\eta)\|_{\Lip},\ \|J^*_{\beta}(\eta^{{\cal T}_n,Q_n})\|_{\Lip}\Bigr\};
\]
in particular, the bound $\frac{2\beta}{1-\beta}\|J^*_{\beta}(\eta^{{\cal T}_n,Q_n})\|_{\Lip}\,\dWW(\eta,\eta^{{\cal T}_n,Q_n})$ requires the Lipschitz continuity of the approximate value function only.

\begin{corollary}\label{discount_robust_twosided}
Suppose that $\U$ is compact, that $\eta$ and $\eta^{{\cal T}_n, Q_n}$ are weak Feller, and that Assumption \ref{main_assumption} (or Assumption \ref{weakTtvQweakTtvQ2}) together with Assumption \ref{cost_assumption} holds for both the true and the approximate models, with $\beta K_2 < 1$ and $\beta K_2^{\T_n, Q_n} < 1$, where $K_2^{\T_n, Q_n}$ denotes the regularity coefficient of Remark \ref{rem:tradeoff} computed for the approximate model (i.e., $\alpha_n D(3-2\delta(Q_n))/2$ under Assumption \ref{main_assumption}, or $\theta_n + 3\theta_n L_Q D/2$ under Assumption \ref{weakTtvQweakTtvQ2}). Then, we have:
\begin{align*}
&\left\Vert J_{\beta}(\eta,\gamma^*_{\eta^{{\cal T}_n,Q_n}}) - J^*_{\beta}(\eta) \right\Vert_{\infty}\\
&\qquad\leq \frac{\beta}{1-\beta} \left(\frac{K_1}{1 - \beta K_2}
+\frac{K_1}{1 - \beta K_2^{\T_n, Q_n}}\right) \frac{3D}{2} \left(  d_{TV}(\T_n,\T) + d_{TV}(Q_n,Q)\right).\nonumber
\end{align*}
Furthermore, under Assumption~\ref{channel_reg}, we have the refined bound:
\begin{align*}
& \|J_{\beta}(\eta,\gamma^*_{\eta^{{\cal T}_n,Q_n}})-J^*_{\beta}(\eta)\|_{\infty}\\
&\leq \frac{\beta}{1-\beta} \left(\frac{K_1}{1 - \beta K_2}
+\frac{K_1}{1 - \beta K_2^{\T_n, Q_n}}\right)\left[\frac{3D}{2} \left( L_{Q}\,d_{W_1}(\T_n,\T) + d_{TV}(Q_n,Q)\right) + d_{W_1}(\T_n,\T)\right].\nonumber
\end{align*}
\end{corollary}

\medskip
The following lemma is used to bound the performance loss in both settings of Section \ref{sec:prior_robust}. It bounds the cost of applying $\gamma^*_{\eta^{{\cal T}_n,Q_n}}$ to any belief process computed from the observations in terms of the model error and the distance between that belief process and the true filter. It is applied with $\hat\pi=\pi$ below and with $\hat\pi=\pi^n$ in Section \ref{sec:misspecified}.

\begin{lemma}[Discounted cost performance difference]\label{lem:perfdiff}
Let $\{\hat\pi_t\}$ be a $\PP(\X)$-valued process such that $\hat\pi_t$ is a function of $Y_{[0,t]}$ and $U_{[0,t-1]}$, computed with either the true or the approximate model and prior, and let $U_t=\gamma^*_{\eta^{{\cal T}_n,Q_n}}(\hat\pi_t)$. Here $J_\beta(c,\mu_0,{\cal T},Q,\gamma^*_{\eta^{{\cal T}_n,Q_n}}\circ\hat\pi)$ denotes the cost of this policy, $z_t$ is the true filter, computed with the true model, the true prior $\mu_0$ and the actions $U_t$.
\begin{itemize}
\item[(i)] Under the conditions of Corollary \ref{cor:prior_disc}, with $L:=K_1/(1-\beta K_2)$,
\begin{align*}
0&\le J_{\beta}\bigl(c,\mu_0,{\cal T},Q,\gamma^*_{\eta^{{\cal T}_n,Q_n}}\circ\hat\pi\bigr) - J^*_{\beta}\bigl(c,\mu_0,{\cal T},Q\bigr)\\
&\le \frac{2\beta L}{(1-\beta)^2}\,\Delta_n + 2L\sum_{t=0}^{\infty}\beta^t\, E^{\mu_0}\bigl[W_1(z_t,\hat\pi_t)\bigr].
\end{align*}
\item[(ii)] If, in addition, the conditions of Corollary \ref{discount_robust_twosided} hold, then, with $L$ as in (i) and $L_n:=K_1/(1-\beta K_2^{\T_n,Q_n})$,
\begin{align*}
&J_{\beta}\bigl(c,\mu_0,{\cal T},Q,\gamma^*_{\eta^{{\cal T}_n,Q_n}}\circ\hat\pi\bigr) - J^*_{\beta}\bigl(c,\mu_0,{\cal T},Q\bigr)\\
&\qquad\le \frac{\beta}{1-\beta}\,(L+L_n)\,\Delta_n + \bigl(K_1+(1+\beta K_2)L_n\bigr)\sum_{t=0}^{\infty}\beta^t\, E^{\mu_0}\bigl[W_1(z_t,\hat\pi_t)\bigr].
\end{align*}
\end{itemize}
\end{lemma}

\begin{proof}
(i) Write $v:=J^*_\beta(\eta)$ and $v_n:=J^*_\beta(\eta^{{\cal T}_n,Q_n})$ for the value functions of the true and the approximate belief-MDPs, both functions on $\PP(\X)$, and $u_t:=\gamma^*_{\eta^{{\cal T}_n,Q_n}}(\hat\pi_t)$. Since $u_t$ depends only on the observations, the true filter $z_t$ evolves according to $\eta(\cdot|z_t,u_t)$. Hence $J_\beta(c,\mu_0,{\cal T},Q,\gamma^*_{\eta^{{\cal T}_n,Q_n}}\circ\hat\pi) = E^{\mu_0}[\sum_t\beta^t\tilde c(z_t,u_t)]$ and, since the first action is chosen after $Y_0$, $J^*_\beta(c,\mu_0,{\cal T},Q) = E^{\mu_0}[v(z_0)]$. Therefore
\begin{align*}
J_\beta(c,\mu_0,{\cal T},Q,\gamma^*_{\eta^{{\cal T}_n,Q_n}}\circ\hat\pi) - J^*_\beta(c,\mu_0,{\cal T},Q)
&= E^{\mu_0}\Bigl[\sum_{t=0}^\infty \beta^t\, \tilde c(z_t,u_t)\Bigr] - E^{\mu_0}\bigl[v(z_0)\bigr]\\
&= E^{\mu_0}\Bigl[\sum_{t=0}^\infty \beta^t \bigl(\tilde c(z_t,u_t) + \beta\, v(z_{t+1}) - v(z_t)\bigr)\Bigr]\\
&= \sum_{t=0}^\infty \beta^t E^{\mu_0}\bigl[r(z_t,\hat\pi_t)\bigr],
\end{align*}
where, with $u=\gamma^*_{\eta^{{\cal T}_n,Q_n}}(\pi)$,
\[
r(z,\pi) := \tilde c(z,u) + \beta\int v(z_1)\,\eta(dz_1|z,u) - v(z).
\]
The third equality holds since $E^{\mu_0}[v(z_{t+1})\,|\,Y_{[0,t]}] = \int v(z_1)\,\eta(dz_1|z_t,u_t)$.

For fixed $u$, the functions $\tilde c(\cdot,u)$, $z\mapsto\beta\int v(z_1)\,\eta(dz_1|z,u)$ and $v$ are $K_1$-, $\beta K_2L$- and $L$-Lipschitz with respect to $W_1$ (Assumption \ref{cost_assumption}, Theorems \ref{ergodicity}--\ref{WassersteinCont2}, and Theorem \ref{ExistenceDiscount}, eq. (\ref{LipBoundJ})), and $K_1+\beta K_2 L = L$. Thus $r(z,\pi)\le r(\pi,\pi) + 2L\,W_1(z,\pi)$, and
\begin{align*}
&J_\beta(c,\mu_0,{\cal T},Q,\gamma^*_{\eta^{{\cal T}_n,Q_n}}\circ\hat\pi) - J^*_\beta(c,\mu_0,{\cal T},Q)
\\
&\qquad\le \underbrace{\sum_{t=0}^\infty \beta^t E^{\mu_0}\bigl[r(\hat\pi_t,\hat\pi_t)\bigr]}_{\text{model error}}
+ 2L \underbrace{\sum_{t=0}^\infty \beta^t E^{\mu_0}\bigl[W_1(z_t,\hat\pi_t)\bigr]}_{\text{belief error}} .
\end{align*}

We now bound the model error. When $\hat\pi_t=z_t$, the belief error vanishes and the model error is $J_\beta(\eta,\gamma^*_{\eta^{{\cal T}_n,Q_n}})(\mu_0) - J^*_\beta(\eta)(\mu_0)$, which is the left-hand side of Theorem \ref{upperbound3proof} at $\mu_0$. In general the same bound holds, because the one-step estimate behind Theorem \ref{upperbound3proof} holds for every belief: for every $\pi\in\PP(\X)$, with $u=\gamma^*_{\eta^{{\cal T}_n,Q_n}}(\pi)$, since $\gamma^*_{\eta^{{\cal T}_n,Q_n}}$ is optimal for the approximate model, the discounted cost optimality equation of the approximate model at $\pi$ reads $v_n(\pi)=\tilde c(\pi,u)+\beta\int v_n(z_1)\,\eta^{{\cal T}_n,Q_n}(dz_1|\pi,u)$, so that
\begin{align*}
r(\pi,\pi) &= v_n(\pi)-v(\pi) + \beta\Bigl(\int v(z_1)\,\eta(dz_1|\pi,u)-\int v(z_1)\,\eta^{{\cal T}_n,Q_n}(dz_1|\pi,u)\Bigr)\\
&\qquad + \beta\int \bigl(v(z_1)-v_n(z_1)\bigr)\,\eta^{{\cal T}_n,Q_n}(dz_1|\pi,u)\\
&\le (1+\beta)\norm{v-v_n}_\infty + \beta L\,\Delta_n \le \frac{2\beta L}{1-\beta}\,\Delta_n,
\end{align*}
by Theorem \ref{upperbound1proof}. Summing over $t$,
\[
\sum_{t=0}^\infty \beta^t E^{\mu_0}\bigl[r(\hat\pi_t,\hat\pi_t)\bigr] \le \frac{2\beta}{(1-\beta)^2}\,L\,\Delta_n ,
\]
which is the bound of Theorem \ref{upperbound3proof}. The lower bound holds since $\gamma^*_{\eta^{{\cal T}_n,Q_n}}\circ\hat\pi$ is admissible. This proves (i).

(ii) Under the conditions of Corollary \ref{discount_robust_twosided}, $v_n$ is $L_n$-Lipschitz by (\ref{LipBoundJ}) applied to the approximate model. Repeat the argument of (i) with $v_n$ in place of $v$: with $u=\gamma^*_{\eta^{{\cal T}_n,Q_n}}(\pi)$ and
\[
r_n(z,\pi) := \tilde c(z,u) + \beta\int v_n(z_1)\,\eta(dz_1|z,u) - v_n(z),
\]
the telescoping identity gives $J_\beta(c,\mu_0,{\cal T},Q,\gamma^*_{\eta^{{\cal T}_n,Q_n}}\circ\hat\pi) - E^{\mu_0}[v_n(z_0)] = \sum_t\beta^tE^{\mu_0}[r_n(z_t,\hat\pi_t)]$, and $r_n(z,\pi)\le r_n(\pi,\pi)+(K_1+\beta K_2L_n+L_n)W_1(z,\pi)$. By the optimality equation of the approximate model at $\pi$ and (\ref{eq:app_lip_dw1}),
\[
r_n(\pi,\pi)=\beta\int v_n(z_1)\,\bigl(\eta-\eta^{{\cal T}_n,Q_n}\bigr)(dz_1|\pi,u)\le\beta L_n\Delta_n .
\]
Finally, $E^{\mu_0}[v_n(z_0)]-E^{\mu_0}[v(z_0)]\le\norm{v_n-v}_\infty\le\frac{\beta}{1-\beta}L\Delta_n$ by Theorem \ref{upperbound1proof}. Summing over $t$ proves (ii).
\end{proof}

For an incorrect prior the belief error is the effect of the prior only, and it is controlled by the following filter stability estimate in $W_1$. The one-step estimate is that of Theorems \ref{ergodicity} and \ref{WassersteinCont2}, applied to two filters driven by the same observations, and the proof follows that of \cite[Theorems 3.3 and 4.1]{mcdonald2020exponential} (see also \cite[Theorem 3.3]{demirci2023average}), where the same recursion is obtained in total variation; see Appendix \ref{app:filter_mismatch}.
\begin{lemma}[Filter stability in $W_1$]\label{lem:filter_stability}
Let Assumption \ref{main_assumption} or Assumption \ref{weakTtvQweakTtvQ2} hold for the true model, with $K_2$ as in Remark \ref{rem:tradeoff}. Let $\mu_0,\nu_0\in\PP(\X)$, let $\{U_t\}$ be any admissible policy, let $\pi_t$ be the filter of the true model with prior $\nu_0$, as in (\ref{eq:true_filter_nu}), and let $z_t$ be the true filter, computed with the true model and the prior $\mu_0$ under the same actions $U_t$. Then, for all $t\ge0$,
\begin{align*}
E^{\mu_0}\bigl[W_1(z_{t+1},\pi_{t+1})\bigr]&\le K_2\,E^{\mu_0}\bigl[W_1(z_t,\pi_t)\bigr],\\
E^{\mu_0}\bigl[W_1(z_0,\pi_0)\bigr]&\le \frac{D}{2}\bigl(2-\delta(Q)\bigr)\normt{\mu_0-\nu_0}.
\end{align*}
\end{lemma}

\begin{theorem}[Policy designed for the approximate model, applied to the true model]\label{thm:transfer}
Suppose that $\U$ is compact, that $\eta$ and $\eta^{{\cal T}_n,Q_n}$ are weak Feller, and that Assumption \ref{main_assumption} or Assumption \ref{weakTtvQweakTtvQ2} holds for the true model together with Assumption \ref{cost_assumption}, with $\beta K_2 < 1$. Let $\pi_t$ be the filter of the true model with prior $\nu_0$ under the policy $U_t=\gamma^*_{\eta^{{\cal T}_n,Q_n}}(\pi_t)$, as in (\ref{eq:true_filter_nu}), and let $L:=K_1/(1-\beta K_2)$ and $L_n:=K_1/(1-\beta K_2^{\T_n,Q_n})$ be as in Lemma \ref{lem:perfdiff}. Then
\begin{align*}
&J_{\beta}\bigl(c,\mu_0,{\cal T},Q,\gamma^*_{\eta^{{\cal T}_n,Q_n}}\circ\pi\bigr) - J^*_{\beta}\bigl(c,\mu_0,{\cal T},Q\bigr)\\
&\qquad\le \frac{2\beta L}{(1-\beta)^2}\,\Delta_n + \frac{L D\,(2-\delta(Q))}{1-\beta K_2}\,\normt{\mu_0-\nu_0}.
\end{align*}
For $\nu_0=\mu_0$ this reduces to Theorem \ref{upperbound3proof}. If, in addition, the conditions of Corollary \ref{discount_robust_twosided} hold, then
\begin{align*}
&J_{\beta}\bigl(c,\mu_0,{\cal T},Q,\gamma^*_{\eta^{{\cal T}_n,Q_n}}\circ\pi\bigr) - J^*_{\beta}\bigl(c,\mu_0,{\cal T},Q\bigr)\\
&\qquad\le \frac{\beta}{1-\beta}\,(L+L_n)\,\Delta_n + \frac{\bigl(K_1+(1+\beta K_2)L_n\bigr) D\,(2-\delta(Q))}{2(1-\beta K_2)}\,\normt{\mu_0-\nu_0}.
\end{align*}
If $\mu_0\ll\nu_0$ and
\[
\bar\alpha := \bigl(1-\tilde\delta({\cal T})\bigr)\bigl(2-\delta(Q)\bigr) < 1,
\qquad \tilde\delta({\cal T}) := \inf_{u \in \U}\, \delta\bigl({\cal T}(\cdot \mid \cdot\,,u)\bigr),
\]
then $E^{\mu_0}\bigl[\normt{z_t-\pi_t}\bigr] \le 2\bar\alpha^t$ (see \cite{mcdonald2020exponential} and \cite[Theorem 3.3]{demirci2023average}), and the last terms above can be replaced by $\frac{2LD}{1-\beta\bar\alpha}$ and $\frac{(K_1+(1+\beta K_2)L_n)D}{1-\beta\bar\alpha}$, respectively.
\end{theorem}
\begin{proof}
By Lemma \ref{lem:filter_stability}, $E^{\mu_0}[W_1(z_t,\pi_t)]\le K_2^t\,\frac D2(2-\delta(Q))\normt{\mu_0-\nu_0}$, so
\[
\sum_{t=0}^\infty\beta^tE^{\mu_0}\bigl[W_1(z_t,\pi_t)\bigr]\le\frac{D\,(2-\delta(Q))}{2(1-\beta K_2)}\,\normt{\mu_0-\nu_0},
\]
and Lemma \ref{lem:perfdiff}(i) and (ii) give the two bounds. For the last claim use instead $E^{\mu_0}[W_1(z_t,\pi_t)]\le\frac D2E^{\mu_0}[\normt{z_t-\pi_t}]\le D\bar\alpha^t$ (Remark \ref{rem:metrics}(ii)).
\end{proof}

\subsection{Robustness of Policies under Average Cost}\label{sec:avg_robust}

The following policy bound follows from \cite[Theorem 2.9]{zhou2024robustness} applied on $(\PP(\X),W_1)$; a proof sketch is given in Appendix \ref{app:mdp_proofs}. No ergodicity condition on the belief process under the approximate policy is needed; the proof uses the comparison lemma \cite[Lemma 5.2.5(a)]{HernandezLermaMCP} instead.

\begin{theorem}\label{thm:avg_policy}
Under the conditions of Theorem \ref{vanishingproof1}, let $\gamma^*_{\eta^{{\cal T}_n,Q_n}}$ denote the stationary policy attaining the infimum in the ACOE of the approximate model. Then
\[
\left\Vert J_{\infty}(\eta,\gamma^*_{\eta^{{\cal T}_n,Q_n}}) - J^*_{\infty}(\eta) \right\Vert_{\infty} \leq  \left(L + L_n \right)\dWW(\eta,\eta^{{\cal T}_n,Q_n}).
\]
\end{theorem}

\cite[Theorem 2.9]{zhou2024robustness} is stated with the constant $2\max\{L,L_n\}$; the constant $L+L_n$ follows from the same proof, as in \cite[Corollary 2.3]{zhou2024robustness}.

\begin{corollary}\label{robust_average}
Under the same conditions as in Corollary~\ref{cont_average}, let $\gamma^*_{\eta^{{\cal T}_n,Q_n}}$ be the stationary policy attaining the infimum in the ACOE of the approximate model. Then the performance loss due to using the approximate optimal policy in the original system is bounded as follows:
\begin{align*}
&\left\Vert J_{\infty}(\eta,\gamma^*_{\eta^{{\cal T}_n,Q_n}}) - J^*_{\infty}(\eta) \right\Vert_{\infty}\\
&\qquad\leq \left( \frac{K_1}{1 - K_2}+\frac{K_1}{1 - K_2^{\T_n, Q_n}} \right)  \frac{3D}{2} \left(  d_{TV}(\T_n,\T) + d_{TV}(Q_n,Q)\right).\nonumber
\end{align*}
Furthermore, under Assumption~\ref{channel_reg}, we have the refined bound:
\begin{align*}
&\left\Vert J_{\infty}(\eta,\gamma^*_{\eta^{{\cal T}_n,Q_n}}) - J^*_{\infty}(\eta) \right\Vert_{\infty}\\
&\qquad\leq \left( \frac{K_1}{1 - K_2}+\frac{K_1}{1 - K_2^{\T_n, Q_n}} \right) \left[ \frac{3D}{2} \left( L_Q\, d_{W_1}(\T_n,\T) + d_{TV}(Q_n,Q)\right) + d_{W_1}(\T_n,\T)\right].\nonumber
\end{align*}
\end{corollary}

\medskip
\noindent{\bf Incorrect prior.} The following is the average cost counterpart of Lemma \ref{lem:perfdiff}.

\begin{lemma}[Performance difference, average cost]\label{lem:perfdiff_avg}
Let $\{\hat\pi_t\}$ be a $\PP(\X)$-valued process such that $\hat\pi_t$ is a function of $Y_{[0,t]}$ and $U_{[0,t-1]}$, and let $U_t=\gamma^*_{\eta^{{\cal T}_n,Q_n}}(\hat\pi_t)$. Under the conditions of Corollary \ref{robust_average}, with $L:=K_1/(1-K_2)$, $L_n:=K_1/(1-K_2^{\T_n,Q_n})$, and $\gamma^*_{\eta^{{\cal T}_n,Q_n}}$ the stationary policy attaining the infimum in the ACOE of the approximate model,
\begin{align*}
0&\le J_\infty\bigl(c,\mu_0,{\cal T},Q,\gamma^*_{\eta^{{\cal T}_n,Q_n}}\circ\hat\pi\bigr)-J^*_\infty\bigl(c,\mu_0,{\cal T},Q\bigr)\\
&\le (L+L_n)\,\Delta_n+\bigl(K_1+(1+K_2)L_n\bigr)\limsup_{N\to\infty}\frac1N\sum_{t=0}^{N-1}E^{\mu_0}\bigl[W_1(z_t,\hat\pi_t)\bigr].
\end{align*}
\end{lemma}

\begin{proof}
Let $(h_n,\rho^*_n)$ solve the ACOE of the approximate model, with $\norm{h_n}_{\Lip}\le L_n$, and let $u=\gamma^*_{\eta^{{\cal T}_n,Q_n}}(\pi)$. For $z,\pi\in\PP(\X)$, since $\tilde c(\cdot,u)$, $z\mapsto\int h_n(z_1)\,\eta(dz_1|z,u)$ and $h_n$ are $K_1$-, $K_2L_n$- and $L_n$-Lipschitz, and since $u$ attains the infimum in the ACOE of the approximate model at $\pi$,
\begin{align*}
\tilde c(z,u)+\int h_n(z_1)\,\eta(dz_1|z,u)-h_n(z)
&\le \tilde c(\pi,u)+\int h_n(z_1)\,\eta(dz_1|\pi,u)-h_n(\pi)\\
&\qquad+\bigl(K_1+(1+K_2)L_n\bigr)W_1(z,\pi)\\
&= \rho^*_n+\int h_n(z_1)\,\bigl(\eta-\eta^{{\cal T}_n,Q_n}\bigr)(dz_1|\pi,u)\\
&\qquad+\bigl(K_1+(1+K_2)L_n\bigr)W_1(z,\pi)\\
&\le \rho^*_n + L_n\Delta_n + \bigl(K_1+(1+K_2)L_n\bigr)W_1(z,\pi).
\end{align*}
Apply this at $(z_t,\hat\pi_t)$, take expectations and sum over $t<N$. Since $U_t$ is a function of the observations, $E^{\mu_0}[h_n(z_{t+1})\,|\,Y_{[0,t]}]=\int h_n(z_1)\,\eta(dz_1|z_t,U_t)$. Hence the terms $\int h_n(z_1)\,\eta(dz_1|z_t,U_t)-h_n(z_t)$ telescope to $E^{\mu_0}[h_n(z_N)]-E^{\mu_0}[h_n(z_0)]$, which is bounded since $h_n$ is. Dividing by $N$ and letting $N\to\infty$,
\begin{align*}
&J_\infty\bigl(c,\mu_0,{\cal T},Q,\gamma^*_{\eta^{{\cal T}_n,Q_n}}\circ\hat\pi\bigr)\\
&\qquad\le\rho^*_n+L_n\Delta_n+\bigl(K_1+(1+K_2)L_n\bigr)\limsup_{N\to\infty}\frac1N\sum_{t=0}^{N-1}E^{\mu_0}\bigl[W_1(z_t,\hat\pi_t)\bigr].
\end{align*}
Finally, $\rho^*_n\le\rho^*+L\Delta_n$ by Theorem \ref{vanishingproof1}, and the left-hand side is at least $J^*_\infty(c,\mu_0,\T,Q)=\rho^*$ by Theorem \ref{mainEmre}.
\end{proof}

\begin{corollary}[Average cost]\label{cor:transfer_avg}
Under the conditions of Corollary \ref{robust_average}, let $\gamma^*_{\eta^{{\cal T}_n,Q_n}}$ be the stationary policy attaining the infimum in the ACOE of the approximate model, and let $\pi_t$ be as in (\ref{eq:true_filter_nu}). Then, for all $\mu_0,\nu_0\in\PP(\X)$,
\[
0\le J_\infty\bigl(c,\mu_0,{\cal T},Q,\gamma^*_{\eta^{{\cal T}_n,Q_n}}\circ\pi\bigr)-J^*_\infty\bigl(c,\mu_0,{\cal T},Q\bigr)\le \Bigl(\frac{K_1}{1-K_2}+\frac{K_1}{1-K_2^{\T_n,Q_n}}\Bigr)\Delta_n .
\]
The prior does not enter, since the optimal average cost is constant over initial beliefs.
\end{corollary}
\begin{proof}
Since $K_2<1$, Lemma \ref{lem:filter_stability} gives $\frac1N\sum_{t<N}E^{\mu_0}[W_1(z_t,\pi_t)]\to0$, and Lemma \ref{lem:perfdiff_avg} applies.
\end{proof}
Sufficient conditions for $K_2<1$, and examples showing that the coefficient $K_2$ of Theorem \ref{ergodicity} is attained, are collected in Section \ref{sub:ex_K2}.

\subsection{Misspecified Model: Controller Running the Filter of the Approximate Model}\label{sec:misspecified}

In setting (II) of Section \ref{sec:prior_robust} the belief error has two sources: the prior, and the fact that each update of $\pi^n_t$ uses the wrong Bayes map. The latter is measured, per step, by
\begin{align}
\epsilon_n &:= \sup_{z\in\PP(\X),\,u\in\U}\int_\Y W_1\bigl(F(z,u,y),F^{\T_n,Q_n}(z,u,y)\bigr)\,H(dy|z,u),\label{eq:eps_def}\\
\epsilon_{0,n} &:= \sup_{z\in\PP(\X)}\int_\Y W_1\bigl(\Phi^{Q}(z,y),\Phi^{Q_n}(z,y)\bigr)\,P^{Q}_z(dy),\qquad P^{Q}_z(dy):=\int_\X Q(dy|x)\,z(dx),\nonumber
\end{align}
where the averages are taken under the true observation laws. Both are bounded by the proofs of Theorems \ref{W1main_1} and \ref{W1main_2}; see Appendix \ref{app:filter_mismatch} for the proofs of the following two lemmas.

\begin{lemma}\label{lem:mismatch_bounds}
Let $\X$ be compact with diameter $D$. Then
\[
\epsilon_n\le D\,d_{TV}(\T_n,\T)+\frac{3D}{2}\,d_{TV}(Q_n,Q),\qquad \epsilon_{0,n}\le D\,d_{TV}(Q_n,Q),
\]
and, under Assumption \ref{channel_reg},
\[
\epsilon_n\le (1+DL_Q)\,d_{W_1}(\T_n,\T)+\frac{3D}{2}\,d_{TV}(Q_n,Q).
\]
In particular, $\epsilon_n,\epsilon_{0,n}\le\bar\Delta_n$, where $\bar\Delta_n$ is the bound given by (\ref{eq:Deltabar_tv}) and (\ref{eq:Deltabar_w1}).
\end{lemma}

\begin{lemma}[Filter mismatch under a misspecified model]\label{lem:filter_mismatch}
Let Assumption \ref{main_assumption} or Assumption \ref{weakTtvQweakTtvQ2} hold for the true model, with $K_2$ as in Remark \ref{rem:tradeoff}. Let $\mu_0,\nu_0\in\PP(\X)$, let $\{U_t\}$ be any admissible policy, let $\pi^n_t$ be the filter of the model $(\T_n,Q_n)$ with prior $\nu_0$, as in (\ref{eq:approx_filter}), and let $z_t$ be the true filter, computed with the true model and the prior $\mu_0$ under the same actions $U_t$. Then, for all $t\ge 0$,
\begin{align*}
E^{\mu_0}\bigl[W_1(z_{t+1},\pi^n_{t+1})\bigr]&\le K_2\,E^{\mu_0}\bigl[W_1(z_t,\pi^n_t)\bigr]+\epsilon_n,\\
E^{\mu_0}\bigl[W_1(z_0,\pi^n_0)\bigr]&\le \frac{D}{2}\bigl(3-2\delta(Q)\bigr)\normt{\mu_0-\nu_0}+\epsilon_{0,n}.
\end{align*}
Consequently, if $\beta K_2<1$,
\begin{align}\label{eq:filter_mismatch_sum}
\sum_{t=0}^\infty\beta^t E^{\mu_0}\bigl[W_1(z_t,\pi^n_t)\bigr]\le\frac{1}{1-\beta K_2}\Bigl(\frac{D}{2}\bigl(3-2\delta(Q)\bigr)\normt{\mu_0-\nu_0}+\epsilon_{0,n}+\frac{\beta}{1-\beta}\,\epsilon_n\Bigr),
\end{align}
and, if $K_2<1$,
\begin{align}\label{eq:filter_mismatch_avg}
\limsup_{N\to\infty}\frac1N\sum_{t=0}^{N-1}E^{\mu_0}\bigl[W_1(z_t,\pi^n_t)\bigr]\le\frac{\epsilon_n}{1-K_2}.
\end{align}
\end{lemma}

Lemmas \ref{lem:filter_stability} and \ref{lem:filter_mismatch} do not use the form of the policy. See \cite{mcdonald2020exponential} for filter stability in total variation, and \cite{chigansky2007model,cohen2023exponential} for filters computed with a misspecified model.

\begin{theorem}[Controller assuming the approximate model]\label{thm:misspecified}
Suppose that $\U$ is compact, that $\eta$ and $\eta^{{\cal T}_n,Q_n}$ are weak Feller, and that Assumption \ref{main_assumption} or Assumption \ref{weakTtvQweakTtvQ2} holds for the true model together with Assumption \ref{cost_assumption}, with $\beta K_2 < 1$. Let $\pi^n_t$ be the filter of the approximate model with prior $\nu_0$ under the policy $U_t=\gamma^*_{\eta^{{\cal T}_n,Q_n}}(\pi^n_t)$, as in (\ref{eq:approx_filter}), and let $L:=K_1/(1-\beta K_2)$ and $L_n:=K_1/(1-\beta K_2^{\T_n,Q_n})$ be as in Lemma \ref{lem:perfdiff}. Then
\begin{align*}
&J_{\beta}\bigl(c,\mu_0,{\cal T},Q,\gamma^*_{\eta^{{\cal T}_n,Q_n}}\circ\pi^n\bigr) - J^*_{\beta}\bigl(c,\mu_0,{\cal T},Q\bigr)\\
&\qquad\le \frac{2\beta L}{(1-\beta)^2}\,\Delta_n + \frac{2L}{1-\beta K_2}\Bigl(\epsilon_{0,n}+\frac{\beta}{1-\beta}\,\epsilon_n\Bigr)+\frac{LD\,(3-2\delta(Q))}{1-\beta K_2}\,\normt{\mu_0-\nu_0}\\
&\qquad\le C_\beta\,\bar\Delta_n+\frac{LD\,(3-2\delta(Q))}{1-\beta K_2}\,\normt{\mu_0-\nu_0},
\end{align*}
where
\[
C_\beta:=2L\Bigl[\frac{\beta}{(1-\beta)^2}+\frac{1}{(1-\beta)(1-\beta K_2)}\Bigr].
\]
If, in addition, the conditions of Corollary \ref{discount_robust_twosided} hold, then
\begin{align*}
&J_{\beta}\bigl(c,\mu_0,{\cal T},Q,\gamma^*_{\eta^{{\cal T}_n,Q_n}}\circ\pi^n\bigr) - J^*_{\beta}\bigl(c,\mu_0,{\cal T},Q\bigr)\\
&\qquad\le \frac{\beta}{1-\beta}\,(L+L_n)\,\Delta_n\\
&\qquad\qquad + \frac{K_1+(1+\beta K_2)L_n}{1-\beta K_2}\Bigl(\frac{D}{2}\bigl(3-2\delta(Q)\bigr)\normt{\mu_0-\nu_0}+\epsilon_{0,n}+\frac{\beta}{1-\beta}\,\epsilon_n\Bigr).
\end{align*}
\end{theorem}
\begin{proof}
The first bound is Lemma \ref{lem:perfdiff}(i) with $\hat\pi_t=\pi^n_t$ and (\ref{eq:filter_mismatch_sum}). For the second, use $\Delta_n\le\bar\Delta_n$ and Lemma \ref{lem:mismatch_bounds}. The last bound is Lemma \ref{lem:perfdiff}(ii) with (\ref{eq:filter_mismatch_sum}).
\end{proof}

\begin{corollary}[Average cost]\label{cor:misspecified_avg}
Under the conditions of Corollary \ref{robust_average}, let $\gamma^*_{\eta^{{\cal T}_n,Q_n}}$ be the stationary policy attaining the infimum in the ACOE of the approximate model, and let $\pi^n_t$ be as in (\ref{eq:approx_filter}). Then, for all $\mu_0,\nu_0\in\PP(\X)$, with $L=\frac{K_1}{1-K_2}$ and $L_n=\frac{K_1}{1-K_2^{\T_n,Q_n}}$,
\begin{align*}
0&\le J_\infty\bigl(c,\mu_0,{\cal T},Q,\gamma^*_{\eta^{{\cal T}_n,Q_n}}\circ\pi^n\bigr)-J^*_\infty\bigl(c,\mu_0,{\cal T},Q\bigr)\\
&\le (L+L_n)\,\Delta_n+\bigl(K_1+(1+K_2)L_n\bigr)\frac{\epsilon_n}{1-K_2}
\le \Bigl[L+L_n+\frac{K_1+(1+K_2)L_n}{1-K_2}\Bigr]\bar\Delta_n .
\end{align*}
\end{corollary}
\begin{proof}
Lemma \ref{lem:perfdiff_avg} with $\hat\pi_t=\pi^n_t$, (\ref{eq:filter_mismatch_avg}) and Lemma \ref{lem:mismatch_bounds}.
\end{proof}

Theorems \ref{thm:transfer} and \ref{thm:misspecified} and Corollaries \ref{cor:transfer_avg} and \ref{cor:misspecified_avg} answer the two questions of Section \ref{sec:prior_robust}: in both settings the performance loss is bounded, with explicit constants, by the model error $\Delta_n\le\bar\Delta_n$ and the prior mismatch $\normt{\mu_0-\nu_0}$. The two settings differ only in the belief error. In setting (I) it is the effect of the prior and decays geometrically (Lemma \ref{lem:filter_stability}). In setting (II) each update adds the error of the wrong Bayes map, which contributes additional terms in $\epsilon_{0,n}$ and $\epsilon_n$. Since $\epsilon_n,\epsilon_{0,n}\le\bar\Delta_n$ by Lemma \ref{lem:mismatch_bounds}, these additional terms are of the same order as the model error, and all bounds are explicit in $d_{TV}(\T_n,\T)$, $d_{TV}(Q_n,Q)$ (or $d_{W_1}(\T_n,\T)$ under Assumption \ref{channel_reg}) and $\normt{\mu_0-\nu_0}$.

\begin{remark}\label{rem:costpert}
Although our focus is on kernel perturbations, mismatches in the stage-wise cost can be incorporated at no extra technical expense, since the results of \cite{zhou2024robustness} are stated for a pair $(c,\T)$ versus $(\hat c,{\cal S})$. If $\hat c$ is an approximate cost, then $\norm{\tilde c - \tilde{\hat c}}_\infty \le \norm{c-\hat c}_\infty$ for the belief costs, and when both the cost and the kernels are perturbed the following terms are added to the bounds: $\frac{1}{1-\beta}\norm{c-\hat c}_\infty$ in Corollary \ref{discount_cont} \cite[Theorem 2.4]{zhou2024robustness}; $\frac{2}{(1-\beta)^2}\norm{c-\hat c}_\infty$ in Corollary \ref{discount_robust} \cite[eq. (2.8)]{zhou2024robustness}; $\frac{2}{1-\beta}\norm{c-\hat c}_\infty$ in Corollary \ref{discount_robust_twosided} \cite[eq. (2.9)]{zhou2024robustness}; $\norm{c-\hat c}_\infty$ in Corollary \ref{cont_average} \cite[Theorem 2.6]{zhou2024robustness}; and $2\norm{c-\hat c}_\infty$ in Corollary \ref{robust_average} \cite[Theorem 2.9]{zhou2024robustness}. The same terms are added to the results for an incorrect prior and for a misspecified model, since only the model error terms of Lemmas \ref{lem:perfdiff} and \ref{lem:perfdiff_avg} involve the cost of the approximate model: $\frac{2}{(1-\beta)^2}\norm{c-\hat c}_\infty$ in the first bounds and $\frac{2}{1-\beta}\norm{c-\hat c}_\infty$ in the second bounds of Theorems \ref{thm:transfer} and \ref{thm:misspecified}, and $2\norm{c-\hat c}_\infty$ in Corollaries \ref{cor:transfer_avg} and \ref{cor:misspecified_avg}, where in $L_n$ the constant $K_1$ is replaced by the Lipschitz constant of $\hat c$. The belief error terms are unchanged.
\end{remark}

\section{Application: Finite Model Approximation Bounds via State and Measurement Quantization}\label{sec:quantization}

We now show that the analysis leads to explicit bounds via a quantization approach where both the state and the measurements are approximated with finite models. We first explain why the Wasserstein refinements of Section \ref{sec:proximity} are important here.

\begin{remark}\label{rem:whyW1}
Let $\T_n$ be any approximation of $\T$ whose conditional distributions are supported on a finite set (as is the case for any quantized model). If $\T(\cdot|x,u)$ is non-atomic for some $(x,u)$, then
\[
\normt{\T_n(\cdot|x,u) - \T(\cdot|x,u)} = 2 \quad \text{for every } n,
\]
so $d_{TV}(\T_n,\T) = 2$, by contrast, $d_{W_1}(\T_n,\T) \to 0$ as the quantization becomes finer (Lemma \ref{quantizedT}); accordingly, the bounds of Theorems \ref{main2} and \ref{W1main_2}, in which the transition-kernel term is measured in $d_{W_1}$, are exactly the ones that yield convergent finite-model guarantees. For the measurement channel the situation is different: quantized measurements can be analyzed through an informationally equivalent channel that remains absolutely continuous (see Section \ref{sec:observation_quant}), so the channel term may remain in total variation.
\end{remark}

\subsection{Approximate Finite State Space Model}\label{sec:transition_quant}

In this subsection, we apply the quantization introduced in \cite[Section 3]{SaYuLi15c}, which together with the following subsection leads to a finite POMDP model. For simplicity, we assume a finite action space $\mathbb{U}$, noting that general compact action spaces can be similarly approximated under weak Feller regularity (see \cite[Chapter 3]{SaLiYuSpringer}). We note that discounted and average cost near-optimality of quantized-state (fully observed) models was established in \cite{KSYContQLearning,SaLiYuSpringer,ky2023qaverage}.

Let $\mathbb{X}$ be a compact metric space. 
The quantization scheme partitions the state space $\mathbb{X}$ into disjoint Borel sets $\{B_i\}_{i=1}^n$ with $\bigcup_i B_i = \mathbb{X}$. A finite set of representative states is then defined as
\[
\mathbb{X}_n := \{ x_1, \dots, x_n \}, \qquad x_i \in B_i \ \text{ for each } i,
\]
and the quantization map $\psi_{\mathbb{X}}: \mathbb{X} \rightarrow \mathbb{X}_n$ is defined by
\[
\psi_{\mathbb{X}}(x) = x_i, \quad \text{if } x \in B_i.
\]
Fix a reference measure $\mu$ on $\mathbb{X}$ with $\mu(B_i) > 0$, and let
\[
\mu_i^*(\cdot) := \frac{\mu(\cdot \cap B_i)}{\mu(B_i)}
\]
denote its normalized restriction to $B_i$. The quantized transition kernel $\T_n: \mathbb{X} \times \mathbb{U} \to \mathcal{P}(\X_n)$ is then defined as
\[
\T_n(x_j \mid x, u) := \int_{B_i} \mathcal{T}(B_j \mid x', u)\, \mu_i^*(dx'),
\]
for any $x \in B_i$ and $u \in \mathbb{U}$. The approximation error due to quantization is controlled by the maximal cell diameter
\[
L_{\mathbb{X}_n} := \max_{i \in \{1,\dots, n\}} \sup_{x, x' \in B_i} d(x,x').
\]

We then have the following results.

\begin{lemma}\label{lem:weakW1}
    Suppose that $\X$ and $\U$ are compact and that the stochastic kernel $\mathcal{T}(\cdot \mid x, u)$ is weakly continuous in $(x, u)$.  
    Then, as $L_{\mathbb{X}_n} \to 0$, 
    \(
    d_{W_{1}}(\mathcal{T}, {\mathcal{T}}_n) \to 0.
    \)
\end{lemma}

If $\mathcal{T}(\cdot | x, u)$ is Lipschitz in $x$ with respect to $W_1$, we obtain an explicit rate (see also \cite[Lemma 3.2]{devranby2025near}):

\begin{lemma}\label{quantizedT}
        Suppose Assumption~\ref{weakTtvQweakTtvQ2}(i)-(ii) holds, i.e., $W_1(\T(\cdot|x,u),\T(\cdot|x',u)) \leq \theta\, d(x,x')$ uniformly in $u$. Then the quantized kernel $\mathcal{T}_n$ satisfies
        \[
        d_{W_{1}}(\mathcal{T}, \mathcal{T}_n)  \leq (\theta+1)L_{\mathbb{X}_n}.
        \]
\end{lemma}
        
\begin{proof}
        Define the intermediate kernel
        \[
        \widetilde{\mathcal{T}}_n(\cdot \mid x, u) := \int_{B_i} \mathcal{T}(\cdot \mid x', u) \, \mu_i^*(dx'), \qquad x \in B_i,
        \]
        where $\mu_i^*$ is the normalized restriction of the reference measure $\mu$ to $B_i$.
        We split the approximation error as
        \[
        W_1\left( \mathcal{T}_n(\cdot \mid x, u), \mathcal{T}(\cdot \mid x, u) \right) \leq A_1 + A_2,
        \]
        where
        \begin{align*}
        A_1 &:= W_1\left( \widetilde{\mathcal{T}}_n(\cdot \mid x, u), \mathcal{T}(\cdot \mid x, u) \right),\quad A_2:= W_1\left( \mathcal{T}_n(\cdot \mid x, u), \widetilde{\mathcal{T}}_n(\cdot \mid x, u) \right).
        \end{align*}       
        Using the convexity of $W_1$ in its arguments and the Lipschitz continuity of $\mathcal{T}$ in $x$, for $x \in B_i$,
        \begin{align*}
        A_1 &\leq \int_{B_i} W_1\left( \mathcal{T}(\cdot \mid x', u), \mathcal{T}(\cdot \mid x, u) \right) \mu_i^*(dx') \leq \int_{B_i} \theta\, d(x', x) \, \mu_i^*(dx') \leq \theta L_{\mathbb{X}_n}.
        \end{align*}
        Note that $\mathcal{T}_n(\cdot \mid x, u)$ is obtained from $\widetilde{\mathcal{T}}_n(\cdot \mid x, u)$ by transporting the mass of each cell $B_j$ to its representative point $x_{j}$, with a maximum displacement of $\sup_{y \in B_j} d(y, x_{j}) \leq L_{\mathbb{X}_n}$. We thus have \cite[Theorem 2.6]{Kre11}:
        \[
        A_2 \leq \sum_{j=1}^{n} \int_{B_j} d(y, x_{j}) \, \widetilde{\mathcal{T}}_n(dy \mid x, u) \leq L_{\mathbb{X}_n}.
        \]
        Combining both terms concludes the proof.
\end{proof}

Combining Lemma \ref{quantizedT} with Theorem \ref{W1main_2}(b) (with $Q_n = Q$), we immediately obtain the following.

\begin{lemma}\label{quantizedeta}
    Under Assumptions \ref{channel_reg} and \ref{weakTtvQweakTtvQ2}(i)-(ii), with $D$ the diameter of $\X$,
    \[
    \dWW\big(\eta^{{\cal T},Q},\eta^{{\cal T}_n,Q}\big) \leq \left[\frac{3D}{2} L_Q + 1\right] (\theta+1)L_{\mathbb{X}_n}.
    \]
    In particular, $\dWW\big(\eta^{{\cal T},Q},\eta^{{\cal T}_n,Q}\big) \to 0$ as $L_{\mathbb{X}_n} \to 0$. By Theorem \ref{W1main_2}(a), the same bound holds for $\dWBL$ with $3D/2$ replaced by $\CD$.
\end{lemma}

Using these results and Corollary \ref{discount_robust}, the policy obtained from the state-quantized model is near-optimal, with an explicit rate:

\begin{corollary}\label{cor:state_quant_disc}
    Under Assumptions \ref{main_assumption}, \ref{cost_assumption}, and \ref{channel_reg}, with $ \beta K_2 < 1 $ and $\eta$, $\eta^{\T_n,Q}$ weak Feller, we have:
    \[
    \left\Vert J_{\beta}\big(\eta,\gamma^*_{\eta^{{\cal T}_n,Q}}\big) - J^*_{\beta}(\eta) \right\Vert_{\infty} \leq \frac{2\beta K_1}{(1-\beta)^2(1-\beta K_2)}\left[\frac{3D}{2} L_Q + 1\right] (\theta+1) L_{\mathbb{X}_n},
    \]
    where $\theta$ is a $W_1$-Lipschitz constant of $\T(\cdot|x,u)$ in $x$ as in Assumption \ref{weakTtvQweakTtvQ2}(ii); by Remark \ref{rem:metrics}(ii), under Assumption \ref{main_assumption} one may always take $\theta = \alpha D/2$ (Lemma \ref{quantizedT} does not use $\theta<1$).
\end{corollary}

If the controller assumes the quantized model and runs its filter, Theorem \ref{thm:misspecified} gives the same bound with the constant $\frac{2\beta K_1}{(1-\beta)^2(1-\beta K_2)}$ replaced by $C_\beta$ of that theorem. A similar statement holds for the average cost criterion by invoking Corollary \ref{robust_average} instead, under the additional regularity condition on the quantized model discussed in Corollary \ref{thm:joint_average_bound} and Remark \ref{rem:quantized_not_lipschitz}; we defer the combined statement to Section \ref{sec:joint_quant}.

\subsection{Approximate Finite Measurement Space Model}\label{sec:observation_quant}

We now investigate the effect of quantizing the measurement space. The quantization scheme partitions the observation space $\mathbb{Y}$ into disjoint Borel sets $\{B_i\}_{i=1}^n$ with $\bigcup_i B_i = \mathbb{Y}$. A finite set of representative observations is then defined as
\[
\mathbb{Y}_n := \{ y_1, \dots, y_n \},
\]
and the quantization map $\psi_{\mathbb{Y}}: \mathbb{Y} \rightarrow \mathbb{Y}_n$ is defined by
\[
\psi_{\mathbb{Y}}(y) = y_i, \quad \text{if } y \in B_i.
\]
Using this mapping, a new POMDP model is constructed with a quantized observation channel $Q_n$ defined by
\[
Q_n(y_i \mid x) := Q(B_i \mid x),
\]
where $Q$ is the original observation channel. Given an initial distribution $\mu$ over the state space, the optimal cost under the quantized channel is defined as
\[
J_\beta^*(\mu, Q_n) := \inf_{\widehat{\gamma} \in \widehat{\Gamma}} J_\beta(\mu, Q_n, \widehat{\gamma}),
\]
where the policies $\widehat{\gamma}$ are admissible with respect to the (finite) $\mathbb{Y}_n$-valued measurements.

\begin{assumption}\label{quantMeasurementL}
$\mathbb{Y} \subset \mathbb{R}^m$ is compact and $Q(dy \mid x) = g(x, y)\lambda(dy)$ for a reference probability measure $\lambda$ on $\Y$ (e.g., normalized Lebesgue measure; the normalization is without loss of generality and only rescales $g$ and $\alpha_\Y$), where $g(x, y)$ is Lipschitz continuous in $y$ with constant $\alpha_{\mathbb{Y}}$, i.e.,
\begin{align*}
|g(x, y) - g(x, y')| \leq \alpha_{\mathbb{Y}} \|y - y'\|, \quad \forall y, y' \in \mathbb{Y}, \; x \in \mathbb{X}.
\end{align*}
\end{assumption}

To facilitate the analysis, we define an intermediate channel $\tilde{Q}_n$ with density $\tilde{g}$ with respect to $\lambda$, such that for $y' \in B_i$,
\begin{align*}
\tilde{g}(x,y')=\frac{Q(B_i|x)}{\lambda(B_i)}=\frac{1}{\lambda(B_i)}\int_{B_i}g(x,y)\lambda(dy).
\end{align*}
This channel and the discretized channel achieve the same value function as they are informationally equivalent \cite{devranby2025near}; note that $\tilde{Q}_n$ is a garbling of $Q$ (conditional averaging over the cells), so by Remark \ref{rem:LQuniform} the family $\{\tilde{Q}_n\}$ inherits the TV-Lipschitz constant $L_Q$ of $Q$ whenever $Q$ satisfies Assumption \ref{weakTtvQweakTtvQ2}(iii). The approximation error due to quantization is controlled by
\[
L_{\mathbb{Y}_n} := \max_{i \in \{1,\dots, n\}} \sup_{y, y' \in B_i} \|y - y'\|.
\]
Under Assumption~\ref{quantMeasurementL}, it is shown in \cite[p. 257]{devranby2025near} that
\begin{align}\label{TVchannelquant}
d_{TV}(\tilde{Q}_n, Q)=\sup_x \|\tilde{Q}_n(dy|x)-Q(dy | x)\|_{TV} \leq \alpha_{\mathbb{Y}} L_{\mathbb{Y}_n}
\end{align}
(indeed, $|\tilde g(x,y)-g(x,y)|\le \alpha_\Y L_{\Y_n}$ pointwise, and integrating against the probability measure $\lambda$ gives (\ref{TVchannelquant}); for a general finite reference measure the right-hand side carries the additional factor $\lambda(\Y)$).
Given the total variation bound (\ref{TVchannelquant}), \cite[Theorem 5.2]{devranby2025near} (see also \cite[Theorem 6.2]{YukselOptimizationofChannels}) yields
\begin{align}\label{observation_cont}
J_\beta^*(\mu, Q_n) - J_\beta^*(\mu, Q) \leq \frac{\beta}{(1 - \beta)^2} \|c\|_\infty \alpha_{\mathbb{Y}} L_{\mathbb{Y}_n},
\end{align}
and, since the value of the optimal policy $\hat\gamma$ designed for the degraded (quantized) channel, applied under the original channel, coincides with its value under the degraded channel,
\begin{align}\label{observation_robust}
J_\beta(\mu, Q, \hat{\gamma}) - J_\beta^*(\mu, Q) \leq \frac{\beta}{(1 - \beta)^2} \|c\|_\infty \alpha_{\mathbb{Y}} L_{\mathbb{Y}_n}.
\end{align}
Alternatively, using Corollary~\ref{discount_cont} with $\T_n = \T$, we obtain the following bound, which replaces the factor $\norm{c}_\infty/(1-\beta)$ in (\ref{observation_cont}) with the Lipschitz-based factor $\frac{K_1}{1-\beta K_2}\frac{3D}{2}$; either bound may be sharper depending on the model. Since $J_\beta(\mu,Q,\hat\gamma) = J^*_\beta(\mu,Q_n)$, the same bound also controls the left-hand side of (\ref{observation_robust}).

\begin{corollary}\label{cor:obs_quant}
    Under Assumptions \ref{main_assumption}, \ref{cost_assumption}, and
    \ref{quantMeasurementL}, with $\beta K_2 < 1$ and $\eta$, $\eta^{\T,\tilde Q_n}$ weak Feller,
\begin{align*}
& J_\beta^*(\mu, Q_n) - J_\beta^*(\mu, Q) \leq
\frac{\beta}{1-\beta} \frac{K_1}{1 - \beta K_2} \frac{3D}{2}  \alpha_{\mathbb{Y}} L_{\mathbb{Y}_n},
\end{align*}
and the same bound holds for $J_\beta(\mu, Q, \hat{\gamma}) - J_\beta^*(\mu, Q)$, where $\hat{\gamma}$ denotes the optimal policy corresponding to the approximate model. Furthermore, if Assumption \ref{main_assumption} holds for both models with $K_2<1$ and $K_2^{\T,\tilde Q_n}<1$, and $\hat\gamma = \gamma^*_{\eta^{{\cal T}, \tilde Q_n}}$ is the stationary policy attaining the infimum in the ACOE of the approximate model, then
\[
\left\Vert J_{\infty}(\eta, \gamma^*_{\eta^{{\cal T}, \tilde Q_n}}) - J^*_{\infty}(\eta) \right\Vert_{\infty}
\leq \left( \frac{K_1}{1 - K_2} + \frac{K_1}{1 - K_2^{\T, \tilde Q_n}} \right)
\cdot \frac{3D}{2}  \alpha_{\mathbb{Y}} L_{\mathbb{Y}_n}.
\]
\end{corollary}

If the controller assumes the model $(\T,\tilde Q_n)$ and runs its filter, which depends on the observations only through $\psi_\Y(Y_t)$ and is therefore the filter of the quantized model, Theorem \ref{thm:misspecified} and Corollary \ref{cor:misspecified_avg} give the same bounds with $\frac{\beta}{1-\beta}\frac{K_1}{1-\beta K_2}$ replaced by $C_\beta$ and $\frac{K_1}{1-K_2}+\frac{K_1}{1-K_2^{\T,\tilde Q_n}}$ replaced by $\frac{K_1}{1-K_2}+\frac{K_1}{1-K_2^{\T,\tilde Q_n}}+\frac{K_1+(1+K_2)K_1/(1-K_2^{\T,\tilde Q_n})}{1-K_2}$, respectively. Since only the channel is perturbed here, Assumption \ref{main_assumption} holds for the approximate model $(\T,\tilde Q_n)$ whenever it holds for the true model, with $K_2^{\T,\tilde Q_n} = \alpha D(3-2\delta(\tilde Q_n))/2$. This is not the case for the state-quantized models of Section \ref{sec:joint_quant}, see Remark \ref{rem:quantized_not_lipschitz}.

\subsection{Finite-POMDP Approximation via Joint Quantization of State and Measurement Spaces}\label{sec:joint_quant}

We now combine the results of the preceding two subsections to obtain an explicit performance bound when both the transition kernel and the observation channel are jointly quantized. That is, we consider approximate models of the form $(\mathcal{T}_n, Q_n)$ constructed via quantization of both the state and measurement spaces. This leads to a finite model approximation, suitable both for numerical and learning-theoretic methods.

\paragraph{Original POMDP model.} The controlled process $\{x_k, y_k\}_{k \geq 0}$ is governed by
\[
x_{k+1} \sim \mathcal{T}(\cdot \mid x_k, u_k), \quad y_k \sim Q(\cdot \mid x_k).
\]

\paragraph{Joint finite approximate model.}  
Let $\mathcal{T}_n$ be the quantized transition kernel constructed as in Section~\ref{sec:transition_quant}, and $Q_n$ be the quantized observation channel as in Section~\ref{sec:observation_quant} (analyzed through the informationally equivalent channel $\tilde{Q}_n$). Then the approximate finite model is defined by:
\[
x_{k+1} \sim \mathcal{T}_n(\cdot \mid x_k, u_k), \quad y_k \sim Q_n(\cdot \mid x_k).
\]
Since both $\mathcal{T}_n$ and $Q_n$ take values in finite sets, the resulting belief space is finite-dimensional and suitable for standard finite-POMDP algorithms.

\begin{theorem}\label{thm:joint_kernel_bound}
    Under Assumptions \ref{channel_reg}, \ref{weakTtvQweakTtvQ2}(i)-(ii) and \ref{quantMeasurementL}, we have:
    \[
    \dWW\big(\eta^{{\cal T},Q},\eta^{{\cal T}_n,\tilde Q_n}\big) \leq \frac{3D}{2}
\left[ L_Q \, (\theta+1)L_{\mathbb{X}_n}+\alpha_\Y L_{\Y_n} \right] + (\theta+1)L_{\mathbb{X}_n}.
    \]
    The same bound holds for $\dWBL$ with $3D/2$ replaced by $\CD$.
\end{theorem}

\begin{proof}
Apply Theorem \ref{W1main_2}(b), respectively (a), to the pair $(\T_n, \tilde Q_n)$. By Lemma \ref{quantizedT} and (\ref{TVchannelquant}), $d_{W_1}(\T_n,\T) \leq (\theta+1)L_{\X_n}$ and $d_{TV}(\tilde Q_n, Q) \leq \alpha_\Y L_{\Y_n}$, and Assumption \ref{channel_reg} for the family $\{\tilde Q_n\}$ holds by Remark \ref{rem:LQuniform}.
\end{proof}

\begin{corollary}
\label{thm:joint_discounted_bound}
Let $\mathbb{X}$ and $\mathbb{Y}$ be compact metric spaces. Suppose Assumptions~\ref{main_assumption}, \ref{channel_reg}, \ref{cost_assumption}, and \ref{quantMeasurementL} hold, $\beta K_2 < 1$, and $\eta$, $\eta^{\T_n,\tilde Q_n}$ are weak Feller; let $\theta$ be as in Corollary \ref{cor:state_quant_disc}. Then the optimal policy $\gamma^*_{\eta^{{\cal T}_n, Q_n}}$ for the approximate finite model $(\mathcal{T}_n, Q_n)$ satisfies:
\begin{align*}
&\left\Vert J_\beta(\eta, \gamma^*_{\eta^{{\cal T}_n, Q_n}}) - J^*_\beta(\eta) \right\Vert_{\infty}\\
&\qquad\leq \frac{2\beta K_1}{(1-\beta)^2(1-\beta K_2)} \left\{ \frac{3D}{2}
\left[ L_Q \, (\theta+1)L_{\mathbb{X}_n}+\alpha_\Y L_{\Y_n} \right] + (\theta+1)L_{\mathbb{X}_n} \right\}.
\end{align*}
\end{corollary}

If the controller assumes the finite model $(\T_n,\tilde Q_n)$ and runs its filter, Theorem \ref{thm:misspecified} gives the same bound with $\frac{2\beta K_1}{(1-\beta)^2(1-\beta K_2)}$ replaced by $C_\beta$. A similar result holds for the average cost case, under an additional regularity condition on the finite model:

\begin{corollary}
\label{thm:joint_average_bound}
Suppose the assumptions of Corollary~\ref{thm:joint_discounted_bound} hold with $K_2 < 1$, and that there exist $\beta^*\in(0,1)$ and $L_n<\infty$ such that $\|J^*_\beta(\eta^{\T_n,\tilde Q_n})\|_{\Lip} \le L_n$ for all $\beta\in[\beta^*,1)$. Let $\gamma^*_{\eta^{\T_n,Q_n}}$ be the stationary policy attaining the infimum in the ACOE of the approximate model (which exists by Theorem \ref{vanishingproof1}). Then the suboptimality in the average cost is bounded as
\begin{align*}
&\left\Vert J_{\infty}(\eta, \gamma^*_{\eta^{{\cal T}_n, Q_n}}) - J^*_{\infty}(\eta) \right\Vert_{\infty}\\
&\qquad\leq \left( \frac{K_1}{1 - K_2} + L_n \right)
\left\{ \frac{3D}{2}
\left[ L_Q \, (\theta+1)L_{\mathbb{X}_n} + \alpha_{\mathbb{Y}} L_{\mathbb{Y}_n} \right] + (\theta+1)L_{\mathbb{X}_n}\right\}.
\end{align*}
\end{corollary}

If the controller assumes the finite model and runs its filter, Corollary \ref{cor:misspecified_avg} gives the same bound with $\frac{K_1}{1-K_2}+L_n$ replaced by $\frac{K_1}{1-K_2}+L_n+\frac{K_1+(1+K_2)L_n}{1-K_2}$.

\begin{remark}[On the regularity of the quantized model]\label{rem:quantized_not_lipschitz}
The quantized kernel $\T_n(\cdot|x,u)$ of Section \ref{sec:transition_quant} is piecewise constant in $x$, so it satisfies neither Assumption \ref{main_assumption} nor Assumption \ref{weakTtvQweakTtvQ2}, and $K_2^{\T_n,\tilde Q_n}$ is not defined for it. Therefore the Lipschitz bound $L_n$ in Corollary \ref{thm:joint_average_bound} cannot be obtained from (\ref{LipBoundJ}) and has to be verified separately; see also the discussion of piecewise constant models in \cite[Section 3.1]{zhou2024robustness}. The discounted cost bound of Corollary \ref{thm:joint_discounted_bound} only uses the regularity of the true model and does not need such a condition. We also note that, since $\U$ is finite here, the discounted cost optimality equation of the finite model has a measurable minimizer without weak Feller continuity, so the weak Feller requirement on $\eta^{\T_n,\tilde Q_n}$ in Corollaries \ref{cor:state_quant_disc} and \ref{thm:joint_discounted_bound} can be dropped \cite[Remark 2.8]{zhou2024robustness}.
\end{remark}

\begin{remark}
The approach here is a computationally more direct counterpart of the belief-quantization approach, under which the space of probability measures is quantized \cite{SYLTAC2017POMDP,tutorialkara2024partially,devranby2025near}, leading to a more tedious implementation. Notably, under only weak Feller regularity of the kernel, \cite{SYLTAC2017POMDP} leads to asymptotic near-optimality; and under Wasserstein regularity of the kernel, via the analysis in \cite[Theorem 5]{KSYContQLearning}, one can obtain approximations with a rate of convergence.
\end{remark}

\section{Examples}\label{sec:examples}

We collect here the examples referenced in the paper: sharpness of the bounds (Section~\ref{sub:ex_sharp}), verification of the assumptions through $F$ and $G$ (Section~\ref{sec:FGconditions}), sufficient conditions for and sharpness of the coefficient $K_2$ (Section~\ref{sub:ex_K2}), and a model-specific bound from the literature (Section~\ref{sub:ex_bearing}).

\subsection{Sharpness of the Bounds}\label{sub:ex_sharp}

Since the bounds of Section~\ref{sec:proximity} are linear in the kernel distances, it is natural to ask whether this linear dependence, and the metrics used, can be improved. The following examples show that they cannot; in each case the filter kernels are explicit and the distances are computed exactly. In this subsection we compute $\dWW$, as in parts (b) of Theorems \ref{W1main_1}--\ref{W1main_2}. When $\X=\{0,1\}$ with $d(0,1)=1$, we write $b_s:=s\,\delta_0+(1-s)\,\delta_1$ for the belief with mass $s\in[0,1]$ at state $0$ (as in Remark \ref{rem:scaling}); then $W_1(b_s,b_{s'})=|s-s'|$, so $(\PP(\X),W_1)$ is isometric to $[0,1]$. Since also $\rho_{BL}(b_s,b_{s'})=\frac23|s-s'|$ (the supremum of $|g(0)-g(1)|$ over $\norm{g}_\infty+\norm{g}_L\le 1$ equals $2/3$), in Examples \ref{ex:channel_sharp} and \ref{prop:TVnecessity} the values of $\dWBL$ are exactly $\frac23$ times the values of $\dWW$ computed below, and the conclusions hold for parts (a).

We first consider the transition term.

\begin{example}\label{prop:sharpness}
Let $\X$ be a compact metric space, $\Y=\X$ and let $Q=Q_n$ be the noiseless channel $Q(\cdot|x)=\delta_x$. Since the observation reveals the state, the posterior after observing $y_1=x_1$ is $\delta_{x_1}$, and therefore $\eta(\cdot|z_0)$ is the pushforward of the predicted measure $\bar\T(z_0):=\int\T(\cdot|x)\,z_0(dx)$ under the map $x\mapsto\delta_x$; the same holds for the approximate model. This map is an isometry of $(\X,d)$ into $({\cal P}(\X),W_1)$, and we obtain
\[
\dWW(\eta,\eta^{{\cal T}_n,Q}) = \sup_{z_0}W_1\bigl(\bar\T(z_0),\bar\T_n(z_0)\bigr) = \sup_x W_1\bigl(\T(\cdot|x),\T_n(\cdot|x)\bigr) = d_{W_1}(\T_n,\T).
\]
Thus the dependence on $d_{W_1}(\T_n,\T)$ is linear and cannot be improved: any bound of the form $\dWW(\eta,\eta^{{\cal T}_n,Q_n}) \le C\, d_{W_1}(\T_n,\T)+C'\, d_{TV}(Q_n,Q)$ requires $C \ge 1$, which is the constant of the last term in Theorem \ref{W1main_2}. (For the noiseless channel Assumption \ref{channel_reg} holds only when $\X$ is finite, with $L_Q \ge 2/D$, so this example does not show equality in Theorem \ref{W1main_2} as a whole.)
\end{example}

We next consider the channel term.

\begin{example}\label{ex:channel_sharp}
Let $\X=\Y=\{0,1\}$ with $d(0,1)=1$, let $\T=\T_n$ be the identity kernel, and let $Q$ and $Q_n$ be the binary symmetric channels with crossover probabilities $\epsilon$ and $\epsilon_n$ in $(0,1/2)$, so that $d_{TV}(Q_n,Q)=2|\epsilon-\epsilon_n|$. At the uniform prior $z_0=b_{1/2}$, the posteriors after $y=0$ and $y=1$ are $b_{1-\epsilon}$ and $b_\epsilon$. The filter kernels are therefore the explicit two-atom measures
\[
\eta(\cdot|z_0)=\tfrac12\, \delta_{b_{1-\epsilon}}+\tfrac12\, \delta_{b_{\epsilon}} \qquad \text{and} \qquad \eta^{{\cal T},Q_n}(\cdot|z_0)=\tfrac12\, \delta_{b_{1-\epsilon_n}}+\tfrac12\, \delta_{b_{\epsilon_n}},
\]
whose $\WW$ distance equals $|\epsilon-\epsilon_n|$. It follows that
\[
\dWW(\eta,\eta^{{\cal T},Q_n}) \ge |\epsilon-\epsilon_n| = \tfrac12\,d_{TV}(Q_n,Q),
\]
and a direct computation gives the bounded-Lipschitz distance $2|\epsilon-\epsilon_n|/(3+|\epsilon-\epsilon_n|)$ between the same two measures (with $\PP(\X)$ metrized by $\rho_{BL}$, as in $d_{BL}$). Both distances are of the order $d_{TV}(Q_n,Q)$, so the bounds of Theorems~\ref{main1} and~\ref{W1main_1} are tight up to a universal constant. Together with Example~\ref{prop:sharpness}, this also shows that the sum on the right-hand side of the bounds cannot be avoided.
\end{example}

We now show that the metrics themselves cannot be relaxed.

\begin{example}\label{prop:TVnecessity}
Let $\X=\{0,1\}$ with $d(0,1)=1$, let $\Y=[0,1]$, and let $\T=\T_n$ be the identity kernel. Let $Q(dy|x)=\lambda(dy)$ be the uniform distribution for both states, and, for even $n$, let
\[
Q_n(dy|0)=2\cdot 1_{A_n}(y)\,\lambda(dy),\qquad Q_n(dy|1)=2\cdot 1_{A_n^c}(y)\,\lambda(dy),\qquad A_n:=\bigcup_{\substack{0\le k<n\\ k\text{ even}}}\Big[\tfrac kn,\tfrac{k+1}n\Big),
\]
so that $\lambda(A_n)=1/2$. $W_1(Q_n(\cdot|x),Q(\cdot|x))\le 1/n$, so that $Q_n(\cdot|x)$ converges to $Q(\cdot|x)$ in the Wasserstein distance, and in fact setwise. On the other hand $Q_n(\cdot|0)$ and $Q_n(\cdot|1)$ have disjoint supports, so that $\normt{Q_n(\cdot|0)-Q_n(\cdot|1)}=2=2\,d(0,1)$; hence Assumption~\ref{channel_reg} holds with $L_Q=2$, while $d_{TV}(Q_n,Q)=1$ for every $n$.

The true channel is uninformative, so the posterior coincides with the prior and $\eta(\cdot|z_0)=\delta_{z_0}$. Each $Q_n$, on the other hand, is perfectly informative, so that for $z_0=b_s$ we have $\eta^{{\cal T},Q_n}(\cdot|z_0)=s\,\delta_{b_1}+(1-s)\,\delta_{b_0}$. Their $\WW$ distance is $2s(1-s)$, which is maximized at $s=1/2$, and therefore
\[
\dWW(\eta,\eta^{{\cal T},Q_n})=\tfrac12 \quad\text{for every even }n, \qquad\text{whereas}\qquad d_{W_1}(Q_n,Q)\to 0.
\]
Consequently, in Theorems~\ref{main1}--\ref{W1main_2} (both parts) the term $d_{TV}(Q_n,Q)$ cannot be replaced by $d_{W_1}(Q_n,Q)$ or by $d_{BL}(Q_n,Q)$, even under Assumption~\ref{channel_reg}. An arbitrarily small perturbation of the channel in the Wasserstein distance can thus turn an uninformative channel into a perfectly informative one.
\end{example}

\begin{example}\label{prop:LQnecessity}
Let $\X=[0,1]$ and $\Y=\{0,1\}$, and let $Q=Q_n$ be the threshold channel $Q(\{1\}|x)=1_{[1/2,1]}(x)$, which admits no finite $L_Q$ since $\normt{Q(\cdot|x)-Q(\cdot|x')}=2$ whenever $x<1/2\le x'$. Let the dynamics be constant in the state,
\[
\T(\cdot|x)=\tfrac12\delta_0+\tfrac12\delta_{1/2},\qquad \T_n(\cdot|x)=\tfrac12\delta_0+\tfrac12\delta_{1/2-1/n},
\]
so that $d_{W_1}(\T_n,\T)=1/(2n)\to 0$ while $d_{TV}(\T_n,\T)=1$. $\eta(\cdot|z_0)=\tfrac12\delta_{\delta_0}+\tfrac12\delta_{\delta_{1/2}}$. Under $\T_n$, the observation is therefore uninformative, and $\eta^{{\cal T}_n,Q}(\cdot|z_0)=\delta_{\mu_n}$ with $\mu_n=\tfrac12\delta_0+\tfrac12\delta_{1/2-1/n}$. As both kernels are independent of $z_0$
\[
\dWW(\eta,\eta^{{\cal T}_n,Q})=\tfrac12 W_1(\delta_0,\mu_n)+\tfrac12 W_1(\delta_{1/2},\mu_n)=\tfrac12\Big(\tfrac14-\tfrac1{2n}\Big)+\tfrac12\Big(\tfrac14+\tfrac1{2n}\Big)=\tfrac14
\]
for every $n$, while $d_{W_1}(\T_n,\T)\to 0$ (here $\X=[0,1]$, $D=1$, and by Remark \ref{rem:metrics}(ii) the corresponding $\dWBL$ is at least $\frac{2}{3}\cdot\frac14=\frac16$). Hence, in the absence of Assumption~\ref{channel_reg}, the term $d_{TV}(\T_n,\T)$ in Theorems~\ref{main1} and~\ref{W1main_1} (both parts) cannot be replaced by $d_{W_1}(\T_n,\T)$.
\end{example}

 We refer to~\cite{kara2020robustness} for related counterexamples at the level of optimal costs, when the transition kernels converge only weakly or setwise.

\subsection{Verifying the Assumptions via the System Functions $F$ and $G$}\label{sec:FGconditions}
We record conditions on $F$, $G$, and the noise under which the regularity assumptions hold and the kernel distances can be estimated from the perturbations of $F$ and $G$. We write the perturbed model as
\begin{align}\label{updateEqPert}
X_{k+1} = F_n(X_k,U_k,W_k^{(n)}), \qquad Y_k = G_n(X_k, V_k).
\end{align}

\begin{proposition}\label{prop:FG_regularity}
Consider the model (\ref{updateEq}).
\begin{itemize}
\item[(i)] If $F$ is Lipschitz in $x$ uniformly in $(u,w)$, that is, $d(F(x,u,w),F(x',u,w)) \leq \theta\, d(x,x')$ for all $x,x',u,w$, then $\T$ is $W_1$-Lipschitz in $x$ with the same constant,
\[
W_1(\T(\cdot|x,u), \T(\cdot|x',u)) \leq \theta\, d(x,x').
\]
Similarly if $d(F_n(x,u,w),F(x,u,w)) \leq \theta$ with $W^{(n)}\sim W$, then $d_{W_1}(\T_n, \T) \leq \theta$.
\item[(ii)] Suppose $\Y \subseteq \mathbb{R}^m$ and $G(x,v) = h(x) + v$, where $h$ is Lipschitz and $v$ admits a density $\phi$ satisfying $\int |\phi(t - a) - \phi(t- b)|\, dt \leq C_\phi \normx{a-b}$ for all $a,b$. Then $Q$ is TV-Lipschitz in $x$,
\[
\normt{Q(\cdot|x) - Q(\cdot|x')} \leq L_Q\, d(x,x'), \qquad L_Q := C_\phi \norm{h}_{\Lip},
\]
so that Assumption \ref{weakTtvQweakTtvQ2}(iii) holds; Assumption \ref{channel_reg} holds as well if, in addition, each $Q_n$ is a garbling of $Q$ (Remark \ref{rem:LQuniform}) or is of the form $G_n(x,v) = h_n(x)+v$ with the same noise density and $\norm{h_n}_{\Lip} \le \norm{h}_{\Lip}$.
\item[(iii)] Suppose $\X \subseteq \mathbb{R}^p$ and the two models are additive, $X_{k+1} = f(X_k,U_k) + W_k$ and $X_{k+1} = f_n(X_k,U_k) + W_k^{(n)}$, with $W_k \sim \mu_W$ and $W_k^{(n)} \sim \mu_W^{(n)}$. If $\sup_{x,u}\normx{f_n(x,u) - f(x,u)} \leq C_n$, then
\[
d_{W_1}(\T_n, \T) \leq C_n + W_1(\mu_W^{(n)}, \mu_W).
\]
\item[(iv)] Suppose $\Y = \mathbb{R}^m$, $Y_k = h(X_k) + V_k$, and $Y_k = h_n(X_k) + V_k$ with $V_k \sim {\cal N}(0, \sigma^2 I)$. If $\sup_x \normx{h_n(x) - h(x)} \leq C_n$, then
\[
d_{TV}(Q_n, Q) \leq \sqrt{2/\pi}\; C_n/\sigma.
\]
\end{itemize}
\end{proposition}

\begin{proof}
(i) For fixed $u$, the pair $(F(x,u,W), F(x',u,W))$ with $W$ drawn from the noise is a coupling of $\T(\cdot|x,u)$ and $\T(\cdot|x',u)$, so \[W_1(\T(\cdot|x,u),\T(\cdot|x',u)) \leq \mathbb{E}\,d(F(x,u,W),F(x',u,W)) \leq \theta\, d(x,x').\]

(ii) Since $Q(\cdot|x)$ has density,
\[
\normt{Q(\cdot|x)-Q(\cdot|x')} = \int |\phi(t-h(x)) - \phi(t-h(x'))|\, dt \leq C_\phi \normx{h(x)-h(x')} \leq C_\phi \norm{h}_{\Lip}\, d(x,x').
\]

(iii) Couple the two kernels through a optimal coupling $(W, W^{(n)})$ of the noise laws: for any $(x,u)$, $\normx{(f_n(x,u)+W^{(n)}) - (f(x,u)+W)} \leq C_n + \normx{W^{(n)} - W}$; taking expectations and the supremum over $(x,u)$ gives the bound.

(iv) Apply (ii) with $a = h_n(x)$, $b=h(x)$, and $C_\phi = \sqrt{2/\pi}/\sigma$.
\end{proof}

We illustrate with two examples.

\begin{example}[Truncated Gaussian dynamics]\label{ex:truncnormal}
Let $\mathbb{X} = [0,1]$, let $\mathbb{U} \subset \mathbb{R}$ be compact, and let $\mathcal{T}(\cdot \mid x,u) = \overline{N}(x+u,\sigma^2)$ be the law of $N(x+u,\sigma^2)$ truncated to $[0,1]$. For any $0 \le x < x' \le 1$ and $u$,
\[
\frac{\bigl\|\mathcal{T}(\cdot \mid x',u) - \mathcal{T}(\cdot \mid x,u)\bigr\|_{TV}}{x'-x} \;\le\; \frac{\sqrt{2}}{\sigma \sqrt{\pi}},
\]
so Assumption~\ref{main_assumption} holds with $\alpha = \sqrt{2}/(\sigma \sqrt{\pi})$ (see \cite{demirci2023average}); by Remark \ref{rem:K2suff}(i), $K_2 < 1$ once $\sigma > 3/\sqrt{2\pi}$.
\end{example}

\begin{example}[Additive-noise models]\label{ex:additive}
Let $x_{k+1} = f(x_k,u_k,w_k)$ with $f$ Lipschitz in $x$ uniformly in $(u,w)$ with constant $\theta$, and let $y_k = h(x_k) + v_k$ with $h$ Lipschitz and $v_k$ admitting a density as in Proposition \ref{prop:FG_regularity}(ii). By parts (i) and (ii), if $\X$ is compact and $\theta<1$, Assumption \ref{weakTtvQweakTtvQ2} holds with $\theta$ and $L_Q = C_\phi \norm{h}_{\Lip}$. If, in addition, the approximate channels $Q_n$ are garblings of $Q$ or share its noise density, then Assumption \ref{channel_reg} holds with the same $L_Q$ by Remark \ref{rem:LQuniform}. See \cite[Section~2.1]{KSYContQLearning} for further explicit models.
\end{example}

\subsection{The Coefficient $K_2$: Sufficient Conditions and Sharpness}\label{sub:ex_K2}

The results of Section \ref{robustness} require $K_2<1$ (or $\beta K_2<1$), where $K_2$ is the regularity coefficient of Remark \ref{rem:tradeoff}. We first give explicit sufficient conditions for $K_2<1$, and then show that $K_2$ is attained for some models, so that it cannot be improved in general.

\begin{remark}[Sufficient conditions for $K_2<1$]\label{rem:K2suff}\noindent
\begin{itemize}
\item[(i)] Under Assumption \ref{main_assumption}, since $\delta(Q)\in[0,1]$ implies $3-2\delta(Q) \in [1,3]$, a sufficient condition for $K_2 = \alpha D(3-2\delta(Q))/2 < 1$ is
\[
\alpha <\frac{2}{(3-2\delta(Q))D}, \qquad \text{in particular} \qquad \alpha<\frac{2}{3D},
\]
regardless of the channel. 
For instance, for the model $x_{k+1} = f(x_k,u_k) + w_k$ with $w_k$ admitting a density $\phi$, one may take $\alpha = \normx{\phi'}_{1}\,\sup_{u}\norm{f(\cdot,u)}_{\Lip}$.
Furthermore, as the channel becomes uninformative ($\delta(Q) \to 1$), the factor $3-2\delta(Q)$ decreases to $1$.
\item[(ii)] Under Assumption \ref{weakTtvQweakTtvQ2}, $K_2 = \theta(1 + 3 L_Q D/2) < 1$ holds when the dynamics are contractive ($\theta<1$) and  $L_Q < \frac{2(1-\theta)}{3\theta D}$. For instance, for the model $x_{k+1} = f(x_k,u_k) + w_k$, $y_k = h(x_k) + v_k$ of Example \ref{ex:additive}, one may take $\theta = \sup_{u}\norm{f(\cdot,u)}_{\Lip}$ and $L_Q = \norm{h}_{\Lip}\,\normx{\phi'}_{1}$ where $\phi$ is the density of $v_k$ (see Proposition \ref{prop:FG_regularity}).
\end{itemize}
\end{remark}

We next show that the coefficient $K_2$ of Theorem \ref{ergodicity} is not merely an upper bound: it is attained with equality for some models.

\begin{example}[Failure of contraction when $K_2 \geq 1$]\label{prop:K2exact}
Let $\X=\{0,1\}$ with $d(0,1)=1$ (so $D=1$), let the channel be uninformative, $Q(\cdot|0)=Q(\cdot|1)$ (so $\delta(Q)=1$), and let $\T(0|0)=a$, $\T(0|1)=b$. 
When the channel is uninformative, the one-step filter equals the predictor: $\eta(\cdot|z) = \delta_{z\T}$, where $z\T(\cdot) = \int \T(\cdot|x)z(dx)$. Hence $\WW(\eta(\cdot|z),\eta(\cdot|z')) = W_1(z\T, z'\T)$.
Writing $t = z(\{0\})$, we have $z\T(\{0\}) = b + t(a-b)$, so $W_1(z\T,z'\T) = |z\T(\{0\}) - z'\T(\{0\})| = |a-b|\,|t-t'| = |a-b|\, W_1(z,z')$ for all $z,z'$, and $\alpha = \normt{\T(\cdot|0)-\T(\cdot|1)} = 2|a-b|$.
Then $K_2 = \alpha D (3-2\delta(Q))/2 = |a-b|$, and
\[
\WW\bigl(\eta(\cdot|z), \eta(\cdot|z')\bigr) \;=\; K_2\, W_1(z,z') \qquad \text{for \emph{all} } z,z' \in {\cal P}(\X),
\]
i.e., the exact Lipschitz coefficient of $\eta$ equals $K_2$. In particular, for $a=1$, $b=0$, the filter kernel $\eta(\cdot|z)=\delta_{z}$ ($K_2=1$): every belief is a fixed point, and for a control-free problem with a nonconstant Lipschitz cost the optimal average cost $J^*_{\infty}(z)=\int c(x)\,z(dx)$ depends on the initial belief. Hence the conclusions of Theorem \ref{mainEmre} might fail at $K_2=1$.
\end{example}

The next example shows that the $\alpha D$ term is unavoidable.
\begin{example}[Necessity of the $\alpha D$]\label{prop:K2channel}
Let $0<h<D$, let $\X = \{x^{(0)},x^{(1)},x^{(2)}\} = \{0,h,D\} \subset \mathbb{R}$, and let $\Y = \{A,B\}$ with the (perfectly informative on $x^{(1)}$) channel
\[
Q(A|x^{(1)}) = 1, \qquad Q(A|x^{(0)}) = Q(A|x^{(2)}) = 0,
\]
so that $\delta(Q) = 0$. For $p,r \in (0,1)$ with $p + r < 1$, define
\begin{align*}
\T(\cdot|x^{(0)}) &= p\, \delta_{x^{(0)}} + (1-p-r)\,\delta_{x^{(1)}} + r\, \delta_{x^{(2)}},\\
\T(\cdot|x^{(1)}) &= \T(\cdot|x^{(2)}) = p\, \delta_{x^{(0)}} + (1-p)\,\delta_{x^{(1)}}.
\end{align*}
Then $\normt{\T(\cdot|x^{(0)}) - \T(\cdot|x^{(1)})} = 2r$ with $d(x^{(0)},x^{(1)}) = h$, $\normt{\T(\cdot|x^{(0)}) - \T(\cdot|x^{(2)})} = 2r$ with $d(x^{(0)},x^{(2)}) = D > h$; hence 
$\alpha = 2r/h$, so that $\alpha D/2 = r D/h$ and $K_2 = 3r D/h$, while, for $z = \delta_{x^{(1)}}$ and $z_\epsilon = (1-\epsilon)\delta_{x^{(1)}} + \epsilon\,\delta_{x^{(0)}}$,
\[
\lim_{\epsilon \downarrow 0} \frac{\WW\bigl(\eta(\cdot|z), \eta(\cdot|z_\epsilon)\bigr)}{W_1(z, z_\epsilon)} \;=\; \frac{r(D+h)}{h} \;=\; \frac{D+h}{D}\cdot \frac{\alpha D}{2}.
\]

Under the prior $z = \delta_{x^{(1)}}$, the predictor is $p\delta_{x^{(0)}} + (1-p)\delta_{x^{(1)}}$; the observation $A$ (respectively, $B$) reveals that the state is $x^{(1)}$ (respectively, $x^{(0)}$), so
\(
\eta(\cdot|z) = (1-p)\,\delta_{\delta_{x^{(1)}}} + p\, \delta_{\delta_{x^{(0)}}}.
\)
Under $z_\epsilon$, the predictor is $p\delta_{x^{(0)}} + (1-p-\epsilon r)\delta_{x^{(1)}} + \epsilon r\,\delta_{x^{(2)}}$, and the posterior given $B$ is $\nu_\epsilon := (p\, \delta_{x^{(0)}} + \epsilon r\, \delta_{x^{(2)}})/(p+\epsilon r)$, so
\(
\eta(\cdot|z_\epsilon) = (1-p-\epsilon r)\,\delta_{\delta_{x^{(1)}}} + (p+\epsilon r)\,\delta_{\nu_\epsilon}.
\)
The $\WW$ distance between these two measures is:
\[
\epsilon r\, W_1(\delta_{x^{(1)}}, \nu_\epsilon) + p\, W_1(\delta_{x^{(0)}}, \nu_\epsilon)
= \frac{\epsilon r\,\bigl[p h + \epsilon r (D-h)\bigr]}{p + \epsilon r} + \frac{p\,\epsilon r D}{p+\epsilon r},
\]
using $W_1(\delta_{x^{(1)}},\nu_\epsilon) = \frac{p h + \epsilon r(D-h)}{p+\epsilon r}$ and $W_1(\delta_{x^{(0)}},\nu_\epsilon) = \frac{\epsilon r D}{p+\epsilon r}$. 
Since $W_1(z,z_\epsilon) = \epsilon h$, we conclude
\[
\frac{\WW\bigl(\eta(\cdot|z), \eta(\cdot|z_\epsilon)\bigr)}{W_1(z,z_\epsilon)}
= \frac{r\,\bigl[p(h+D) + \epsilon r(D-h)\bigr]}{h\,(p+\epsilon r)}
\;\xrightarrow[\epsilon \downarrow 0]{}\; \frac{r(D+h)}{h}\;=\; \frac{D+h}{D}\cdot \frac{\alpha D}{2}.
\]

Letting $h \uparrow D$, the exact Lipschitz coefficient of $\eta$ can thus be made arbitrarily close to $2 \cdot \frac{\alpha D}{2} = \alpha D$.
\end{example}

\subsection{Recovering a Model-Specific Bound: Bearing-Only Localization}\label{sub:ex_bearing}

Finally, we show how the framework recovers a model-specific bound from the literature.

\begin{example}[Observation channel approximation in a bearing-only localization model]\label{ex:bearing}
Consider the bearing-only localization problem of \cite{Dufour2025}. The hidden state $x \in \mathbb{X}$ denotes the fixed (but unknown) location of a target, where $\mathbb{X} = \left[-\frac{l_1}{2}, \frac{l_1}{2}\right] \times \left[-\frac{l_2}{2}, \frac{l_2}{2}\right] \subset \mathbb{R}^2$. The observation space is $\mathbb{Y} = \boldsymbol{\Theta} \times \mathbf{P}$, where $\boldsymbol{\Theta} = \mathbb{R}$ carries the bearing measurement and $\mathbf{P} \subset \mathbb{R}^4$ is finite and encodes the (deterministically controlled) sensor position. At each time $t$, the bearing observation is
\[
\Theta_t = h(x, P_t) + w_t, \quad w_t \sim \mathcal{N}(0, \sigma^2),\qquad
 h(x, p) := \operatorname{atan2}(p^{(1)} - x^{(1)}, p^{(2)} - x^{(2)}),
\]
so that the observation kernel takes the form
\[
Q(d\vartheta, dp \mid x, u) = \delta_{p(u)}(dp) \cdot \mathcal{N}(h(x, p(u)), \sigma^2)(d\vartheta),
\]
where $p(u)$ is the deterministic sensor position under action $u$. \cite{Dufour2025} proposes to quantize the bearing angle $\vartheta \in \boldsymbol{\Theta}$ via a projection $P_M$ onto a finite set $\boldsymbol{\Theta}_M \subset [-M, M]$, where $M > \pi$ is a truncation parameter. This yields an approximate observation channel $\bar{Q}$. \cite[Proposition 6]{Dufour2025} establishes the bound
\[
\rho_{BL}(z_1, \bar{z}_1) \leq C \cdot \frac{e^{-\frac{1}{2 \sigma^2}(M - \pi)^2}}{M - \pi}
\]
on the one-step filtering distributions $z_1$ and $\bar z_1$ under the original and quantized observation models, together with bounds on the induced value function degradation. Since $\bar Q$ is finitely supported, $d_{TV}(\bar Q,Q)=2$ (Remark \ref{rem:whyW1}); as in Section \ref{sec:observation_quant}, we analyze $\bar Q$ through an informationally equivalent channel $\tilde Q$ with a density. Then $d_{TV}(\tilde Q,Q)$ is bounded by a truncation term of order $e^{-\frac{1}{2\sigma^2}(M-\pi)^2}/(M-\pi)$ (Gaussian tail) plus a discretization term of the form (\ref{TVchannelquant}). Theorem~\ref{main1} then gives a filtering-continuity estimate of the same type as above, uniformly over priors and controls, and Corollary~\ref{discount_robust} yields the corresponding bound on the value degradation under policies derived from the approximate model.
\end{example}

\section{Conclusion}\label{sec:conclusion}

In this paper, we studied the sensitivity of non-linear filter kernels to perturbations in the transition and measurement kernels of a POMDP, and the robustness of optimal control performance to such perturbations. The main technical contribution is a set of explicit and uniform proximity bounds for belief-MDP kernels in the bounded-Lipschitz and Wasserstein metrics (Section \ref{sec:proximity}), shown to be sharp in their linear dependence on the kernel distances, and accompanied by verifiable sufficient conditions stated directly in terms of the system functions. Building on these bounds, and on regularity properties of belief-MDPs, we derived explicit robustness bounds for discounted and average cost criteria, for policies designed for an approximate model, both when the controller knows the true model and when it assumes the approximate model and runs its filter, allowing for an incorrect prior. We also obtained performance guarantees for finite models obtained by jointly quantizing the state and measurement spaces, where the Wasserstein-based refinements were shown to be essential.

\appendix

\section{Proofs of the Results of Section \ref{sec:proximity}}\label{app:proofs}

We begin with the following lemma. Throughout the appendix, for notational economy we suppress the control action $u$ (which is fixed throughout the arguments and plays no role), and we denote by $z_1^{{\cal T}',Q'}(z_0,y_1) := P^{{\cal T}',Q'}(X_1 \in \cdot \mid z_0, Y_1 = y_1)$ the one-step filter of the model $({\cal T}',Q')$, and by $P^{{\cal T}',Q'}(dy_1|z_0)$ the corresponding one-step observation law.

\begin{lemma}\label{Py}
Let $z_0 \in \mathcal{P}(\mathbb{X})$. Then, we have:
$\left\|P^{{\cal T}_n,Q_n}(dy_1 \mid z_0,u) - P(dy_1 \mid z_0, u)\right\|_{TV}$
    $ \leq  d_{TV}(Q_n, Q)+ d_{TV}(\T_n, \T)$.
\end{lemma}

\begin{proof}
First we define an intermediate model determined by \((\T, Q_n)\), leading to a filter kernel $\eta^{{\cal T},Q_n}$. This, given a fixed prior, induces the measure $P^{{\cal T},Q_n}$ on ${\cal B}(\mathbb{X}^{\mathbb{Z}_+} \times \mathbb{Y}^{\mathbb{Z}_+})$.
We first prove that 
\begin{itemize}
    \item[(i)] $\left\|P^{{\cal T},Q_n}(dy_1 \mid z_0) - P(dy_1 \mid z_0)\right\|_{TV} \leq d_{TV}(Q_n, Q)$,
    \item[(ii)] $\left\|P^{{\cal T},Q_n}(dy_1 | z_0) - P^{{\cal T}_n,Q_n}(dy_1 | z_0)\right\|_{TV}$ 
    $\leq d_{TV}(\T_n, \T)$.
\end{itemize}

    \begin{itemize}
        \item[(i)] Note that \begin{align}
& \normt{P^{{\cal T},Q_n}\left(d y_1 \mid z_0\right)-P\left(d y_1 \mid z_0\right)} \nonumber\\
&= \sup _{\norm{g}_\infty \leq 1}\left[\int z_0\left(d x_0\right) \T\left(d x_1 \mid x_0\right) Q_n\left(d y_1 | {x_1}\right) g\left(y_1\right)\right.\nonumber\\
&\qquad\qquad\left. -\int z_0\left(d x_0\right) \T\left(d x_1 \mid x_0\right) Q\left(d{y_1} \mid x_1\right) g (y_1) \right] \nonumber\\
&= \sup _{\norm{g}_\infty \leq 1}\left[\int z_0\left(d x_0\right) \T\left(d x_1 \mid x_0\right) h_n(x_1) -\int z_0\left(d x_0\right) \T\left(d x_1 \mid x_0\right) h(x_1) \right]\!, \nonumber
\end{align}
where $h_n\left(x_1\right)=\int Q_n\left(d{y_1}| {x_1}\right) g\left(y_1\right)$ and $h\left(x_1\right)=\int Q\left(d{y_1} \mid x_1\right) g (y_1)$.
Now observe that
$ h_n\left(x_1\right)-h\left(x_1\right) \leq d_{TV}(Q_n, Q)$ by definition.
Thus:
\begin{align}\label{Q_b}
& \normt{P^{{\cal T},Q_n}\left(d y_1 \mid z_0\right)-P\left(d y_1 \mid z_0\right)} \nonumber\\
&= \sup _{\norm{g}_\infty \leq 1}\left[\int z_0\left(d x_0\right) \T\left(d x_1 \mid x_0\right) h_n(x_1)-\int z_0\left(d x_0\right) \T\left(d x_1 \mid x_0\right) h(x_1) \right]\nonumber\\
&\leq d_{TV}(Q_n, Q).
\end{align}
        \item[(ii)]
        Similarly, note that:
\begin{align}
& \normt{P^{{\cal T},Q_n}\left(d y_1 \mid z_0\right)-P^{{\cal T}_n,Q_n}\left(d y_1 \mid z_0\right)} \nonumber\\
&= \sup _{\norm{g}_\infty\leq 1}\left[\int z_0\left(d x_0\right) \T\left(d x_1 \mid x_0\right) Q_n\left(d y_1 | {x_1}\right) g\left(y_1\right)\right.\nonumber\\
&\qquad\qquad\left. -\int z_0\left(d x_0\right) \T_n\left(d x_1 \mid x_0\right) Q_n\left(d{y_1} \mid x_1\right) g (y_1) \right] \nonumber\\
&= \sup _{\norm{g}_\infty\leq 1}\left[\int z_0\left(d x_0\right) \T\left(d x_1 \mid x_0\right) h_n(x_1) -\int z_0\left(d x_0\right) \T_n\left(d x_1 \mid x_0\right) h_n(x_1) \right] \nonumber
\end{align}
where $h_n\left(x_1\right)=\int Q_n\left(d{y_1}| {x_1}\right) g\left(y_1\right)$, with $\|h_n\|_\infty \leq 1$.
Observe that:
$$ \int  \T\left(d x_1 \mid x_0\right) h_n(x_1)- \int \T_n\left(d x_1 \mid x_0\right) h_n(x_1)  \leq d_{TV}(\T_n, \T)$$ 
by definition.
Hence:
\begin{align}
& \normt{P^{{\cal T},Q_n}\left(d y_1 \mid z_0\right)-P^{{\cal T}_n,Q_n}\left(d y_1 \mid z_0\right)} \nonumber\\
&= \sup _{\norm{g}_\infty \leq 1} \left[ \int z_0\left(d x_0\right) \T\left(d x_1 \mid x_0\right) h_n(x_1) - \int z_0\left(d x_0\right) \T_n\left(d x_1 \mid x_0\right) h_n(x_1) \right]\nonumber\\
&\leq d_{TV}(\T_n, \T). \nonumber
\end{align}
    \end{itemize}
The proof follows from the triangle inequality.
\end{proof}

\noindent\textbf{Proof of Theorem \ref{main1}.} 
First we prove that
\begin{itemize}
    \item[(i)] $d_{BL}(\eta,\eta^{{\cal T},Q_n}) \leq 2 d_{TV}(Q_n,Q)$,
    \item[(ii)] $d_{BL}(\eta^{{\cal T},Q_n},\eta^{{\cal T}_n,Q_n}) \leq 2  d_{TV}(\T_n,\T)$.
\end{itemize}
    (i)
    First, consider:  \begin{align}\label{i}
d_{BL}(\eta,\eta^{{\cal T},Q_n})= \sup_{z_0\in \Z} \sup_{f \in \operatorname{BL}(\Z)} \left|\int_{{\cal P}(\mathbb{X})} f\left(z_1\right) \eta\left(d z_1 \mid z_0\right)-\int_{{\cal P}(\mathbb{X})} f\left(z_1\right) \eta^{{\cal T},Q_n}\left(d z_1 \mid z_0\right)\right|.\end{align}
For a fixed \( f \) such that \( \norm{f}_\infty + \norm{f}_L \leq 1 \), we can write:
\begin{align}
&
\left|\int_{{\cal P}(\mathbb{X})} f\left(z_1\right) \eta^{{\cal T},Q_n}\left(d z_1 \mid z_0\right)-\int_{{\cal P}(\mathbb{X})} f\left(z_1\right) \eta\left(d z_1 \mid z_0\right)\right| \nonumber\\
& =\left|\int_{\mathbb{Y}} f\left(z_1^{{\cal T},Q_n}\left(z_0, y_1\right)\right) P^{{\cal T},Q_n}\left(d y_1 \mid z_0\right)-\int_{\mathbb{Y}} f\left(z_1^{{\cal T},Q}\left(z_0,  y_1\right)\right) P\left(d y_1 \mid z_0\right)\right|\nonumber \\
&\leq \left|\int_{\mathbb{Y}} f\left(z_1^{{\cal T},Q_n}\left(z_0^{ }, y_1\right)\right) P^{{\cal T},Q_n}\left(d y_1 \mid z_0^{ }\right)-\int_{\mathbb{Y}} f\left(z_1^{{\cal T},Q_n}\left(z_0^{ }, y_1\right)\right) P\left(d y_1 \mid z_0\right)\right| \nonumber\\
&\quad+ \int_{\mathbb{Y}}\left|f\left(z_1^{{\cal T},Q_n}\left(z_0^{ }, y_1\right)\right)-f\left(z_1^{{\cal T},Q}\left(z_0, y_1\right)\right)\right| P\left(d y_1 \mid z_0\right) \nonumber\\
&\leq \norm{f}_\infty \left\|P^{{\cal T},Q_n}\left(dy_1 \mid z_0^{ } \right)-P\left(dy_1 \mid z_0\right)\right\|_{T V}\label{prob_dist}\\
&\quad+ \int_{\mathbb{Y}}\left|f\left(z_1^{{\cal T},Q_n}\left(z_0^{ }, y_1\right)\right)-f\left(z_1^{{\cal T},Q}\left(z_0, y_1\right)\right)\right| P\left(d y_1 \mid z_0\right). \label{second}
\end{align}
By Lemma \ref{Py}, we have:
\begin{align}\label{TV1}
&\normt{P^{{\cal T},Q_n}\left(d y_1 \mid z_0\right)-P\left(d y_1 \mid z_0\right)} \leq d_{TV}(Q_n,Q). 
\end{align}
Now we can analyze the second term in (\ref{second}). Using the Lipschitz property of \( f \), we can write:
\begin{align}\label{eq3}
&\int_{\mathbb{Y}}\left|f\left(z_1^{{\cal T},Q_n}\left(z_0^{ }, y_1\right)\right)-f\left(z_1^{{\cal T},Q}\left(z_0, y_1\right)\right)\right| P\left(d y_1 \mid z_0\right)\nonumber\\
&\leq \norm{f}_L \int_{\mathbb{Y}}\rho_{BL} (z_1^{{\cal T},Q_n}\left(z_0, y_1\right), z_1^{{\cal T},Q}\left(z_0, y_1\right)) P\left(d y_1 \mid z_0\right)\nonumber\\
&=\norm{f}_L \int_{\mathbb{Y}} \sup_{g \in \operatorname{BL}_1(\mathbb{X})}\Big(\int_{\mathbb{X}} g(x_1)z_1^{{\cal T},Q_n}\left(z_0, y_1\right)(dx_1)\nonumber\\
&\qquad\qquad - \int_{\mathbb{X}} g(x_1)z_1^{{\cal T},Q}\left(z_0, y_1\right)(dx_1)\Big)P\left(d y_1 \mid z_0\right).
\end{align}
If we examine the term inside, we have:
\begin{align}
&\sup_{g \in \operatorname{BL}_1(\mathbb{X})}\left(\int_{\mathbb{X}} g(x_1)z_1^{{\cal T},Q_n}\left(z_0 , y_1\right)(dx_1) - \int_{\mathbb{X}} g(x_1)z_1^{{\cal T},Q}\left(z_0, y_1\right)(dx_1)\right)\nonumber\\
&\qquad=\sup_{g \in \operatorname{BL}_1(\mathbb{X})}\left(\int_{\mathbb{X}} g(x_1)w_{y_1}(dx_1)\right),
\end{align}
where $w_{y_1}=(z_1^{{\cal T},Q_n}\left(z_0 , y_1\right)- z_1^{{\cal T},Q}\left(z_0, y_1\right))$
is a signed measure on $\mathbb{X}$.
We need a measurable (in $y$) family of (almost) maximizing test functions. Since $\X$ is Polish, there is a countable family $\{g_k\}_{k\in\mathbb{N}} \subset \operatorname{BL}_1(\X)$ with $\rho_{BL}(\mu,\nu) = \sup_k \int g_k(x)\,(\mu-\nu)(dx)$ for all $\mu,\nu\in\PP(\X)$ (see, e.g., \cite[Chapter 8]{Bog07}).
Since $y\mapsto z_1^{{\cal T},Q_n}(z_0,y)$ and $y\mapsto z_1^{{\cal T},Q}(z_0,y)$ are measurable kernels, $y\mapsto \int g_k(x)\,w_y(dx)$ is measurable for each $k$, and so is $y\mapsto \rho_{BL}(z_1^{{\cal T},Q_n}(z_0,y),z_1^{{\cal T},Q}(z_0,y)) = \sup_k\int g_k(x)\,w_y(dx)$. Fix $\varepsilon>0$ and let $k_\varepsilon(y)$ be the smallest $k$ with $\int g_k(x)\,w_y(dx) \ge \sup_j \int g_j(x)\,w_y(dx) - \varepsilon$; then $k_\varepsilon$ is measurable and $g_y := g_{k_\varepsilon(y)}$ defines a jointly measurable map $(x,y)\mapsto g_y(x)$ with $g_y\in\operatorname{BL}_1(\X)$ and
\[
\int_{\X} g_y(x)\,w_y(dx) \;\ge\; \rho_{BL}\bigl(z_1^{{\cal T},Q_n}(z_0,y),z_1^{{\cal T},Q}(z_0,y)\bigr) - \varepsilon \qquad\text{for every } y\in\Y .
\]
The argument below only uses $g_y\in\operatorname{BL}_1(\X)$ and this joint measurability; it yields (\ref{last2}) with an additional term $\varepsilon\norm{f}_L$ on the right-hand side, and letting $\varepsilon\downarrow 0$ gives (\ref{last2}) as stated. (When $\X$ is compact, $\operatorname{BL}_1(\X)$ is compact in the sup-norm by the Arzel\`a--Ascoli theorem and an exact measurable maximizer exists by the measurable maximum theorem \cite[Theorem 18.19]{AlBo06}.)
Using this, we can proceed with (\ref{eq3}) (up to the additive $\varepsilon\norm{f}_L$ just described):
\begin{align}
&\int_{\mathbb{Y}} \sup_{g \in \operatorname{BL}_1(\mathbb{X})}\left(\int_{\mathbb{X}} g(x_1)z_1^{{\cal T},Q_n}\left(z_0 , y_1\right)(dx_1) - \int_{\mathbb{X}} g(x_1)z_1^{{\cal T},Q}\left(z_0, y_1\right)(dx_1)\right)P\left(d y_1 \mid z_0\right)\nonumber\\
&=\int_{\mathbb{Y}} \int_{\mathbb{X}} g_{y_1}(x_1)z_1^{{\cal T},Q_n}\left(z_0 , y_1\right)(dx_1) P(d y_1 \mid z_0)\nonumber\\
&\qquad-\int_{\mathbb{Y}}\int_{\mathbb{X}} g_{y_1}(x_1)z_1^{{\cal T},Q_n}\left(z_0, y_1\right)(dx_1) P^{{\cal T},Q_n}(d y_1 \mid z_0)\nonumber\\
&\quad+\int_{\mathbb{Y}} \int_{\mathbb{X}} g_{y_1}(x_1)z_1^{{\cal T},Q_n}\left(z_0 , y_1\right)(dx_1) P^{{\cal T},Q_n}(d y_1 \mid z_0)\nonumber\\
&\qquad-\int_{\mathbb{Y}}\int_{\mathbb{X}} g_{y_1}(x_1)z_1^{{\cal T},Q}\left(z_0, y_1\right)(dx_1) P(d y_1 \mid z_0 ).\label{gy1}
\end{align}
For the first difference, we know that
\(
y \mapsto \int_{\mathbb{X}} g_{y}(x_1)z_1^{{\cal T},Q_n}\left(z_0 , y\right)(dx_1) \in \operatorname{T}(\Y),
\)
and therefore:
\begin{align}\label{sup-1}
&\int_{\mathbb{Y}} \int_{\mathbb{X}} g_{y_1}(x_1)z_1^{{\cal T},Q_n}\left(z_0 , y_1\right)(dx_1) P(d y_1 \mid z_0)\nonumber\\
&\qquad-\int_{\mathbb{Y}}\int_{\mathbb{X}} g_{y_1}(x_1)z_1^{{\cal T},Q_n}\left(z_0, y_1\right)(dx_1) P^{{\cal T},Q_n}(d y_1 \mid z_0)\nonumber\\
& \leq \normt{P^{{\cal T},Q_n}\left(d y_1 \mid z_0\right)-P\left(d y_1 \mid z_0\right)} \leq d_{TV}(Q_n,Q).  
\end{align}
For the second difference, we can write by smoothing:
\begin{align}\label{sup-2}
&\int_{\mathbb{Y}} \int_{\mathbb{X}} g_{y_1}(x_1)z_1^{{\cal T},Q_n}\left(z_0 , y_1\right)(dx_1) P^{{\cal T},Q_n}(d y_1 \mid z_0)\nonumber\\
&\qquad-\int_{\mathbb{Y}}\int_{\mathbb{X}} g_{y_1}(x_1)z_1^{{\cal T},Q}\left(z_0, y_1\right)(dx_1) P(d y_1 \mid z_0 ) \nonumber \\
&=\int_{\mathbb{X}} \int_{\mathbb{Y}} g_{y_1}(x_1)Q_n\left(d y_1 \mid x_1\right) \mathcal{T}\left(d x_1 \mid z_0^{ }\right)-\int_{\mathbb{X}} \int_{\mathbb{Y}} g_{y_1}(x_1) Q\left(d y_1 \mid x_1\right) \mathcal{T}\left(d x_1 \mid z_0\right)\nonumber\\
&=\int_{\mathbb{X}} \omega_n(x_1) \mathcal{T}\left(d x_1 \mid z_0^{ }\right)-\int_{\mathbb{X}}\omega(x_1) \mathcal{T}\left(d x_1 \mid z_0\right),
\end{align}
where $\omega_n(x_1)=\int_{\mathbb{Y}} g_{y_1}(x_1)Q_n\left(d y_1 \mid x_1\right)$ and $ \omega(x_1)=\int_{\mathbb{Y}} g_{y_1}(x_1)Q\left(d y_1 \mid x_1\right).$
The first equality follows from (\ref{kernelFilter}) and an application of Fubini's theorem, since both integrands are bounded by $1$.  
Moreover, from the above we obtain 
\(
\omega_n(x_1) - \omega(x_1) \;\leq\; d_{TV}(Q_n, Q),
\)
and therefore
\begin{align}\label{w_s}
&\int_{\mathbb{X}} \omega_n(x_1) \mathcal{T}\left(d x_1 \mid z_0^{ }\right)-\int_{\mathbb{X}}\omega(x_1) \mathcal{T}\left(d x_1 \mid z_0\right) \leq d_{TV}(Q_n,Q).
\end{align}
Using the bounds from (\ref{eq3}), (\ref{gy1}), (\ref{sup-1}), (\ref{sup-2}), and (\ref{w_s}), we obtain:
\begin{align}\label{last2}
& \int_{\mathbb{Y}}\left|f\left(z_1^{{\cal T},Q_n}\left(z_0^{ }, y_1\right)\right)-f\left(z_1^{{\cal T},Q}\left(z_0, y_1\right)\right)\right| P\left(d y_1 \mid z_0\right)  \leq 2 \norm{f}_L d_{TV}(Q_n,Q).
\end{align}
Finally, combining this with (\ref{i}), (\ref{second}), (\ref{TV1}), and (\ref{last2}), we conclude:
\begin{align}\label{rhobl}
    d_{BL}(\eta,\eta^{{\cal T},Q_n}) \leq \sup_{f\in \operatorname{BL}(\Z)} (\norm{f}_\infty + 2 \norm{f}_L)\, d_{TV}(Q_n,Q)\leq 2 d_{TV}(Q_n,Q).
\end{align}

(ii) Let us now analyze the case where the observation kernels remain the same, but the transition kernels differ. 
Starting with:  \begin{align}\label{ii}
    &d_{BL}(\eta^{{\cal T}_n,Q_n},\eta^{{\cal T},Q_n})\nonumber\\
    &\qquad= \sup_{z_0\in \Z} \sup_{f \in \operatorname{BL}(\Z)} \left|\int_{{\cal P}(\mathbb{X})} f\left(z_1\right) \eta^{{\cal T}_n,Q_n}\left(d z_1 \mid z_0\right)-\int_{{\cal P}(\mathbb{X})} f\left(z_1\right) \eta^{{\cal T},Q_n}\left(d z_1 \mid z_0\right)\right|.\end{align}
    For a fixed \( f \) such that \( \norm{f}_\infty + \norm{f}_L \leq 1 \), following an argument similar to inequality (\ref{second}), we obtain:
    \begin{align}
    &
    \left|\int_{{\cal P}(\mathbb{X})} f\left(z_1\right) \eta^{{\cal T}_n,Q_n}\left(d z_1 \mid z_0\right)-\int_{{\cal P}(\mathbb{X})} f\left(z_1\right) \eta^{{\cal T},Q_n}\left(d z_1 \mid z_0\right)\right| \nonumber\\
    & =\left|\int_{\mathbb{Y}} f\left(z_1^{{\cal T}_n,Q_n}\left(z_0, y_1\right)\right) P^{{\cal T}_n,Q_n}\left(d y_1 \mid z_0\right)-\int_{\mathbb{Y}} f\left(z_1^{{\cal T},Q_n}\left(z_0,  y_1\right)\right) P^{{\cal T},Q_n}\left(d y_1 \mid z_0\right)\right|\nonumber \\
    &\leq \norm{f}_\infty \left\|P^{{\cal T}_n,Q_n}\left(dy_1 \mid z_0^{ } \right)-P^{{\cal T},Q_n}\left(dy_1 \mid z_0\right)\right\|_{T V} \nonumber\\
    &\quad+ \int_{\mathbb{Y}}\left|f\left(z_1^{{\cal T}_n,Q_n}\left(z_0^{ }, y_1\right)\right)-f\left(z_1^{{\cal T},Q_n}\left(z_0, y_1\right)\right)\right| P^{{\cal T},Q_n}\left(d y_1 \mid z_0\right). \label{secondi}
    \end{align}
    For the first term, using Lemma \ref{Py}, we have:
    \begin{align}\label{TV1i}
    &\normt{P^{{\cal T}_n,Q_n}\left(d y_1 \mid z_0\right)-P^{{\cal T},Q_n}\left(d y_1 \mid z_0\right)} \leq d_{TV}(\T_n,\T). 
    \end{align}
    
    For the second term in (\ref{secondi}), we write:
    \begin{align}\label{eq3i}
    &\int_{\mathbb{Y}}\left|f\left(z_1^{{\cal T}_n,Q_n}\left(z_0^{ }, y_1\right)\right)-f\left(z_1^{{\cal T},Q_n}\left(z_0, y_1\right)\right)\right| P^{{\cal T},Q_n}\left(d y_1 \mid z_0\right)\nonumber\\
    &\leq \norm{f}_L \int_{\mathbb{Y}}\rho_{BL} (z_1^{{\cal T}_n,Q_n}\left(z_0, y_1\right), z_1^{{\cal T},Q_n}\left(z_0, y_1\right)) P^{{\cal T},Q_n}\left(d y_1 \mid z_0\right)\nonumber\\
    &=\norm{f}_L \int_{\mathbb{Y}} \sup_{g \in \operatorname{BL}_1(\mathbb{X})}\Big(\int_{\mathbb{X}} g(x_1)z_1^{{\cal T}_n,Q_n}\left(z_0, y_1\right)(dx_1)\nonumber\\
&\qquad\qquad - \int_{\mathbb{X}} g(x_1)z_1^{{\cal T},Q_n}\left(z_0, y_1\right)(dx_1)\Big)P^{{\cal T},Q_n}\left(d y_1 \mid z_0\right).
    \end{align}
For the inner term, selecting a measurable family $\{g_{y}\} \subset \operatorname{BL}_1(\mathbb{X})$ of maximizers exactly as in part (i), we obtain
\begin{align}
    &\int_{\mathbb{Y}}\sup_{g \in \operatorname{BL}_1(\mathbb{X})}\left(\int_{\mathbb{X}} g(x_1)z_1^{{\cal T}_n,Q_n}\left(z_0 , y_1\right)(dx_1) - \int_{\mathbb{X}} g(x_1)z_1^{{\cal T},Q_n}\left(z_0, y_1\right)(dx_1)\right)P^{{\cal T},Q_n}\left(d y_1 \mid z_0\right)\nonumber\\
   &=\int_{\mathbb{Y}} \int_{\mathbb{X}} g_{y_1}(x_1)z_1^{{\cal T}_n,Q_n}\left(z_0 , y_1\right)(dx_1) P^{{\cal T},Q_n}(d y_1 \mid z_0)\nonumber\\
     &\qquad-\int_{\mathbb{Y}}\int_{\mathbb{X}} g_{y_1}(x_1)z_1^{{\cal T}_n,Q_n}\left(z_0, y_1\right)(dx_1) P^{{\cal T}_n,Q_n}(d y_1 \mid z_0)\nonumber\\
     &\quad+\int_{\mathbb{Y}} \int_{\mathbb{X}} g_{y_1}(x_1)z_1^{{\cal T}_n,Q_n}\left(z_0 , y_1\right)(dx_1) P^{{\cal T}_n,Q_n}(d y_1 \mid z_0)\nonumber\\
 &\qquad-\int_{\mathbb{Y}}\int_{\mathbb{X}} g_{y_1}(x_1)z_1^{{\cal T},Q_n}\left(z_0, y_1\right)(dx_1) P^{{\cal T},Q_n}(d y_1 \mid z_0 ).\label{gy1i}
\end{align}

For the first difference above, using the same reasoning as in (\ref{sup-1}), we have
    \begin{align}\label{sup-1i}
    &\int_{\mathbb{Y}} \int_{\mathbb{X}} g_{y_1}(x_1)z_1^{{\cal T}_n,Q_n}\left(z_0 , y_1\right)(dx_1) P^{{\cal T},Q_n}(d y_1 \mid z_0)\nonumber\\
    &\qquad-\int_{\mathbb{Y}}\int_{\mathbb{X}} g_{y_1}(x_1)z_1^{{\cal T}_n,Q_n}\left(z_0, y_1\right)(dx_1) P^{{\cal T}_n,Q_n}(d y_1 \mid z_0)\nonumber\\
    & \leq \normt{P^{{\cal T},Q_n}\left(d y_1 \mid z_0\right)-P^{{\cal T}_n,Q_n}\left(d y_1 \mid z_0\right)} \leq d_{TV}(\T_n,\T).  
    \end{align}
    
    For the second difference, we apply smoothing:
    \begin{align}\label{sup-2i}
    &\int_{\mathbb{Y}} \int_{\mathbb{X}} g_{y_1}(x_1)z_1^{{\cal T}_n,Q_n}\left(z_0 , y_1\right)(dx_1) P^{{\cal T}_n,Q_n}(d y_1 \mid z_0)\nonumber\\
    &\qquad-\int_{\mathbb{Y}}\int_{\mathbb{X}} g_{y_1}(x_1)z_1^{{\cal T},Q_n}\left(z_0, y_1\right)(dx_1) P^{{\cal T},Q_n}(d y_1 \mid z_0 ) \nonumber \\
    &=\int_{\mathbb{X}} \int_{\mathbb{Y}} g_{y_1}(x_1)Q_n\left(d y_1 \mid x_1\right) \mathcal{T}_n\left(d x_1 \mid z_0^{ }\right)-\int_{\mathbb{X}} \int_{\mathbb{Y}} g_{y_1}(x_1) Q_n\left(d y_1 \mid x_1\right) \mathcal{T}\left(d x_1 \mid z_0\right)\nonumber\\
    &=\int_{\mathbb{X}} \omega_n(x_1) \mathcal{T}_n\left(d x_1 \mid z_0^{ }\right)-\int_{\mathbb{X}}\omega_n(x_1) \mathcal{T}\left(d x_1 \mid z_0\right) \leq d_{TV}(\T_n,\T),
    \end{align}
    where $\omega_n(x_1)=\int_{\mathbb{Y}} g_{y_1}(x_1)Q_n\left(d y_1 \mid x_1\right)$, and the last inequality holds since $\norm{\omega_n}_\infty \leq 1$.

    By combining the results from (\ref{eq3i}), (\ref{gy1i}), (\ref{sup-1i}), and (\ref{sup-2i}), we obtain:
    \begin{align}\label{last2i}
    & \int_{\mathbb{Y}}\left|f\left(z_1^{{\cal T}_n,Q_n}\left(z_0^{ }, y_1\right)\right)-f\left(z_1^{{\cal T},Q_n}\left(z_0, y_1\right)\right)\right| P^{{\cal T},Q_n}\left(d y_1 \mid z_0\right) \leq 2 \norm{f}_L d_{TV}(\T_n,\T).
    \end{align}
   
Using inequalities (\ref{ii}), (\ref{secondi}), (\ref{TV1i}), and (\ref{last2i}), we can write:
 \begin{align}\label{rhobli}
        d_{BL}(\eta^{{\cal T}_n,Q_n},\eta^{{\cal T},Q_n}) \leq \sup_{f\in \operatorname{BL}(\Z)} (\norm{f}_\infty + 2 \norm{f}_L)\, d_{TV}(\T_n,\T)
        \leq 2 d_{TV}(\T_n,\T).
    \end{align}
The result holds by the triangle inequality.\hfill$\square$

\begin{lemma}\label{Py2}
Under Assumption \ref{channel_reg}, we have that
     $\normt{P^{{\cal T}_n,Q_n}\left(d y_1 \mid z_0\right)-P\left(d y_1 \mid z_0\right)}\leq  d_{TV}(Q_n, Q) +  L_Q\, d_{W_1}\left(\T_n,\T\right)$.
\end{lemma}

\begin{proof}
First we prove that
$\normt{P^{{\cal T},Q_n}\left(d y_1 \mid z_0\right)-P^{{\cal T}_n,Q_n}\left(d y_1 \mid z_0\right)}\leq L_Q\, d_{W_1}\left(\T_n,\T\right) $.
Denoting by $\T(dx_1|z):=\int \T(dx_1|x)z(dx)$, we start from the following term for some $\|g\|_\infty \leq 1$:
\begin{align*}
& \left|\int g(y_1)Q_n(dy_1|x_1)\T_n\left(d x_1 \mid z_0\right)  -\int g(y_1)Q_n(dy_1|x_1)\T\left(d x_1 \mid z_0\right)  \right|\\
&\leq \Big\|\int g(y_1)Q_n(dy_1|\cdot)\Big\|_{\Lip} W_1\left(\T(\cdot|z_0),\T_n(\cdot|z_0)\right) \leq L_Q\, d_{W_1}(\T,\T_n),
\end{align*}
where we used the following:
\begin{align*}
\left|\int g(y_1)Q_n(dy_1|x_1) - \int g(y_1)Q_n(dy_1|x'_1)\right|\leq \norm{g}_\infty \normt{Q_n(\cdot|x_1) - Q_n(\cdot|x_1')} \leq L_Q\, d(x_1,x'_1),
\end{align*}
which follows from Assumption \ref{channel_reg}. Taking the supremum over all $\|g\|_\infty \leq 1$ implies that 
$$ \normt{P^{{\cal T},Q_n}\left(d y_1 \mid z_0\right)-P^{{\cal T}_n,Q_n}\left(d y_1 \mid z_0\right)} \leq L_Q\, d_{W_1}(\T,\T_n).$$
The result follows from the triangle inequality and inequality (\ref{Q_b}).
\end{proof}

\noindent\textbf{Proof of Theorem \ref{main2}.} 
We prove
 $d_{BL}(\eta^{{\cal T},Q_n},\eta^{{\cal T}_n,Q_n}) \leq (2 L_Q + 1)\, d_{W_1}\left(\T_n,\T\right)$; the result then follows from the triangle inequality together with the bound $d_{BL}(\eta,\eta^{{\cal T},Q_n}) \leq 2 d_{TV}(Q_n,Q)$ established in part (i) of the proof of Theorem \ref{main1}. We first note that (\ref{TV1i}) and (\ref{sup-1i}) are bounded by $L_Q\, d_{W_1}\left(\T_n,\T\right)$ directly, using Lemma \ref{Py2}.

For (\ref{sup-2i}), we have that, for the selected $g_{y_1}\in \operatorname{BL}_1(\mathbb{X})$,
\begin{align*}
 |\omega_n(x_1)-\omega_n(x_1')| &=\left| \int_\mathbb{Y} g_{y_1}(x_1)Q_n(dy_1|x_1) - \int_\mathbb{Y} g_{y_1}(x_1')Q_n(dy_1|x_1') \right| \\ &\leq
\left| \int_\mathbb{Y} \bigl(g_{y_1}(x_1)-g_{y_1}(x_1')\bigr)Q_n(dy_1|x_1) \right|\\
&\quad + \left| \int_\mathbb{Y} g_{y_1}(x_1')Q_n(dy_1|x_1) - \int_\mathbb{Y} g_{y_1}(x_1')Q_n(dy_1|x_1') \right| \\ & \leq 
d(x_1,x_1')+ \normt{Q_n(\cdot|x_1)-Q_n(\cdot|x_1')} \leq (1+L_Q)\, d(x_1,x_1'),
\end{align*}
under Assumption \ref{channel_reg}, where we used that $\norm{g_{y}}_L \leq 1$ and $\norm{g_{y}}_\infty \leq 1$ for every $y$. Hence $\omega_n$ is $(1+L_Q)$-Lipschitz, and (\ref{sup-2i}) can instead be bounded by $(1+L_Q)\, d_{W_1}\left(\T_n,\T\right)$. Combining these bounds, for $f$ with $\norm{f}_\infty + \norm{f}_L \leq 1$,
\begin{align*}
&\left|\int f(z_1)\,\eta^{{\cal T}_n,Q_n}(dz_1|z_0) - \int f(z_1)\,\eta^{{\cal T},Q_n}(dz_1|z_0)\right|\\
&\qquad\leq \Bigl(\norm{f}_\infty L_Q + \norm{f}_L\,\bigl(L_Q + (1+L_Q)\bigr)\Bigr) d_{W_1}(\T_n,\T)\\
&\qquad\leq (2L_Q+1)\, d_{W_1}(\T_n,\T),
\end{align*}
which concludes the proof.
\hfill$\square$

\noindent\textbf{Proof of Theorem \ref{W1main_1}.}
\emph{(a) $\PP(\X)$ metrized by $\rho_{BL}$.} Let $f:\Z\to\bR$ satisfy $|f(z)-f(z')|\le\rho_{BL}(z,z')$. For $g\in\operatorname{BL}_1(\X)$ with $a:=\norm{g}_\infty$ we have $\norm{g}_L\le 1-a$, hence $\sup_{x,x'}|g(x)-g(x')|\le\min\{2a,(1-a)D\}\le \frac{2D}{D+2}$. Therefore the diameter of $(\Z,\rho_{BL})$ is at most $\frac{2D}{D+2}$, and for $c_f:=\inf_z f(z)+\frac{D}{D+2}$ we have $\norm{f-c_f}_\infty\le\frac{D}{D+2}$. Replacing $f$ by $f-c_f$ changes neither side of (\ref{uniform_metric_belief}), so we may assume $\norm{f}_\infty\le\frac{D}{D+2}$ and $\norm{f}_{L,\rho_{BL}}\le 1$. The estimates (\ref{prob_dist})--(\ref{last2}) in part (i) of the proof of Theorem \ref{main1} use only $\norm{f}_\infty$ and the $\rho_{BL}$-Lipschitz constant of $f$ (the test functions $g_y$ remain in $\operatorname{BL}_1(\X)$), and therefore:
\begin{align}\label{rhobl_W1a}
\dWBL(\eta,\eta^{{\cal T},Q_n}) \leq \sup_{f} \bigl(\norm{f}_\infty + 2 \norm{f}_{L,\rho_{BL}}\bigr)\,  d_{TV}(Q_n,Q)\leq \CD\, d_{TV}(Q_n,Q).
\end{align}
The same argument applied to part (ii) gives $\dWBL(\eta^{{\cal T},Q_n},\eta^{{\cal T}_n,Q_n}) \leq \CD\, d_{TV}(\T_n,\T)$, and (a) follows by the triangle inequality.

\emph{(b) $\PP(\X)$ metrized by $W_1$.} Since the diameter of $\X$ is $D$, the diameter of $(\Z,W_1)$ is also $D$. Hence, for $f$ with $\norm{f}_{\Lip}\le 1$, the constant $c_f:=\inf_z f(z)+D/2$ satisfies $\norm{f-c_f}_\infty\le D/2$, and as above we may assume $\norm{f}_\infty\le D/2$. Likewise, for $\mu,\nu\in\Z$,
\[
W_1(\mu,\nu)=\sup_{g\in{\cal G}}\Bigl(\int_\X g(x)\,\mu(dx)-\int_\X g(x)\,\nu(dx)\Bigr),\]
where \[{\cal G}:=\{g:\X\to\bR \mid \norm{g}_L\le 1,\ \norm{g}_\infty\le D/2\}.
\]
By the Arzel\`a--Ascoli theorem ${\cal G}$ is a compact subset of the separable Banach space $C(\X)$, and $(y,g)\mapsto\int g(x)\,w_y(dx)$ is measurable in $y$ and continuous in $g$; hence, by the measurable maximum theorem \cite[Theorem 18.19]{AlBo06}, there is a measurable selection $y\mapsto g_y\in{\cal G}$ of maximizers, and $(x,y)\mapsto g_y(x)$ is jointly measurable and bounded by $D/2$.

We now follow part (i) of the proof of Theorem \ref{main1} with the following modifications, which account for the fact that the functions $g_y$ are now bounded by $D/2$ rather than by $1$. In (\ref{prob_dist}), $\norm{f}_\infty\le D/2$ gives the bound $\frac{D}{2}\,d_{TV}(Q_n,Q)$. In (\ref{eq3}), since $f$ is $1$-Lipschitz with respect to $W_1$, $\rho_{BL}$ is replaced by $W_1$ and $\operatorname{BL}_1(\X)$ by ${\cal G}$. In (\ref{sup-1}), the map $y\mapsto\int g_y(x_1)\,z_1^{{\cal T},Q_n}(z_0,y)(dx_1)$ is bounded by $D/2$, so that (\ref{sup-1}) is at most $\frac{D}{2}\,d_{TV}(Q_n,Q)$. In (\ref{sup-2})--(\ref{w_s}), $|\omega_n(x_1)-\omega(x_1)|\le \sup_y\norm{g_y}_\infty\,\normt{Q_n(\cdot|x_1)-Q(\cdot|x_1)}\le\frac{D}{2}\,d_{TV}(Q_n,Q)$, so that (\ref{w_s}) is at most $\frac{D}{2}\,d_{TV}(Q_n,Q)$. Therefore
\begin{align}\label{rhobl_W1b_prev}
\int_\Y W_1\bigl(z_1^{{\cal T},Q_n}(z_0,y_1),z_1^{{\cal T},Q}(z_0,y_1)\bigr)\,P(dy_1|z_0)\le D\, d_{TV}(Q_n,Q),
\end{align}
and, since $f$ is $1$-Lipschitz with respect to $W_1$,
\begin{align}\label{rhobl_W1b}
\dWW(\eta,\eta^{{\cal T},Q_n}) \leq \Bigl(\frac{D}{2}+D\Bigr)\, d_{TV}(Q_n,Q)=\frac{3D}{2}\, d_{TV}(Q_n,Q).
\end{align}
In the same way, in part (ii) of the proof of Theorem \ref{main1} each of the terms (\ref{TV1i}), (\ref{sup-1i}) and (\ref{sup-2i}) is bounded by $\frac{D}{2}\,d_{TV}(\T_n,\T)$ (for (\ref{sup-2i}), note that $\norm{\omega_n}_\infty\le D/2$), which gives $\dWW(\eta^{{\cal T},Q_n},\eta^{{\cal T}_n,Q_n})\le\frac{3D}{2}\,d_{TV}(\T_n,\T)$, and (b) follows by the triangle inequality.
\hfill$\square$

\noindent\textbf{Proof of Theorem \ref{W1main_2}.}
\emph{(a)} With $f$ normalized as in the proof of Theorem \ref{W1main_1}(a) ($\norm{f}_\infty\le\frac{D}{D+2}$, $\norm{f}_{L,\rho_{BL}}\le 1$) and $g_y\in\operatorname{BL}_1(\X)$, we bound the terms in part (ii) of the proof of Theorem \ref{main1} under Assumption \ref{channel_reg} exactly as in the proof of Theorem \ref{main2}: (\ref{TV1i}) is at most $\norm{f}_\infty L_Q\, d_{W_1}(\T_n,\T)$ by Lemma \ref{Py2}, (\ref{sup-1i}) is at most $\norm{f}_{L,\rho_{BL}} L_Q\, d_{W_1}(\T_n,\T)$, and (\ref{sup-2i}) is at most $\norm{f}_{L,\rho_{BL}}(1+L_Q)\, d_{W_1}(\T_n,\T)$. Hence
\[
\dWBL(\eta^{{\cal T},Q_n},\eta^{{\cal T}_n,Q_n}) \leq \Bigl(\tfrac{D}{D+2} L_Q + L_Q + (1+L_Q)\Bigr) d_{W_1}(\T_n,\T) = \bigl(\CD L_Q + 1\bigr) d_{W_1}(\T_n,\T),
\]
and (a) follows from (\ref{rhobl_W1a}) and the triangle inequality.

\emph{(b)} With $f$ and $g_y\in{\cal G}$ normalized as in the proof of Theorem \ref{W1main_1}(b): by Lemma \ref{Py2}, (\ref{TV1i}) is at most $\norm{f}_\infty L_Q\, d_{W_1}(\T_n,\T)\le\frac{D}{2}L_Q\, d_{W_1}(\T_n,\T)$; since $y\mapsto\int g_y(x_1)\,z_1^{{\cal T}_n,Q_n}(z_0,y)(dx_1)$ is bounded by $D/2$, (\ref{sup-1i}) is at most $\frac{D}{2}L_Q\, d_{W_1}(\T_n,\T)$; and for (\ref{sup-2i}),
\begin{align*}
|\omega_n(x_1)-\omega_n(x_1')| &\le \Bigl|\int_\Y \bigl(g_{y_1}(x_1)-g_{y_1}(x_1')\bigr)Q_n(dy_1|x_1)\Bigr|\\
&\quad + \Bigl|\int_\Y g_{y_1}(x_1')\bigl(Q_n(dy_1|x_1)-Q_n(dy_1|x_1')\bigr)\Bigr|\\
&\le d(x_1,x_1') + \frac{D}{2}\normt{Q_n(\cdot|x_1)-Q_n(\cdot|x_1')} \le \Bigl(1+\frac{L_Q D}{2}\Bigr) d(x_1,x_1'),
\end{align*}
so that (\ref{sup-2i}) is at most $(1+\frac{L_QD}{2})\, d_{W_1}(\T_n,\T)$. Summing,
\[
\dWW(\eta^{{\cal T},Q_n},\eta^{{\cal T}_n,Q_n}) \leq \Bigl(\frac{D}{2}L_Q+\frac{D}{2}L_Q+1+\frac{D}{2}L_Q\Bigr) d_{W_1}(\T_n,\T) = \Bigl(\frac{3D}{2}L_Q + 1\Bigr) d_{W_1}(\T_n,\T),
\]
and (b) follows from (\ref{rhobl_W1b}) and the triangle inequality.
\hfill$\square$

\noindent\textbf{Proof of Lemma \ref{lem:weakW1}.}
Since $\X$ is compact, $W_1$ metrizes weak convergence on ${\cal P}(\X)$, and thus the map $(x,u) \mapsto \T(\cdot|x,u) \in ({\cal P}(\X), W_1)$ is continuous; by compactness of $\X \times \U$ it is uniformly continuous, i.e., there is a modulus $\omega_{\T}$ with $\omega_{\T}(r) \downarrow 0$ as $r \downarrow 0$ such that $W_1(\T(\cdot|x,u), \T(\cdot|x',u)) \leq \omega_{\T}(d(x,x'))$. Repeating the two-step argument in the proof of Lemma \ref{quantizedT} with $\theta\, d(x',x)$ replaced by $\omega_\T(d(x',x))$ gives $d_{W_1}(\T,\T_n) \leq \omega_{\T}(L_{\X_n}) + L_{\X_n} \to 0$.
\hfill$\square$

\section{Proof of Theorem \ref{WassersteinCont2}}\label{app:wasserstein}

The control $u$ is fixed throughout and suppressed from the notation. We write
\[
{\cal T}(dx_1|z) := \int_{\X} {\cal T}(dx_1|x_0)\, z(dx_0), \qquad
P(dy_1|z) := \int_{\X} Q(dy_1|x_1)\, {\cal T}(dx_1|z),
\]
for the predictor and the one-step observation law associated with a prior $z \in {\cal P}(\X)$, and $z_1(z,y_1) := P(X_1 \in \cdot \mid z, Y_1 = y_1)$ for the one-step filter. 

We claim that for all $z,z' \in {\cal P}(\X)$,
\begin{align}
W_1\bigl({\cal T}(\cdot|z),{\cal T}(\cdot|z')\bigr) &\leq \theta\, W_1(z,z'), \label{eq:wc_pred}\\
\normt{P(\cdot|z)-P(\cdot|z')} &\leq \theta L_Q\, W_1(z,z'). \label{eq:wc_obs}
\end{align}
For \eqref{eq:wc_pred}, take $h : \X \to \bR$ with $\norm{h}_L \leq 1$ and set $H(x_0) := \int_{\X} h(x_1) {\cal T}(dx_1|x_0)$. By Assumption \ref{weakTtvQweakTtvQ2}(ii),
\[
|H(x_0)-H(x_0')| \leq W_1\bigl({\cal T}(\cdot|x_0),{\cal T}(\cdot|x_0')\bigr) \leq \theta\, d(x_0,x_0'),
\]
so $\norm{H}_L \leq \theta$, and therefore
\(
\int h(x_1)\,{\cal T}(dx_1|z) - \int h(x_1)\,{\cal T}(dx_1|z') = \int H(x_0)\,z(dx_0) - \int H(x_0)\,z'(dx_0) \leq \theta\, W_1(z,z');
\)
taking the supremum over such $h$ gives \eqref{eq:wc_pred}. For \eqref{eq:wc_obs}, take $g:\Y \to \bR$ with $\norm{g}_\infty \leq 1$ and set $\tilde{g}(x_1) := \int_{\Y} g(y_1) Q(dy_1|x_1)$. By Assumption \ref{weakTtvQweakTtvQ2}(iii),
\[
|\tilde{g}(x_1)-\tilde{g}(x_1')| \leq \norm{g}_\infty \normt{Q(\cdot|x_1)-Q(\cdot|x_1')} \leq L_Q\, d(x_1,x_1'),
\]
so $\norm{\tilde{g}}_L \leq L_Q$, and hence, by \eqref{eq:wc_pred},
\begin{align*}
\int_{\Y} g(y_1)\,P(dy_1|z) - \int_{\Y} g(y_1)\,P(dy_1|z')
&= \int_{\X} \tilde{g}(x_1)\,{\cal T}(dx_1|z) - \int_{\X} \tilde{g}(x_1)\,{\cal T}(dx_1|z')\\
&\leq L_Q\, W_1\bigl({\cal T}(\cdot|z),{\cal T}(\cdot|z')\bigr) \leq \theta L_Q\, W_1(z,z').
\end{align*}
Taking the supremum over such $g$ gives \eqref{eq:wc_obs}.

Now let $f : {\cal P}(\X) \to \bR$ with $\norm{f}_L \leq 1$, where ${\cal P}(\X)$ is equipped with the $W_1$ metric. Since $\X$ is compact with diameter $D$, the diameter of $({\cal P}(\X),W_1)$ is also $D$; as adding a constant to $f$ changes neither side of the identity below, we may assume without loss of generality that $\norm{f}_\infty \leq D/2$ (the normalization used in the proof of Theorem \ref{W1main_1}). Then
\begin{align}
&\int f(z_1)\,\eta(dz_1|z_0) - \int f(z_1)\,\eta(dz_1|z_0')\nonumber\\
&\qquad= \int_{\Y} f\bigl(z_1(z_0,y)\bigr) P(dy|z_0) - \int_{\Y} f\bigl(z_1(z_0',y)\bigr) P(dy|z_0')\nonumber\\
&\qquad= \underbrace{\int_{\Y} \Bigl[f\bigl(z_1(z_0,y)\bigr) - f\bigl(z_1(z_0',y)\bigr)\Bigr] P(dy|z_0)}_{=:(A)}
+ \underbrace{\int_{\Y} f\bigl(z_1(z_0',y)\bigr)\bigl[P(dy|z_0)-P(dy|z_0')\bigr]}_{=:(B)}.\label{eq:wc_AB}
\end{align}
Since $\norm{f}_\infty \leq D/2$, the inequality \eqref{eq:wc_obs} gives
\begin{align}\label{eq:wc_B}
(B) \;\leq\; \frac{D}{2}\, \normt{P(\cdot|z_0)-P(\cdot|z_0')} \;\leq\; \frac{\theta L_Q D}{2}\, W_1(z_0,z_0').
\end{align}
Since $\norm{f}_L \leq 1$,
\begin{align*}
(A) &\leq \int_{\Y} W_1\bigl(z_1(z_0,y), z_1(z_0',y)\bigr) P(dy|z_0)\\
&= \int_{\Y} \sup_{g \in {\cal G}} \left(\int_{\X} g(x_1)\,z_1(z_0,y)(dx_1) - \int_{\X} g(x_1)\,z_1(z_0',y)(dx_1)\right) P(dy|z_0),
\end{align*}
where ${\cal G} := \{g : \X \to \bR \mid \norm{g}_L \leq 1,\ \norm{g}_\infty \leq D/2\}$. The class ${\cal G}$ is uniformly bounded and equicontinuous on the compact space $\X$, hence compact in the sup-norm by the Arzel\`a--Ascoli theorem, so the supremum is attained for every $y$. Since $(y,g)\mapsto \int g(x_1)\,\bigl(z_1(z_0,y)-z_1(z_0',y)\bigr)(dx_1)$ is measurable in $y$ and continuous in $g$, the measurable maximum theorem \cite[Theorem 18.19]{AlBo06} (exactly as in the proof of Theorem \ref{W1main_1}(b)) then provides a measurable map $y \mapsto g_y \in {\cal G}$ attaining it; since evaluation is jointly continuous, $(x,y) \mapsto g_y(x)$ is jointly measurable and bounded by $D/2$. Adding and subtracting a term, we obtain
\begin{align}
(A) &\leq \int_{\Y}\!\int_{\X} g_y(x_1) z_1(z_0,y)(dx_1) P(dy|z_0) - \int_{\Y}\!\int_{\X} g_y(x_1) z_1(z_0',y)(dx_1) P(dy|z_0)\nonumber\\
&= \underbrace{\int_{\Y}\!\int_{\X} g_y(x_1) z_1(z_0,y)(dx_1) P(dy|z_0) - \int_{\Y}\!\int_{\X} g_y(x_1) z_1(z_0',y)(dx_1) P(dy|z_0')}_{=:(A_1)}\nonumber\\
&\qquad + \underbrace{\int_{\Y}\!\int_{\X} g_y(x_1) z_1(z_0',y)(dx_1)\bigl[P(dy|z_0')-P(dy|z_0)\bigr]}_{=:(A_2)}.\label{eq:wc_A12}
\end{align}
The map $y \mapsto \int_{\X} g_y(x_1) z_1(z_0',y)(dx_1)$ is measurable and bounded in absolute value by $D/2$, so \eqref{eq:wc_obs} yields
\begin{align}\label{eq:wc_A2}
(A_2) \;\leq\; \frac{D}{2}\, \normt{P(\cdot|z_0)-P(\cdot|z_0')} \;\leq\; \frac{\theta L_Q D}{2}\, W_1(z_0,z_0').
\end{align}
For $(A_1)$, by smoothing we have
\(
(A_1) = \int_{\X} \omega(x_1)\,{\cal T}(dx_1|z_0) - \int_{\X} \omega(x_1)\,{\cal T}(dx_1|z_0')
\) where \(
\omega(x_1) := \int_{\Y} g_{y_1}(x_1)\, Q(dy_1|x_1)
\).
Moreover $\omega$ is Lipschitz: for $x_1,x_1' \in \X$,
\begin{align*}
|\omega(x_1)-\omega(x_1')| &\leq \left|\int_{\Y} \bigl(g_{y_1}(x_1)-g_{y_1}(x_1')\bigr) Q(dy_1|x_1)\right|
+ \left|\int_{\Y} g_{y_1}(x_1')\bigl[Q(dy_1|x_1) - Q(dy_1|x_1')\bigr]\right|\\
&\leq d(x_1,x_1') + \frac{D}{2}\,\normt{Q(\cdot|x_1)-Q(\cdot|x_1')}
\;\leq\; \Bigl(1 + \frac{L_Q D}{2}\Bigr) d(x_1,x_1'),
\end{align*}
where we used $\norm{g_y}_L \leq 1$ and $\norm{g_y}_\infty \leq D/2$ for every $y$, together with Assumption \ref{weakTtvQweakTtvQ2}(iii). Hence, by \eqref{eq:wc_pred},
\begin{align}\label{eq:wc_A1}
(A_1) \;\leq\; \Bigl(1+\frac{L_Q D}{2}\Bigr) W_1\bigl({\cal T}(\cdot|z_0),{\cal T}(\cdot|z_0')\bigr)
\;\leq\; \theta\Bigl(1+\frac{L_Q D}{2}\Bigr) W_1(z_0,z_0').
\end{align}

Combining \eqref{eq:wc_AB}, \eqref{eq:wc_B}, \eqref{eq:wc_A12}, \eqref{eq:wc_A2} and \eqref{eq:wc_A1},
\begin{align*}
\int f(z_1)\,\eta(dz_1|z_0) - \int f(z_1)\,\eta(dz_1|z_0')
&\leq \left[\theta\Bigl(1+\frac{L_Q D}{2}\Bigr) + \frac{\theta L_Q D}{2} + \frac{\theta L_Q D}{2}\right] W_1(z_0,z_0')\\
&= \Bigl(\theta + \frac{3\theta L_Q D}{2}\Bigr) W_1(z_0,z_0').
\end{align*}
Taking the supremum over all $f$ with $\norm{f}_L \leq 1$ gives
$\WW\bigl(\eta(\cdot|z_0),\eta(\cdot|z_0')\bigr) \leq \bigl(\theta + \tfrac{3\theta L_Q D}{2}\bigr) W_1(z_0,z_0')$, as claimed.
\hfill$\square$

\section{Proof Sketches for the MDP Results of Section \ref{robustness}}\label{app:mdp_proofs}

We sketch, in the notation of the belief-MDP $(\Z,\U,\eta,\tilde c)$, the proofs of the results of Section \ref{robustness} that are adapted from \cite{zhou2024robustness}. Throughout, $\X$ is compact, $\Z$ is metrized by $W_1$, $\norm{\cdot}_{\Lip}$ refers to $W_1$, and we write $\eta_n := \eta^{{\cal T}_n,Q_n}$, $v := J^*_\beta(\eta)$, $v_n := J^*_\beta(\eta_n)$ and $\gamma_n^* := \gamma^*_{\eta_n}$. We recall two standard facts. First, if $\tilde c$ is bounded and continuous, $\U$ is compact and $\eta$ is weak Feller, then $v$ is the unique bounded continuous solution of the discounted cost optimality equation (DCOE)
\begin{align}\label{eq:app_dcoe}
v(z) = \inf_{u\in\U}\Bigl\{\tilde c(z,u) + \beta\int v(z_1)\,\eta(dz_1|z,u)\Bigr\},
\end{align}
and a measurable minimizer $\gamma^*_\eta$ exists and is optimal \cite[Theorem 4.2.3]{HernandezLermaMCP}. Second, for a fixed measurable stationary policy $\gamma$ and bounded measurable $\tilde c$, $J_\beta(\eta,\gamma)$ is the unique bounded measurable solution of the policy evaluation equation
\begin{align}\label{eq:app_pe}
J_\beta(\eta,\gamma)(z) = \tilde c(z,\gamma(z)) + \beta\int J_\beta(\eta,\gamma)(z_1)\,\eta(dz_1|z,\gamma(z)),
\end{align}
since the right-hand side defines a $\beta$-contraction on bounded measurable functions. Finally, for a function $\varphi$ on $\Z$ with $\norm{\varphi}_{\Lip}<\infty$, the definition (\ref{uniform_metric_belief}) of $\dWW$ gives
\begin{align}\label{eq:app_lip_dw1}
\sup_{z,u}\Bigl|\int\varphi(z_1)\,\eta(dz_1|z,u) - \int\varphi(z_1)\,\eta_n(dz_1|z,u)\Bigr| \le \norm{\varphi}_{\Lip}\, \dWW(\eta,\eta_n),
\end{align}
which is \cite[Lemma 1.1]{zhou2024robustness} on the belief space. The Lipschitz constant and the distance must refer to the same metric on $\Z$.

\medskip
\noindent\textbf{Proof sketch of Theorem \ref{upperbound1proof}} (\cite[Theorem 2.4]{zhou2024robustness}). By (\ref{eq:app_dcoe}) for $\eta$ and for $\eta_n$, and since $|\inf_u a_u - \inf_u b_u|\le\sup_u|a_u-b_u|$, for every $z$
\begin{align*}
|v(z)-v_n(z)| &\le \beta\sup_u\Bigl|\int v(z_1)\,\eta(dz_1|z,u) - \int v_n(z_1)\,\eta_n(dz_1|z,u)\Bigr|\\
&\le \beta\sup_u\Bigl|\int v(z_1)\,\eta(dz_1|z,u) - \int v(z_1)\,\eta_n(dz_1|z,u)\Bigr|\\
&\qquad + \beta\sup_u\Bigl|\int \bigl(v(z_1)-v_n(z_1)\bigr)\,\eta_n(dz_1|z,u)\Bigr|\\
&\le \beta\norm{v}_{\Lip}\,\dWW(\eta,\eta_n) + \beta\norm{v-v_n}_\infty,
\end{align*}
where (\ref{eq:app_lip_dw1}) was used. Taking the supremum over $z$  gives $\norm{v-v_n}_\infty\le\frac{\beta}{1-\beta}\norm{v}_{\Lip}\,\dWW(\eta,\eta_n)$.  \hfill$\square$

\medskip
\noindent\textbf{Proof of Lemma \ref{lem:fixedpolicy}.} Write $\tilde v := J_\beta(\eta,\gamma)$ and $\tilde v_n := J_\beta(\eta_n,\gamma)$. Subtracting the two policy evaluation equations (\ref{eq:app_pe}) (with the same policy $\gamma$ and the same cost), for every $z$,
\begin{align*}
|\tilde v(z)-\tilde v_n(z)| &\le \beta\Bigl|\int \tilde v_n(z_1)\,\eta(dz_1|z,\gamma(z)) - \int \tilde v_n(z_1)\,\eta_n(dz_1|z,\gamma(z))\Bigr|\\
&\qquad + \beta\Bigl|\int\bigl(\tilde v(z_1)-\tilde v_n(z_1)\bigr)\,\eta(dz_1|z,\gamma(z))\Bigr|\\
&\le \beta \sup_{z'}\Bigl|\int \tilde v_n(z_1)\,\bigl(\eta(dz_1|z',\gamma(z')) - \eta_n(dz_1|z',\gamma(z'))\bigr)\Bigr|\\
&\qquad + \beta\norm{\tilde v-\tilde v_n}_\infty .
\end{align*}
Taking the supremum over $z$ and rearranging gives (i). If $\tilde v_n$ is Lipschitz, the first term is bounded by $\beta\norm{\tilde v_n}_{\Lip}\,\dWW(\eta,\eta_n)$ by (\ref{eq:app_lip_dw1}), which gives (ii). \hfill$\square$

\medskip
\noindent\textbf{Proof of Theorem \ref{upperbound2proof}.} Since $\gamma_n^*$ is optimal for $\eta_n$, $J_\beta(\eta_n,\gamma_n^*) = v_n$. By the triangle inequality,
\[
\norm{J_\beta(\eta,\gamma_n^*) - v}_\infty \le \norm{J_\beta(\eta,\gamma_n^*) - J_\beta(\eta_n,\gamma_n^*)}_\infty + \norm{v_n - v}_\infty .
\]
By Lemma \ref{lem:fixedpolicy}(ii) and Theorem \ref{upperbound1proof}, respectively,
\begin{align*}
\norm{J_\beta(\eta,\gamma_n^*) - J_\beta(\eta_n,\gamma_n^*)}_\infty &\le \frac{\beta}{1-\beta}\norm{v_n}_{\Lip}\,\dWW(\eta,\eta_n), \\
\norm{v_n - v}_\infty &\le \frac{\beta}{1-\beta}\norm{v}_{\Lip}\,\dWW(\eta,\eta_n).
\end{align*}
\hfill$\square$

\medskip
\noindent\textbf{Proof of Theorem \ref{upperbound3proof}.} With $\tilde v := J_\beta(\eta,\gamma_n^*)$, Lemma \ref{lem:fixedpolicy}(i) (applied with $\gamma=\gamma_n^*$, for which $\tilde v_n=v_n$) gives
\begin{align*}
(1-\beta)\norm{\tilde v - v_n}_\infty &\le \beta\sup_{z}\Bigl|\int v_n(z_1)\,\bigl(\eta - \eta_n\bigr)(dz_1|z,\gamma_n^*(z))\Bigr|\\
&\le \beta\sup_{z,u}\Bigl|\int v(z_1)\,(\eta-\eta_n)(dz_1|z,u)\Bigr| + 2\beta\norm{v-v_n}_\infty .
\end{align*}
 By (\ref{eq:app_lip_dw1}) and Theorem \ref{upperbound1proof},
\begin{align*}
\norm{\tilde v - v_n}_\infty &\le \frac{\beta}{1-\beta}\norm{v}_{\Lip}\,\dWW(\eta,\eta_n) + \frac{2\beta}{1-\beta}\cdot\frac{\beta}{1-\beta}\norm{v}_{\Lip}\,\dWW(\eta,\eta_n)\\
&= \frac{\beta+\beta^2}{(1-\beta)^2}\norm{v}_{\Lip}\,\dWW(\eta,\eta_n).
\end{align*}
Adding $\norm{v_n - v}_\infty\le\frac{\beta}{1-\beta}\norm{v}_{\Lip}\dWW(\eta,\eta_n)$ (Theorem \ref{upperbound1proof})  gives the claim. \hfill$\square$

\medskip
\noindent\textbf{Proof sketch of Theorems \ref{vanishingproof1} and \ref{thm:avg_policy}.} Fix $z^\circ\in\Z$ and set $h_\beta := v_\beta - v_\beta(z^\circ)$, where $v_\beta := J^*_\beta(\eta)$. By (iv), for $\beta\in[\beta^*,1)$ the family $\{h_\beta\}$ is uniformly $L$-Lipschitz and, $(\Z,W_1)$ having diameter $D$, uniformly bounded by $LD$; moreover $|(1-\beta)v_\beta|\le\norm{c}_\infty$. By the Arzel\`a--Ascoli theorem there is a sequence $\beta_k\uparrow 1$ along which $h_{\beta_k}\to h$ uniformly and $(1-\beta_k)v_{\beta_k}(z^\circ)\to\rho^*$, with $h$ $L$-Lipschitz. Writing the DCOE (\ref{eq:app_dcoe}) as $(1-\beta)v_\beta(z^\circ) + h_\beta(z) = \inf_u\{\tilde c(z,u) + \beta\int h_\beta(z_1)\,\eta(dz_1|z,u)\}$ and letting $k\to\infty$ (uniform convergence and $\beta_k\to 1$) gives the ACOE (\ref{acoe}) for $(h,\rho^*)$; a measurable minimizer $\gamma^*$ exists since $\U$ is compact and the bracket is continuous in $u$ by the weak Feller property, and the verification theorem \cite[Theorem 5.2.4]{HernandezLermaMCP} (applicable since $h$ is bounded) gives $J^*_\infty(\eta)\equiv\rho^*$. The same applies to $\eta_n$ with $L_n$, yielding $(h_n,\rho^*_n,\gamma^*_n)$; by \cite[Remark 2.5]{zhou2024robustness} a common sequence $\beta_k$ may be used for both models.

Next, by Theorem \ref{upperbound1proof} and (iv), for $\beta\in[\beta^*,1)$,
\(
(1-\beta)\norm{v_\beta - v_{n,\beta}}_\infty \le \beta L\, \dWW(\eta,\eta_n).
\)
Letting $\beta=\beta_k\to 1$ along this sequence, where $(1-\beta_k)v_{\beta_k}\to\rho^*$ and $(1-\beta_k)v_{n,\beta_k}\to\rho^*_n$ uniformly, gives $|\rho^*-\rho^*_n|\le L\,\dWW(\eta,\eta_n)$.

Finally, let $\gamma_n^*$ be the ACOE-minimizing selector of the approximate model, so that $\rho^*_n + h_n(z) = \tilde c(z,\gamma_n^*(z)) + \int h_n(z_1)\,\eta_n(dz_1|z,\gamma_n^*(z))$ for every $z$. Then, for every $z$,
\begin{align*}
\tilde c(z,\gamma_n^*(z)) + \int h_n(z_1)\,\eta(dz_1|z,\gamma_n^*(z))
&= \rho^*_n + h_n(z) + \int h_n(z_1)\,\bigl(\eta-\eta_n\bigr)(dz_1|z,\gamma_n^*(z))\\
&\le \rho^* + |\rho^*_n-\rho^*| + L_n\,\dWW(\eta,\eta_n) + h_n(z),
\end{align*}
by (\ref{eq:app_lip_dw1}) with $\varphi=h_n$. Hence the bounded measurable function $h_n$ and the constant $g := \rho^* + |\rho^*_n-\rho^*| + L_n\,\dWW(\eta,\eta_n)$ satisfy $g + h_n(z)\ge \tilde c(z,\gamma_n^*(z)) + \int h_n(z_1)\,\eta(dz_1|z,\gamma_n^*(z))$ for all $z$, and the comparison lemma \cite[Lemma 5.2.5(a)]{HernandezLermaMCP} (iterate the inequality $N$ times, divide by $N$ and use boundedness of $h_n$) yields $J_\infty(\eta,\gamma_n^*)(z)\le g$ for all $z$. Since $J_\infty(\eta,\gamma_n^*)\ge J^*_\infty(\eta)\equiv\rho^*$, we conclude with the bound on $|\rho^*-\rho^*_n|$ that
\[
0\le J_\infty(\eta,\gamma_n^*)(z) - \rho^* \le |\rho^*_n-\rho^*| + L_n\,\dWW(\eta,\eta_n) \le (L+L_n)\,\dWW(\eta,\eta_n).
\]
 \hfill$\square$

\section{Proofs of Lemmas \ref{lem:filter_stability}, \ref{lem:mismatch_bounds} and \ref{lem:filter_mismatch}}\label{app:filter_mismatch}

The control $u$ is fixed and suppressed where possible. For a channel $Q'$ and $\mu\in\PP(\X)$ we write $P^{Q'}_\mu(dy):=\int_\X Q'(dy|x)\,\mu(dx)$, and for a kernel $\T'$ we write $z\T'(\cdot):=\int_\X\T'(\cdot|x)\,z(dx)$ for the predictor. Then $F^{\T',Q'}(z,y)=\Phi^{Q'}(z\T',y)$ and $P^{\T',Q'}(dy|z)=P^{Q'}_{z\T'}(dy)$; in particular $H(dy|z)=P^{Q}_{z\T}(dy)$.
Part (ii) of the following lemma is \cite[Lemma 3.2]{mcdonald2020exponential}, where it is the main step of the filter stability proof; we recall the short proof. As in \cite{mcdonald2020exponential}, we assume that the channels are dominated: $Q'(dy|x)=g'(x,y)\,\lambda(dy)$ for a reference probability measure $\lambda$ on $\Y$ and every $x\in\X$, and we take the Bayes formula as the fixed version of $\Phi^{Q'}$. (In general, the same computation applies with $g'(x,y)\lambda(dy)$ replaced by the disintegrations of $\mu(dx)Q'(dy|x)$ and $\nu(dx)Q'(dy|x)$ with respect to $\frac12(P^{Q'}_\mu+P^{Q'}_\nu)$.)
\begin{lemma}\label{lem:bayesTV}
Let $Q'$ be a channel and $\mu,\nu\in\PP(\X)$. Then
\begin{itemize}
\item[(i)] $\normt{P^{Q'}_\mu-P^{Q'}_\nu}\le(1-\delta(Q'))\normt{\mu-\nu}$;
\item[(ii)] $\displaystyle\int_\Y\normt{\Phi^{Q'}(\mu,y)-\Phi^{Q'}(\nu,y)}\,P^{Q'}_\mu(dy)\le\normt{\mu-\nu}+\normt{P^{Q'}_\mu-P^{Q'}_\nu}$, which is at most $(2-\delta(Q'))\normt{\mu-\nu}$.
\end{itemize}
\end{lemma}
\begin{proof}
(i) is Dobrushin's theorem \cite{dobrushin1956central}. For (ii), let $N^\mu(y):=\int_\X g'(x,y)\,\mu(dx)$, so that $P^{Q'}_\mu(dy)=N^\mu(y)\,\lambda(dy)$ and $\Phi^{Q'}(\mu,y)(dx)=g'(x,y)\,\mu(dx)/N^\mu(y)$ on $\{N^\mu>0\}$, a set of full $P^{Q'}_\mu$-measure; similarly for $\nu$. Fix $y$ with $N^\mu(y)>0$ and $N^\nu(y)>0$. For measurable $f$ with $\norm{f}_\infty\le1$,
\begin{align*}
\int_\X f\,d\Phi^{Q'}(\mu,y)-\int_\X f\,d\Phi^{Q'}(\nu,y)
&=\frac{1}{N^\mu(y)}\int_\X f(x)g'(x,y)\,(\mu-\nu)(dx)\\
&\qquad+\Bigl(\frac{1}{N^\mu(y)}-\frac{1}{N^\nu(y)}\Bigr)\int_\X f(x)g'(x,y)\,\nu(dx),
\end{align*}
where $\bigl|\int_\X f(x)g'(x,y)\,\nu(dx)\bigr|\le N^\nu(y)$. Taking the supremum over $f$,
\begin{align}
\normt{\Phi^{Q'}(\mu,y)-\Phi^{Q'}(\nu,y)}
&\le\frac{1}{N^\mu(y)}\int_\X g'(x,y)\,|\mu-\nu|(dx)+\frac{|N^\nu(y)-N^\mu(y)|}{N^\mu(y)N^\nu(y)}\,N^\nu(y)\nonumber\\
&=\frac{1}{N^\mu(y)}\Bigl(\int_\X g'(x,y)\,|\mu-\nu|(dx)+|N^\mu(y)-N^\nu(y)|\Bigr),\label{eq:bayes_pointwise}
\end{align}
where $|\mu-\nu|$ is the total variation measure of $\mu-\nu$. If $N^\mu(y)>0=N^\nu(y)$, then $g'(\cdot,y)=0$ $\nu$-a.e., so the right-hand side of (\ref{eq:bayes_pointwise}) equals $2$ and (\ref{eq:bayes_pointwise}) holds as well. Multiplying (\ref{eq:bayes_pointwise}) by $N^\mu(y)$, integrating with respect to $\lambda$ and using $\int_\Y g'(x,y)\,\lambda(dy)=1$ for every $x$,
\begin{align*}
\int_\Y\normt{\Phi^{Q'}(\mu,y)-\Phi^{Q'}(\nu,y)}\,P^{Q'}_\mu(dy)
&\le\int_\X\Bigl(\int_\Y g'(x,y)\,\lambda(dy)\Bigr)|\mu-\nu|(dx)\\
&\qquad+\int_\Y|N^\mu(y)-N^\nu(y)|\,\lambda(dy)\\
&=\normt{\mu-\nu}+\normt{P^{Q'}_\mu-P^{Q'}_\nu}.
\end{align*}
The last inequality in (ii) follows from (i).
\end{proof}

\noindent\textbf{Proof of Lemma \ref{lem:mismatch_bounds}.}
The bound on $\epsilon_{0,n}$ is (\ref{rhobl_W1b_prev}) with $\T$ the identity kernel, since the proof of (\ref{rhobl_W1b_prev}) does not depend on $\T$. For $\epsilon_n$, fix $z\in\PP(\X)$ and $u\in\U$. By the triangle inequality,
\begin{align*}
\int_\Y W_1\bigl(F(z,y),F^{\T_n,Q_n}(z,y)\bigr)H(dy|z)
&\le\int_\Y W_1\bigl(F^{\T,Q}(z,y),F^{\T,Q_n}(z,y)\bigr)P^{\T,Q}(dy|z)\\
&\quad+\int_\Y W_1\bigl(F^{\T,Q_n}(z,y),F^{\T_n,Q_n}(z,y)\bigr)P^{\T,Q}(dy|z).
\end{align*}
The first integral is at most $D\,d_{TV}(Q_n,Q)$ by (\ref{rhobl_W1b_prev}). In the second, the integrand takes values in $[0,D]$, so replacing $P^{\T,Q}(dy|z)$ by $P^{\T,Q_n}(dy|z)$ costs at most $\frac D2\normt{P^{\T,Q_n}(\cdot|z)-P^{\T,Q}(\cdot|z)}\le\frac D2\,d_{TV}(Q_n,Q)$ by Lemma \ref{Py}, and the resulting integral is the sum of the terms (\ref{sup-1i}) and (\ref{sup-2i}) with $g_y\in{\cal G}$, which is at most $D\,d_{TV}(\T_n,\T)$ by the proof of Theorem \ref{W1main_1}(b), and at most $(1+DL_Q)\,d_{W_1}(\T_n,\T)$ under Assumption \ref{channel_reg} by the proof of Theorem \ref{W1main_2}(b). Taking the supremum over $z$ and $u$ gives the bounds on $\epsilon_n$. The last claim follows since $D\le\frac{3D}{2}$ and $1+DL_Q\le1+\frac{3D}{2}L_Q$.
\hfill$\square$

\medskip
\noindent\textbf{Proof of Lemmas \ref{lem:filter_stability} and \ref{lem:filter_mismatch}.}
We first show that for all $z,\pi\in\PP(\X)$ and $u\in\U$,
\begin{align}
\int_\Y W_1\bigl(F(z,u,y),F(\pi,u,y)\bigr)H(dy|z,u)&\le K_2' \,W_1(z,\pi),\label{eq:fm_pathwise}\\
\frac D2\normt{H(\cdot|z,u)-H(\cdot|\pi,u)}&\le (K_2-K_2')\,W_1(z,\pi),\label{eq:fm_obs}
\end{align}
with $K_2':=\frac{\alpha D}{2}(2-\delta(Q))$ under Assumption \ref{main_assumption} and $K_2':=\theta+\theta L_QD$ under Assumption \ref{weakTtvQweakTtvQ2}, where $K_2$ is the constant of Remark \ref{rem:tradeoff}. Under Assumption \ref{weakTtvQweakTtvQ2}, (\ref{eq:fm_pathwise}) is the bound on $(A)$ in Appendix \ref{app:wasserstein}, namely (\ref{eq:wc_A1}) and (\ref{eq:wc_A2}), and (\ref{eq:fm_obs}) is (\ref{eq:wc_obs}). Under Assumption \ref{main_assumption}, they follow from Lemma \ref{lem:bayesTV} applied to $z\T_u$ and $\pi\T_u$, where $\T_u:=\T(\cdot|\cdot,u)$, since $\normt{z\T_u-\pi\T_u}\le\alpha\,W_1(z,\pi)$ and $W_1\le\frac D2\normt{\cdot}$ (Remark \ref{rem:metrics}(ii)).

We prove Lemma \ref{lem:filter_mismatch}; Lemma \ref{lem:filter_stability} is the same argument with $F^{\T_n,Q_n}$ and $\Phi^{Q_n}$ replaced by $F$ and $\Phi^{Q}$, in which case the second integrals below vanish. Let $t\ge0$ and let ${\cal F}_t:=\sigma(Y_{[0,t]})$. Since $U_t$ is ${\cal F}_t$-measurable, the conditional law of $Y_{t+1}$ given ${\cal F}_t$ is $H(\cdot|z_t,U_t)$, and $z_t,\pi^n_t$ are ${\cal F}_t$-measurable. Hence, with $u=U_t$,
\begin{align*}
&E^{\mu_0}\bigl[W_1(z_{t+1},\pi^n_{t+1})\,\big|\,{\cal F}_t\bigr]
=\int_\Y W_1\bigl(F(z_t,u,y),F^{\T_n,Q_n}(\pi^n_t,u,y)\bigr)H(dy|z_t,u)\\
&\quad\le\int_\Y W_1\bigl(F(z_t,u,y),F(\pi^n_t,u,y)\bigr)H(dy|z_t,u)\\
&\qquad+\int_\Y W_1\bigl(F(\pi^n_t,u,y),F^{\T_n,Q_n}(\pi^n_t,u,y)\bigr)H(dy|z_t,u)\\
&\quad\le K_2'\,W_1(z_t,\pi^n_t)+\frac D2\normt{H(\cdot|z_t,u)-H(\cdot|\pi^n_t,u)}\\
&\qquad+\int_\Y W_1\bigl(F(\pi^n_t,u,y),F^{\T_n,Q_n}(\pi^n_t,u,y)\bigr)H(dy|\pi^n_t,u)\\
&\quad\le K_2\,W_1(z_t,\pi^n_t)+\epsilon_n.
\end{align*}
Taking expectations gives the first inequality. At time $0$, $Y_0\sim P^{Q}_{\mu_0}$, and
\begin{align*}
E^{\mu_0}\bigl[W_1(z_0,\pi^n_0)\bigr]&\le\int_\Y W_1\bigl(\Phi^{Q}(\mu_0,y),\Phi^{Q}(\nu_0,y)\bigr)P^{Q}_{\mu_0}(dy)\\
&\qquad+\int_\Y W_1\bigl(\Phi^{Q}(\nu_0,y),\Phi^{Q_n}(\nu_0,y)\bigr)P^{Q}_{\mu_0}(dy)\\
&\le\frac D2\,(2-\delta(Q))\normt{\mu_0-\nu_0}+\epsilon_{0,n}+\frac D2\normt{P^{Q}_{\mu_0}-P^{Q}_{\nu_0}}\\
&\le\frac D2\,(3-2\delta(Q))\normt{\mu_0-\nu_0}+\epsilon_{0,n},
\end{align*}
by Lemma \ref{lem:bayesTV}; for Lemma \ref{lem:filter_stability} the second integral is absent, which gives the constant $2-\delta(Q)$. Finally, with $e_t:=E^{\mu_0}[W_1(z_t,\pi^n_t)]$, iterating gives $e_t\le K_2^te_0+\epsilon_n\sum_{s=0}^{t-1}K_2^s$, hence
\[
\sum_{t\ge0}\beta^te_t\le\frac{e_0}{1-\beta K_2}+\epsilon_n\sum_{s\ge0}K_2^s\sum_{t>s}\beta^t=\frac{e_0}{1-\beta K_2}+\frac{\beta\,\epsilon_n}{(1-\beta)(1-\beta K_2)},
\]
which is (\ref{eq:filter_mismatch_sum}), and, if $K_2<1$, $e_t\le K_2^te_0+\frac{\epsilon_n}{1-K_2}$, which gives (\ref{eq:filter_mismatch_avg}).
\hfill$\square$

\bibliographystyle{plainnat}
\bibliography{DKY}

\end{document}